\documentclass[11pt]{article}
\usepackage{overpic}

\usepackage{amsmath,amsfonts,amssymb,amsthm,graphics,epsfig,cite,color,indentfirst}
\usepackage{graphics}
\usepackage{subcaption}
\usepackage{mathtools}
\usepackage{commath}
\usepackage{mathrsfs}
\usepackage{comment}
\usepackage[sc,osf]{mathpazo}
\newtheorem{theorem}{Theorem}[section]
\newtheorem{assumption}{Assumption}[section]
\newtheorem{corollary}{Corollary}[section]
\newtheorem{definition}{Definition}[section]
\newtheorem{lemma}{Lemma}[section]
\newtheorem{claim}{Claim}[section]
\newtheorem{proposition}{Proposition}[section]
\newtheorem{remark}{Remark}[section]
\usepackage[sc,osf]{mathpazo}
\usepackage[colorlinks=true]{hyperref} 
\theoremstyle{plain}

\newcommand{\al}{\alpha}
\newcommand{\be}{\beta}
\newcommand{\La}{\Lambda}
\newcommand{\la}{\lambda}
\newcommand{\ga}{\gamma}
\newcommand{\Ga}{\Gamma}
\newcommand{\ep}{\epsilon}
\newcommand{\sig}{\sigma}
\newcommand{\de}{\delta}

\newcommand{\R}{\mathbb{R}}

\newcommand{\RR}{{\rm I\kern -1.6pt{\rm R}}}

\numberwithin{equation}{section}
\numberwithin{figure}{section}
\allowdisplaybreaks

\begin{document}
	\title{\bf The principal eigenvalue of an age-structured operator with diffusion and advection: qualitative analysis and an application
		\thanks{H. Kang was supported by NSF of China (No. 12301259, 12371169), R. Peng was supported by NSF of China (No. 12271486, 12171176) and Natural Science Foundation of Zhejiang Province, and M. Zhou was supported by the Nankai Zhide Foundation and NSF of China (No. 11971498).}}
	
	\author{{\sc Hao Kang$^{a}$, Rui Peng$^{b}$ and Maolin Zhou$^{c}$}\\[2mm]
		{\small  $^{a}$Center for Applied Mathematics and KL-AAGDM, Tianjin University, Tianjin 300072, China} \\
		{\small Email: haokang@tju.edu.cn}\\
		{\small  $^{b}$School of Mathematical Sciences, Zhejiang Normal University, Jinhua, Zhejiang, 321004, China} \\
		{\small Email: pengrui\_seu@163.com}\\
		{\small  $^{c}$Chern Institute of Mathematics, Nankai University, Tianjin 300071, China} \\
		{\small Email: zhouml123@nankai.edu.cn}
	}
	
	\maketitle
	
	\begin{abstract}
		In this paper, we investigate an eigenvalue problem associated with an age-structured operator incorporating random diffusion and advection. Our primary focus is on examining the asymptotic behaviors of the principal eigenvalue with respect to large advection and small or large diffusion rates. We subsequently apply these results to a nonlinear age-structured model, providing a better understanding of how diffusion and advection influence the spatial distribution of species. Among other ingredients, our approach involves constructing various types of super- and sub-solutions to tackle the novel challenges posed by the nonlocal terms in the problems under consideration.
		
		\vspace{0.3cm} \noindent {\em Key words:} Age-structured operator; Principal eigenvalue; Diffusion and advection; Global dynamics; Positive equilibrium; Asymptotic profile.
		
		\vskip 0.2cm
		
		\noindent {\em AMS subject classifications:}~  35K57, 47A10, 35P15, 35Q92, 92D25
	\end{abstract}
	
	{
		\hypersetup{linkcolor=black}
		\tableofcontents
	}
	
	\section{Introduction}
	In modeling the population dynamics of biological species and the transmission dynamics of infectious diseases, it is crucial to consider directional movement, often described by the advection term. For instance, species migrate to areas with more resources, plant seeds disperse via wind, and fish travel along river streams. Over the past decades, various reaction-diffusion-advection models from physics and biology have been developed and studied in the literature (see \cite{berestycki2005elliptic, chen2008principal, chen2012effects, cui2016spatial, cui2017dynamics, cui2021concentration,KMP2017, liu2019monotonicity, liu2021asymptotics, liu2021asymptoticsII, lou2015evolution, lou2016qualitative, lou2019global, peng2015effects, peng2018effects, peng2019asymptotic, zhao2016lotka, zhou2016lotka, zhou2018global}). Among these works, the investigation of diffusion and advection effects on the principal eigenvalue for both elliptic and parabolic operators has become increasingly important and attracted considerable attention. These studies are essential for understanding threshold dynamics and propagation phenomena in nonlinear diffusive models, as highlighted in \cite{berestycki2005elliptic, chen2008principal, chen2012effects, lam2016asymptotic, peng2015effects, peng2018effects, peng2019asymptotic, liu2021asymptotics, liu2021asymptoticsII} and the references therein.
	
	In describing certain ecological problems, where not only the dispersion of individuals in the environment is considered, but also the birth and death of these individuals are modeled; one may refer to, for instance, Fife \cite{fife2017integrodifferential}, Hutson et al. \cite{hutson2003evolution}, Medlock and Kot \cite{medlock2003spreading}, and Murray \cite{murray2007mathematical}. To incorporate birth and death processes, in the present work we will be concerned with the eigenvalue problem of an age-structured operator with random diffusion and advection. Such an operator can collect information of random dispersion (diffusion), directional movement (advection) in the environment and birth-death process (age-structure) of the population simultaneously.
	
	More precisely, in this paper we are interested in the following linear eigenvalue problem:
	\begin{equation}\label{linear}
		\begin{cases}
			u_a=du_{xx}+\Lambda q(x)u_x-\mu(a, x)u-\lambda u,&\quad (a, x)\in(0, a_m)\times(0, 1),\\
			u_x(a, 0)=0, \; u_x(a, 1)=0, &\quad a\in(0, a_m),\\
			u(0, x)=\int_{0}^{a_m}\beta(a, x)u(a, x)da, &\quad x\in(0, 1).
		\end{cases}
	\end{equation}
	The parameter \( a_m \le \infty \) denotes the maximum age, while the spatial domain is simply considered to be \( (0, 1) \), a bounded interval in \(\mathbb{R}\). The function $q: [0, 1]\to\R_+\setminus\{0\}$ is continuous, representing the directional flow, and the constants \( d>0 \) and \( \Lambda\in\R\) denote the random diffusion rate and directional advection rate, respectively.  The functions \(\beta(a,x)\) and \(\mu(a,x)\) stand for the birth and death rates of the population at age \( a \) and location \( x \), respectively. 
	
	From now on, denote
	\begin{eqnarray}
		&&\underline{\mu}(a):=\min_{x\in[0, 1]}\mu(a, x),\;\ \overline{\mu}(a):=\max_{x\in[0, 1]}\mu(a, x),\nonumber\\
		&&\underline{\beta}(a):=\min_{x\in[0, 1]}\beta(a, x),\;\ \overline{\beta}(a):=\max_{x\in[0, 1]}\beta(a, x).\nonumber
	\end{eqnarray} 
	In light of their realistic implications in age-structured models, we shall impose the following assumptions on the birth rate $\be$ and death rate $\mu$; one may refer to \cite{Ducrot2022Age-structuredI,Ducrot2022Age-structuredII}.
	\begin{assumption}\label{Ass}
		{\rm The birth rate $\beta=\beta(a, x)$ and the death rate $\mu=\mu(a, x)$ satisfy the following conditions:
			\begin{itemize}
				\item [(i)] For any $x\in[0, 1]$, the function $a\mapsto\beta(a, x)\in L^\infty_+(0, a_m)$ and the mapping $x\mapsto\beta(\cdot, x)$ is of $C^2$ from $[0, 1]$ to $L^\infty(0, a_m)$. That is, $\beta\in C^2([0, 1], L_+^\infty(0, a_m))$;
				
				\item [(ii)]  There exists $a_c\in(0, a_m)$ such that $\beta\equiv0$ on $[a_c, a_m)\times[0, 1]$ and $\int_a^{a_c}\underline\beta(l)dl>0$ for any $a\in[0, a_c)$;
				
				\item [(iii)] For any $x\in[0, 1]$, the function $a\mapsto\mu(a, x)\in L^\infty_{\rm{loc}, +}(0, a_m)$ and the mapping $x\mapsto\mu(\cdot, x)$ is of $C^0$ from $[0, 1]$ to $L_{{\rm loc}}^\infty(0, a_m)$. That is, $\mu\in C^0([0, 1], L_{{\rm loc}, +}^\infty(0, a_m)$.
		\end{itemize}}
	\end{assumption} 
	
	Regarding Assumption \ref{Ass}, we would like to make the following comments.
	
	\begin{remark}\label{RK} {\rm 
			\begin{itemize}
				
				\item [(i)] The \( C^2 \)-regularity of \( \beta \) with respect to \( x \) in Assumption \ref{Ass}-(i) is utilized to ensure the smoothness of the solution to the problems under consideration with respect to \( x \) for any given \( a \in [0, a_m) \).
				
				\item [(ii)] Assumption \ref{Ass}-(ii) is reasonable for applications, which means that the birth rate becomes zero when the age of the individuals approaches the maximum age $a_m$. In particular, Assumption \ref{Ass}-(ii) will allow us to consider equivalently the eigenvalue problem \eqref{linear} on $[0, a_+]\times[0, 1]$ for any $a_+\in[a_c, a_m)$; see Remark \ref{remark-2.1} below. Moreover, in the domain $[0, a_+]\times[0, 1]$, the singularity of $\mu$ at $a_m$ is avoided, since $\mu(a_m, \cdot)$ may blow up when $a_m<\infty$. 
				
				\item [(iii)] Note that in Assumption \ref{Ass}-(iii), $\underline{\mu} \in L^\infty_{{\rm loc}, +}(0, a_m)$ allows $\int_0^{a_m}\underline{\mu}(a)da=+\infty$ to hold when $a_m<\infty$, which means in biological modeling that the population density usually reaches zero at its maximum age. The function $\mu$ can be chosen to blow up at $a_m$ when $a_m$ is finite, and one may refer to Figure 1.2 from \cite{inaba2017age} for such an example.
				Assumption \ref{Ass}-(iii) is an improvement over the previous work \cite{kang2022effects} by the first-named author, where such a situation was not addressed. 
				
			\end{itemize}
			
		}
	\end{remark}

	Throughout this paper, when we refer to a real number \(\lambda_0\) as a principal eigenvalue of the eigenvalue problem \eqref{linear}, we mean that there exists an eigenfunction \(u\) corresponding to \(\lambda_0\) such that \((\lambda_0, u)\) solves \eqref{linear} and \(u \geq 0\), \(u \not\equiv 0\). Under Assumption \eqref{Ass}, we can show the existence and uniqueness of the principal eigenvalue of \eqref{linear}; see Lemma \ref{lem3.1} below. The primary aim of this paper is to explore the asymptotic behaviors of this principal eigenvalue concerning large advection and small or large diffusion rates; see Section 2.1 for the precise statements.
	
	As an application of our results on the principal eigenvalue, we examine a nonlinear age-structured model. This model serves to characterize the temporal-spatial dynamics of biological species, taking into account the significant role of age structure within the population and the nature of random dispersal and directional flow. More precisely, the nonlinear model under consideration takes the following form:
	\begin{equation}\label{nonlinear}
		\begin{cases}
			u_t+u_a=du_{xx}-\La u_x-\mu(a, x)u, &(t, a, x)\in(0, \infty)\times(0, a_m)\times(0, 1),\\
			u(t, 0, x)=f\left(x, \int_0^{a_m}\beta(a, x)u(t, a, x)da\right), &(t, x)\in(0, \infty)\times(0, 1),\\
			du_x(t, a, x)-\La u(t, a, x)=0, &(t, a, x)\in(0, \infty)\times(0, a_m)\times\{0, 1\},\\
			u(0, a, x)=u_0(a, x), &(a, x)\in(0, a_m)\times(0, 1).
		\end{cases}
	\end{equation}
	Here, \( u(t, a, x) \) denotes the density of the population at time $t$ with age \( a \) and at position \( x \) and \( f \) is a monotone-type nonlinearity that describes the growth rate of the population. Specific assumptions regarding \( f \) will be made later. The main feature of \eqref{nonlinear} is the combination of nonlinearity and nonlocality. Considering \eqref{nonlinear} in a river environment, when \( \Lambda>0 \), the points \(x=0\) and \(x=1\) can be interpreted as the upstream boundary and downstream boundary, respectively. This means that \(x=0\) is the starting point of the river segment under consideration where the water flows in, and \(x=1\) is the endpoint where the water flows out.

	To our knowledge, age-structured models with random diffusion were first proposed by Gurtin \cite{gurtin1973system} and have since been extensively studied. Research closely related to our models \eqref{nonlinear} with/without advection includes works, for instance, by Busenberg and Iannelli \cite{busenberg198nonlinear}, Chan and Guo \cite{chan1990semigroups}, Di Blasio \cite{di1979non}, Guo and Chan \cite{guo1994semigroup}, Gurtin and MacCamy \cite{gurtin1981diffusion}, Hastings \cite{hastings1992age}, Huyer \cite{huyer1994semigroup}, Langlais \cite{langlais1985nonlinear,langlais1988large}, MacCamy \cite{maccamy1981population}, and Delgado \cite{delgado2006nonlinear,delgado2008nonlinear}. One may refer to the survey by Webb \cite{webb2008population}, which provides detailed results and references. The aforementioned works primarily focus on the well-posedness and long-term behavior of these models under specific assumptions, such as separable forms or spatially homogeneous parameters. More recent studies by Walker such as \cite{walker2008age,walker2009positive,walker2010global,walker2011bifurcation,walker2011positive,
		walker2013global,walker2022principle,walker2023stability} have further advanced the field; his research has concentrated on the existence of nontrivial equilibria in age-structured models similar to \eqref{nonlinear} and established the principle of linearized stability. In addition, we mention that if our model \eqref{nonlinear} is discretized to only two stages (juvenile-adult), then it is similar to a class of switching models which have two states converting to each other, see for example \cite{monmarche2025impacts,wang2025on}.
	
	Now, we make some assumptions on the nonlinear function $f=f(x, u)$ as follows.	
	\begin{assumption}\label{assump-f}
		{\rm The function $f$ satisfies the following conditions:
			\begin{itemize}
				
				\item [(i)] $f\in C^{2, 1}([0, 1]\times[0, \infty))$;
				
				\item [(ii)] $f_u(x, u)>0$ for all $u\in[0, \infty)$ and $x\in[0, 1]$;
				
				\item [(iii)] $f(x, 0)\equiv0$ and $\frac{f(x, u)}{u}$ is decreasing with respect to $u$ for all $x\in[0, 1]$;
				
				\item [(iv)] There exists a constant $L>0$ such that $f(x, u)\le L$ for all $u\in[0, \infty)$ and $x\in[0, 1]$.
			\end{itemize}
		} 
	\end{assumption}
	
	In age-structured models, nonlinearities often play a crucial role in capturing complex biological interactions. A typical example of such a nonlinearity in Assumption \ref{assump-f} is the function \( f(x, u) = \frac{u}{1+\tau u} \) for \( u \ge 0 \), where \( \tau > 0 \) is a constant. This is commonly known as Holling’s type II response and is frequently used in ecological modeling. Another example is the logistic function \( f(x, u) = 1 - e^{-u} \) for \( u \ge 0 \), which is widely used to model population growth due to its ability to represent self-limiting growth processes.
	Assumption \ref{assump-f}-(ii) are employed to avoid the existence of periodic solutions, which are common in age-structured models. This is supported by the works of Magal and Ruan \cite{magal2009center,magal2018theory} and Liu et al. \cite{liu2011hopf}, among others. Assumption \ref{assump-f}-(iii) introduces a sub-linear condition that ensures the uniqueness of the positive equilibrium, preventing multiple equilibria in the model. Lastly, Assumption \ref{assump-f}-(iv) implies that the growth rate of the population is bounded, which is crucial for guaranteeing that the solution to the model \eqref{nonlinear} remains uniformly bounded over time. These assumptions are essential for maintaining the mathematical tractability and biological realism of the models. Unless otherwise specified, Assumptions \ref{Ass} and \ref{assump-f} apply throughout the entire paper. We will indicate any additional assumptions when needed.

	This paper's primary contribution is establishing the asymptotic behaviors of the principal eigenvalue for the eigenvalue problem \eqref{linear} under large advection and small or large diffusion rates. Additionally, we provide a comprehensive understanding of the global dynamics of \eqref{nonlinear}, highlighting how diffusion and advection affect species' spatial distribution. The results for the principal eigenvalue are crucial in this analysis, particularly in exploring the asymptotic profiles of the positive equilibrium under conditions of large advection, small diffusion, or large diffusion. These findings are novel and intriguing compared to previous work. 
	
	We would like to mention another motivation for studying the asymptotic behavior of the principal eigenvalue of the eigenvalue problem \eqref{linear} under conditions of large advection and small or large diffusion rates. Observe that when \(\beta\) is chosen as \(\beta(a, x) \equiv \delta_{a_m}(a)\), where \(\delta_{a_m}\) is the Dirac delta function, the third condition in \eqref{linear} becomes 
	\[ 
	u(0, x) = u(a_m, x).
	\]
	Thus, \eqref{linear} transforms formally into the eigenvalue problem of a parabolic operator under a time-periodic case, with the period being \(a_m\), provided \(\mu(a, x) = \mu(a + a_m, x)\). For such time-periodic parabolic operators, Liu et al. \cite{LiuL2022,liu2021asymptotics, liu2021asymptoticsII} and the second-named author and Zhao \cite{peng2015effects}) investigated the limiting behavior of the principal eigenvalue under conditions of large advection and small diffusion. However, the new challenges in studying \eqref{linear} arises from the nonlocal growth condition (i.e., the third equation in \eqref{linear}). The difficulty lies in incorporating the effects of diffusion and advection terms from the first equation of \eqref{linear} into the nonlocal initial condition, which is the focus of this paper. Without the nonlocal condition, one typically perturbs the problem around the extreme points of \(\mu\) with respect to \(x\) and constructs suitable super-/sub-solutions. The error induced by small perturbations on \(\mu\) can then be compensated by the term \(\lambda\). However, in our problem, the new error terms introduced by perturbations on \(\beta\) cannot be directly absorbed into \(\lambda\). To address this issue posed by the nonlocal term, we construct novel generalized super-/sub-solutions that effectively capture the nonlocal effects. This approach marks a significant difference and innovation compared to previous works that did not involve nonlocal growth conditions. 
	
	An important tool for studying the local and global dynamics of \eqref{nonlinear}, particularly concerning species persistence and extinction, is the examination of the spectrum of the linearized operator at the trivial zero equilibrium. By analyzing the spectrum—specifically, the sign of the spectral bound or the principal eigenvalue, if it exists—we can gain insights into the long-term behavior of \eqref{nonlinear}. Therefore, we will apply the established qualitative results of \eqref{linear} to \eqref{nonlinear} to investigate its global dynamics. To explore the asymptotic profiles of the positive equilibrium of \eqref{nonlinear} with large advection, small diffusion, or large diffusion, we have to further develop techniques to address the challenges caused by the combination of nonlocal and nonlinear terms.
	
	Furthermore, we believe our results on the eigenvalue problem \eqref{linear} have potential applications in other nonlinear modeling problems. On the other hand, although this paper focuses on Neumann boundary conditions, our approach should also be applicable to other boundary conditions, such as Dirichlet and Robin, and is expected to yield similar results for both \eqref{linear} and \eqref{nonlinear}.
	
	In this work, we focus on a one-dimensional domain with an advection function of fixed sign (or non-degeneracy), meaning \(q > 0\) or \(q < 0\) in \eqref{linear}. This assumption is biologically reasonable, as the direction of river streams is consistently one-sided, flowing from upstream to downstream. Mathematically, a non-degenerate advection function \(q\) simplifies the task of deriving the limiting properties of the principal eigenvalue concerning diffusion and advection rates. For readers interested in more general, degenerate cases where \(q\) can change sign or even vanish in certain intervals, we refer to \cite{liu2021asymptotics, liu2021asymptoticsII}. However, given the focus on realistic biological models, we will not address these degenerate cases here.
	
	The rest of the paper is structured as follows. In Section \ref{MR}, we present our main results. Section \ref{Pre} discusses the existence of the principal eigenvalue for \eqref{linear}. We then define generalized super-/sub-solutions and provide an equivalent characterization of the sign of the principal eigenvalue, which is used to explore its limiting properties. In Section \ref{EA}, we examine the effects of diffusion and advection rates on the principal eigenvalue, establishing its asymptotic behaviors with respect to large advection, small or large diffusion. In Section \ref{application}, we apply these results to \eqref{nonlinear}, to investigate its global dynamics and study the asymptotic profiles of its positive equilibrium concerning diffusion and advection. In the appendix, we include two figures to illustrate the profiles of the generalized super-solution \(\overline{w}\) for large advection in the proof of Theorem \ref{Lalambda} and the generalized sub-solution \(\underline{w}\) for small diffusion in the proof of Theorem \ref{Dlambda}.

	\section{Main results}\label{MR}
	In this section, we will present the main findings of this paper, which include the limiting behaviors of the principal eigenvalue of \eqref{linear} and the global dynamics of \eqref{nonlinear} as well as the asymptotic profiles of the positive equilibrium, if it exists.
	
	\subsection{Results on the principal eigenvalue of \eqref{linear}}
	In this subsection, we focus primarily on the limiting behaviors of the principal eigenvalue \(\lambda_0\) of  \eqref{linear} as the rate \(d\) approaches zero or infinity, or as the advection rate \(\Lambda\) approaches infinity. Throughout this section, we will denote the principal eigenvalue as \(\lambda_0(d, \Lambda)\) to emphasize its dependence on \(d\) and \(\Lambda\). 
	\begin{theorem}[Large advection]\label{Lalambda} 
		Given $d>0$, the function $\La\mapsto\lambda_0(d, \La)$ is continuous on $\R$ and satisfies 
		\begin{equation}\label{la-1}
			\lambda_0(d, \La)\rightarrow\begin{cases}
				\alpha_1 &\text{ as } \La\rightarrow\infty,\\
				\alpha_0 &\text{ as } \La\rightarrow-\infty,
			\end{cases}
		\end{equation}
		where $\al_1$ and $\al_0$ are, respectively, uniquely determined by
		\begin{eqnarray}
			&&\int_0^{a_c}\beta(a, 1)e^{-\al_1a}e^{-\int_0^a\mu(s, 1)ds}da=1, \label{al_1}\\
			&&\int_0^{a_c}\beta(a, 0)e^{-\al_0a}e^{-\int_0^a\mu(s, 0)ds}da=1. \label{al_0}
		\end{eqnarray}
	\end{theorem}
	
	\begin{theorem}[Small diffusion with advection]\label{Dlambda} 
		Given $\La\in\R\setminus\{0\}$, the function $d\mapsto\lambda_0(d, \La)$ is continuous on $(0, \infty)$ and satisfies 
		\begin{equation}\label{small-d}
			\lambda_0(d, \La)\rightarrow\begin{cases}
				\alpha_1 &\text{ as } d\rightarrow 0, \ \,\text{ if $\La>0$, }\\
				\alpha_0 &\text{ as } d\rightarrow0,\ \, \text{ if $\La<0$},
			\end{cases}
		\end{equation}
		where $\al_1, \al_0$ are defined in \eqref{al_1} and \eqref{al_0} respectively.
	\end{theorem}
	
	\begin{theorem}[Small diffusion without advection]\label{Dlambda La=0} 
		Given $\La=0$, the function $d\mapsto\lambda_0(d, 0)$ is continuous on $(0, \infty)$ and satisfies 
		\begin{equation}\nonumber
			\lambda_0(d, 0)\rightarrow \alpha_{\max} \ \quad \text{ as } d\rightarrow0,
		\end{equation}
		where $\al_{\max}$ is uniquely determined by
		\begin{eqnarray}\label{almax}
			\max_{x\in[0, 1]}\int_0^{a_c}\beta(a, x)e^{-\al_{\max} a}e^{-\int_0^a\mu(s, x)ds}da=1. \label{al_max}
		\end{eqnarray}
	\end{theorem}
	
	\begin{theorem}[Large diffusion]\label{average} 
		Given $\La\in\R$, we have
		\begin{equation}\nonumber
			\lambda_0(d, \La)=\bar\al\ \quad \text{ as }d\to\infty,
		\end{equation}
		where $\bar\al$ is uniquely determined by
		\begin{equation}\label{baral}
			\int_0^1\int_0^{a_c}\beta(a, x)e^{-\bar\al a}e^{-\int_0^a\int_0^1\mu(s, x)dxds}dadx=1.
		\end{equation}
	\end{theorem}
	
	We would like to mention that it is an interesting issue to explore the asymptotical behavior of $\lambda_0(d, \La)$ when both $d$ and $\La$ vary (see, e.g. \cite{chen2012effects, LiuL2024}). Indeed, one can easily improve the above theorem by showing that $\lambda_0(d, \La)=\bar\al$ as $d\to\infty$ and $\La^2/d\to0$; the other cases seem more involved and are left for future study.
	
	\subsection{Results on the nonlinear model \eqref{nonlinear}}\label{1.1}
	In this subsection, we focus on \eqref{nonlinear}. Let us first write down the problem that the equilibrium corresponding to \eqref{nonlinear} solves:
	\begin{equation}\label{logistic}
		\begin{cases}
			u_a=du_{xx}-\La u_x-\mu(a, x)u, &(a, x)\in(0, a_m)\times(0, 1),\\
			u(0, x)=f\left(x, \int_{0}^{a_m}\beta(a, x)u(a, x)da\right), &x\in(0, 1),\\
			du_x(a, x)-\La u(a, x)=0, &(a, x)\in(0, a_m)\times\{0, 1\}.
		\end{cases}
	\end{equation}
	The eigenvalue problem associated with \eqref{logistic}, which is obtained by linearizing \eqref{logistic} at zero, reads as
	\begin{equation}\label{eigenvalue}
		\begin{cases}
			u_a=du_{xx}-\La u_x-\mu(a, x)u-\la u, &(a, x)\in(0, a_m)\times(0, 1),\\
			u(0, x)=f_u(x, 0)\int_0^{a_m}\beta(a, x)u(a, x)da, &x\in(0, 1),\\
			du_x(a, x)-\La u(a, x)=0, &(a, x)\in(0, a_m)\times\{0, 1\}.
		\end{cases}
	\end{equation}
	In order to apply the results of the principal spectral theory in \eqref{linear} directly to problem \eqref{nonlinear}, we set
	\begin{equation}\nonumber
		v(a, x)=e^{-(\La/d)x}u(a, x),
	\end{equation}
	and thus \eqref{eigenvalue} is rewritten as
	\begin{equation}\label{eigenvalue-prob}
		\begin{cases}
			v_a=dv_{xx}+\La v_x-\mu(a, x)v-\la v, &(a, x)\in(0, a_m)\times(0, 1),\\
			v(0, x)=f_u(x, 0)\int_0^{a_m}\beta(a, x)v(a, x)da, &x\in(0, 1),\\
			v_x(a, x)=0, &(a, x)\in(0, a_m)\times\{0, 1\}.
		\end{cases}
	\end{equation}
	We still use $\la_0=\la_0(d, \La)$ to denote the principal eigenvalue of \eqref{eigenvalue-prob} and the corresponding principal eigenfunction is denoted by $v\in W^{1, 1}((0, a_m), W^{2, p}(0, 1))$ with $p>1$ (see Section \ref{Pre}). 
	
	Our first result concerns the global dynamics of \eqref{nonlinear}, and can be stated as follows.	
	
	\begin{theorem}\label{stability MR} If $\lambda_0>0$, then there exists a unique positive equilibrium $u^*$ which satisfies \eqref{logistic} and is globally stable in the sense that $u(t, a, x)\rightarrow u^*(a, x)$ in $C([0, a_+]\times[0, 1])$ for any $a_+\in[a_c, a_m)$ as $t\rightarrow\infty$, where $u(t, a, x)$ is a solution of \eqref{nonlinear} with initial data $u_0(a, x)\geq0$ and $u_0(a, x)\not\equiv0$ in $[0, a_m)\times[0, 1]$. Otherwise, if $\la_0\le0$, then zero is globally stable  in the sense that $u(t, a, x)\rightarrow0$ in $C([0, a_+]\times[0, 1])$ for any $a_+\in[a_c, a_m)$ as $t\rightarrow\infty$.
	\end{theorem} 
	
	For later use, we make the following hypotheses:
	\begin{eqnarray}
		f_u(1, 0)\int_{0}^{a_c}\beta(a, 1)e^{-\int_0^a\mu(s, 1)ds}da>1,\label{alpha}\\
		\max_{x\in[0, 1]}f_u(x, 0)\int_0^{a_c}\beta(a, x)e^{-\int_0^a\mu(s, x)ds}da>1,\label{al max}\\
		\int_0^1\int_0^{a_c}f_u(x, 0)\beta(a, x)e^{-\int_0^a\int_0^1\mu(s, x)dxds}dadx>1.\label{al average}
	\end{eqnarray}
	Based on the results in Subsection 2.1, we can give some explicit conditions on the global dynamics of \eqref{nonlinear} when the diffusion rate $d$ is large or small, or advection rate $\La$ is large. 
	\begin{proposition}\label{asymptotic}{\it  If one of the following conditions holds:
			\begin{itemize}
				\item [\rm(i)] For fixed $d>0$, $\La>0$ is sufficiently large and \eqref{alpha} is fulfilled;
				
				\item [\rm(ii)] For fixed $\La>0$, $d>0$ is sufficiently small and \eqref{alpha}  is fulfilled;
				
				\item [\rm(iii)] For $\La=0$, $d>0$ is sufficiently small and \eqref{al max}  is fulfilled; 
				
				\item [\rm(iv)] For fixed $\La\in\R$,  $d>0$ is sufficiently large and \eqref{al average}  is fulfilled,
			\end{itemize}
			then $\la_0>0$ and thus $u^*$ is globally stable for  \eqref{nonlinear}.
			
			Similarly, by reversing \eqref{alpha},  \eqref{al max} and \eqref{al average}, we can obtain similar results which ensure that zero is globally stable for \eqref{nonlinear}.
		}
	\end{proposition}  
	
	This proposition suggests that the dynamics of \eqref{nonlinear} are primarily determined by the behavior of growth, birth, and death rates under different conditions:
	(1) With Advection: for large advection or small diffusion, the dynamics are mainly influenced by these rates at the downstream boundary.
	(2) Without Advection: when the diffusion rate is sufficiently small, the dynamics are primarily determined by these rates at the point of maximum value across the entire habitat.
	(3) Regardless of Advection: when the diffusion rate is large enough, the dynamics are governed by the average behavior of the growth, birth, and death rates throughout the entire habitat.
	In summary, the impact of growth, birth, and death rates on the system's dynamics varies depending on the presence of advection and the magnitude of the diffusion rate.
	
	In what follows, we aim to explore the asymptotic behavior of the positive equilibrium \( u^* \) in relation to large advection \( \Lambda \) or small diffusion \( d \), with \( \Lambda \ge 0 \). Our findings are divided into two parts: one addressing the scenario with advection (Theorem \ref{AP LA MR}) and the other without advection (Theorem \ref{AP LA WA}). To emphasize the dependence of \( u^* \) on \( d \)  and \( \Lambda \), we will denote it as \( u^*_{\Lambda,\, d} \) .
	
	\begin{theorem}\label{AP LA MR}  Assume that \eqref{alpha} holds. Then for any $a\in[0, a_m)$, the following assertions hold.
		\begin{itemize}
			\item [\rm(i)] For fixed $d>0$, $\int_0^1 u^*_{\Lambda,\, d}(a,x)dx\to0$ as $\La\to\infty$;
			
			\item [\rm(ii)] For fixed $\La>0$, $\int_0^1 u^*_{\Lambda,\, d}(a,x)dx\to0$ as $d\to0$.
		\end{itemize}
	\end{theorem}
	
	This theorem suggests that, regardless of the age stage, the species will become extinct throughout the entire habitat if the advection rate is sufficiently large or if the diffusion rate is sufficiently small in the presence of advection. The key point is that the growth rate $f=f(x, u)$ is bounded in $[0, 1]\times[0, \infty)$, which avoids the occurrence of concentration behavior.
	
	Next, in the absence of advection (\(\Lambda = 0\)), we aim to investigate the asymptotic behavior of the equilibrium \(  u^*_{0,\, d} \) when diffusion is small. We begin by defining a threshold function: 
	\begin{eqnarray}\label{Q}
		\Gamma(x):=f_u(x, 0)\int_0^{a_c}\beta(a, x)e^{-\int_0^a\mu(s, x)ds}da, \quad x\in[0, 1].
	\end{eqnarray}
	
	\begin{theorem}\label{AP LA WA} 
		For any compact subset $I\subset (0, 1)$, if $\Gamma(x)>1$ for all $x\in I$, then as $d\rightarrow0$, we have 
		\begin{eqnarray}\nonumber
			u^*_{0,\, d}(a, x)\to
			v^*(a, x),\ \ \,\forall (a, x)\in[0, a_+]\times I,
		\end{eqnarray} 
		where $v^*(\cdot, x)\in W^{1, 1}(0, a_+)$ for any $x\in I$, is the unique positive solution to the problem:
		\begin{equation}\label{solutionofQ MR}
			\begin{cases}
				v_a(a, x)=-\mu(a, x)v(a, x),\; a\in(0, a_+),\\
				v(0, x)=f\left(x, \int_{0}^{a_+}\beta(a, x)v(a, x)da\right).
			\end{cases}
		\end{equation} 
		Otherwise, if $\Ga(x)\le1$ for all $x\in I$, then as $d\rightarrow0$,
		\begin{eqnarray}\nonumber
			u^*_{0,\, d}(a, x)\to
			0,\ \ \,\forall (a, x)\in[0, a_+]\times I.
		\end{eqnarray} 
	\end{theorem}
	
	Here note that one can enhance the regularity of $f$ and $\mu$ by approximation of $C^2$ functions in $x$ to ensure that the solution of \eqref{solutionofQ MR} belongs to $C^2([0, 1], W^{1, 1}(0, a_+))$, see Ducrot et al. \cite[Section 5]{Ducrot2022Age-structuredII} for more details. Moreover, the proof of Theorem \ref{AP LA WA} is based on an extension of the theory of singular perturbation developed by L{\'o}pez-G{\'o}mez and his coauthors (see \cite{cano2024singular,fernandez2019singular,lopez2013linear}), to overcome the new challenges arising from the nonlocal growth condition. 
	
	Theorem \ref{AP LA MR} and Theorem \ref{AP LA WA} reveal some critical insights into how the positive equilibrium of a species behaves as diffusion becomes small.
	Specifically, in the absence of advection, according to Theorem \ref{AP LA WA}, the positive equilibrium will converge uniformly to a positive function in any compact subset of $(0, 1)$ where $\Gamma(x)>1$, which represents that the species can persist.
	This pattern emphasizes the role of diffusion in maintaining species in specific regions. In stark contrast, when advection is present, as detailed in Theorem \ref{AP LA MR}, the dynamics change significantly. The presence of advection can lead to species extinction for small diffusion rates. This difference underscores how advection can disrupt the spatial distribution that diffusion alone would maintain, often leading to a more uniform or even depleted distribution across the domain.
	
	Finally, we examine the asymptotic profile of \(  u^*_{\Lambda,\, d} \) as \( d \to \infty \) and present the following result.
	\begin{theorem}\label{d to inf} Assume that \eqref{al average} holds. Then for fixed $\La\in\R$ and any $a_+\in[a_c, a_m)$, as $d\to\infty$, we have
		\begin{equation}\nonumber
			u^*_{\Lambda,\, d}\to u_*\ \ \text{uniformly on }[0, a_+]\times[0, 1],
		\end{equation}
		where $u_*=u_*(a)$ is the unique positive solution to the following problem:
		\begin{equation}\label{h}
			\begin{cases}
				u_*'(a)=-\int_0^1\mu(a, x)dx \; u_*(a), \; a\in(0, a_+),\\
				u_*(0)=\int_0^1f\left(x, \int_0^{a_+}\beta(a, x)u_*(a)da\right)dx.
			\end{cases}
		\end{equation}
	\end{theorem}
	
	We observe that the asymptotic profiles of the unique positive equilibrium are identical for both positive advection and zero advection scenarios. Furthermore, Assumption \eqref{al average} suggests that the environment supports the species' survival in terms of the spatial average. Additionally, this theorem indicates that under large diffusion, the asymptotic profile becomes spatially homogeneous and depends solely on the age.
	
	\section{Preliminaries}\label{Pre}   
	In this section, we first establish the existence of the principal eigenvalue and provide an equivalent characterization of \eqref{linear}. We then introduce generalized super- and sub-solutions to facilitate the study of the limiting behaviors of principal eigenvalues concerning diffusion and advection rates in subsequent analysis.
	
	\subsection{The existence of principal eigenvalue}\label{PRE}
	In this subsection, we investigate the existence of the principal eigenvalue of problem \eqref{linear} and provide an equivalent characterization of \eqref{linear}.
	
	\begin{lemma}\label{lem3.1}
		\label{Walker}
		There exists a unique principal eigenvalue of \eqref{linear}, which is algebraically simple and the corresponding eigenfunction is positive a.e. in $(0, a_m)\times(0, 1)$.
	\end{lemma}
	
	\begin{proof}
		The proof is similar to the approach used in {\cite[Theorem 3]{guo1994semigroup}} and {\cite[Lemma 2.6]{walker2013some}}. For the sake of completeness, we provide a sketch of the proof here. Denote by $X$ the Banach space 
		$$
		X=L^p(0, 1), \,\, p>1 
		$$ 
		and denote its positive cone by $X_+$. Observe that $X_+$ is a normal and generating cone. Define the following function spaces:
		\begin{equation*}
			\mathcal{X}=X\times L^1((0, a_m), X), \quad\,\, \mathcal{X}_0=\{0_X\}\times L^1((0, a_m), X)
		\end{equation*}
		endowed with the product norms and the positive cones:
		\begin{equation*}
			\mathcal{X}^+=X_+\times\{u\in L^1((0, a_m), X): u(a, \cdot)\in X_+,\; \text{a.e. in }(0, a_m)\},\quad \mathcal{X}_0^+=\mathcal{X}^+\cap\mathcal{X}_0.
		\end{equation*}
		
		We consider the following problem posed in $X$ for $0\leq\tau\leq a< a_m$ with Neumann boundary conditions:
		\begin{equation}\label{B_1}
			\begin{cases}
				\begin{array}{ll}
					\smallskip
					v_a(a)=dv_{xx}(a)+\La qv_x(a)-\mu(a, \cdot)v(a), \quad \tau<a<a_m,\\
					v(\tau)=\eta\in X.
				\end{array}
			\end{cases}
		\end{equation}
		It follows that problem \eqref{B_1} generates an evolution family on $X$, denoted by $\{\mathcal U(a, \tau)\}_{0\le\tau\le a<a_m}$. Actually, such $\mathcal U$ can be given by a Green's function $G$ associated with Neumann boundary condition:
		\begin{equation}\nonumber
			(\mathcal U(a, \tau)\eta)(x)=\int_0^1 G(a, \tau; x-y)\eta(y)dy, \quad \forall \,0\le \tau\le a <a_m.
		\end{equation}
		Moreover, there exist $M>0$ and $\omega\in\R$ such that 
		\begin{eqnarray}\nonumber
			\norm{\mathcal U(a, \tau)}_{\mathcal L (X)}\le M e^{\omega(a-\tau)}, \,\,\;\forall \,0\le \tau\le a <a_m.
		\end{eqnarray} 
		
		We also define the following family of bounded linear operators $\{W_\lambda\}_{\lambda>\omega}\subset \mathcal L\left(\mathcal{X},\mathcal{X}_0\right)$ for $(\eta,g)\in \mathcal X$ by 
		\begin{eqnarray}\nonumber
			W_\lambda(\eta, g)=\left(0, \,\,\, e^{-\lambda a}\mathcal U(a, 0)\eta+\int_0^a e^{-\lambda(a-\tau)}\mathcal U(a, \tau)g(\tau)d\tau\right).
		\end{eqnarray}
		Following the argument in Thieme \cite[Section 6]{thieme2009spectral}, we can prove that this provides a family of positive pseudoresolvents. Thus by Pazy \cite[Section 1.9]{pazy2012semigroups}, there exists a unique closed Hille-Yosida operator $B: dom(B)\subset \mathcal{X}\mapsto\mathcal{X}$ with $\overline{dom(B)}=\mathcal X_0$ such that 
		\begin{equation}\nonumber
			\left(\lambda I-B\right)^{-1}=W_\lambda \quad \text{ for all } \lambda>\omega,
		\end{equation}
		where  $I: \mathcal{X}\mapsto\mathcal{X}$ denotes the identity operator. 
		On the other hand, we define $C\in\mathcal L(\mathcal X_0,\mathcal X)$ by
		\begin{equation}\nonumber
			C(0, h)=\left(\int_{0}^{a_m}\beta(a, \cdot)h(a)da, \;  0\right),\; (0, h)\in\mathcal{X}_0,
		\end{equation}
		and  $A: dom(A)\subset \mathcal{X}\mapsto \mathcal{X}$ by
		\begin{equation}\nonumber
			dom(A)=dom(B)\subset \mathcal X_0,\quad
			A=B+C.
		\end{equation}
		It is easily seen that the existence of eigenvalue of \eqref{linear} is equivalent to that of  eigenvalue of the operator $A$. Now, for each $\lambda\in\R$, we  define a linear operator $\mathcal{M}_{\lambda}: X\mapsto X$ by
		\begin{equation}\label{M}
			\mathcal{M}_{\lambda}\phi=\int_{0}^{a_m}\beta(a, \cdot)e^{-\lambda a}\mathcal{U}(a, 0)\phi \,da, \quad \forall \phi\in X.
		\end{equation}
		In fact, $\mathcal{M}_{\lambda}$ is obtained by plugging the resolvent of $B$ into the integral initial condition \eqref{eigenvalue-prob}, and we refer to \cite{guo1994semigroup,walker2013some} for the derivation. Moreover, by Assumption \ref{Ass}-(ii), $\mathcal{M}_{\lambda}$ can be shown to be a compact and nonsupporting operator in $X$ \cite{guo1994semigroup}, where \textit{nonsupporting} is a generalization of strong positivity when working on a Banach space with a positive cone which has empty interior, for example the $L^p$ space (see \cite{marek1970frobenius} or \cite{sawashima1964spectral} for a complete definition). Thus by the Krein-Rutman theorem, for each $\lambda\in\R$, the spectral radius $r(\mathcal M_{\lambda})$ of operator $\mathcal M_{\la}$ is the principal eigenvalue, which is algebraically simple and the corresponding eigenfunction can be positive. 
		
		Next, note from \eqref{M} that $\lambda\mapsto r(\mathcal M_{\lambda})$ is continuous and strictly decreasing. It follows that such $\lambda$ satisfying $r(\mathcal M_{\lambda})=1$ indeed exists and is unique, denoted by $\la_0$. 
		By definition, such $\la_0$ is an eigenvalue of operator $A$. Moreover, if $\lambda'>\la_0$, one has $r(\mathcal M_{\lambda'})<r(\mathcal M_{\la_0})=1$, which implies that $(I-\mathcal M_{\lambda'})^{-1}$ exists, and so does $(\lambda' I-A)^{-1}$. This prevents $\lambda'$ to be an eigenvalue of $A$ and therefore $\la_0$ is the principal eigenvalue of $A$. Furthermore, the algebraic simplicity follows from the one of $r(\mathcal M_{\la_0})$. In addition, by the assumptions of $\be, \mu, f$ and classical parabolic estimates, the principal eigenfunction of $A$ associated with $\la_0$ belongings to $W^{1, 1}((0, a_m), W^{2, p}(0, 1))$ for any $p>1$. Thus the proof is complete.
	\end{proof}
	
	\begin{remark}\label{remark-2.1}
		{\rm Due to $\beta\equiv 0$ on $[a_c,a_m)\times[0, 1]$ as in  Assumption {\rm \ref{Ass}-(ii)}, the characteristic equation \eqref{M} can be rewritten as follows: 
			$$
			\mathcal{M}_{\la}\phi=\int_{0}^{a_c}\beta(a, \cdot)e^{-\la a}\mathcal{U}(a, 0)\phi\, da,\quad \forall\phi\in X. 
			$$
			We observe from the proof of Lemma \ref{Walker} that the principal eigenvalue of \eqref{linear} is the unique value such that $r(\mathcal M_{\la_0})=1$. Hence, $\lambda_0$ is also the principal eigenvalue of \eqref{linear} with $[0, a_m)$ replaced by $[0, a_+]$ for any $a_+\in[a_c, a_m)$. Moreover, the value of $\la_0$ is independent on the choice of $a_+\in[a_c, a_m)$. Thus, for the remainder of the paper, we only need to focus on the eigenvalue problem \eqref{linear} posed on \([0, a_c]\times\mathbb{R}\) instead of on \([0, a_m)\times\mathbb{R}\). }
	\end{remark}

	\subsection{Generalized super/sub-solutions}
	In this subsection, we introduce the concept of generalized super/sub-solutions, inspired by the work of Liu et al. \cite{liu2021asymptotics, liu2021asymptoticsII}. Keeping Remark \ref{remark-2.1} in mind,
	we first consider the following linear problem:
	\begin{equation}\label{equ}
		\begin{cases}
			w_a= dw_{xx}+\La q(x)w_x-\mu(a, x)w,& (a, x)\in(0, a_c)\times(0, 1),\\
			w_x(a, 0)=0, \; w_x(a, 1)=0, & a\in[0, a_c],\\
			w(0, x)=\int_{0}^{a_c}\beta(a, x)w(a, x)da,& x\in(0, 1).
		\end{cases}
	\end{equation}
	We now give the definition of  generalized super-/sub-solutions associated with \eqref{equ}.
	\begin{definition}\label{super/sub} The function $\overline w$ in $[0, a_c]\times[0, 1]$ is called a  generalized super-solution of \eqref{equ} if there exists a set $\mathbb X$ consisting of at most finitely many points:
		\begin{eqnarray*}
			\mathbb{X}=\emptyset\ \text{ or }\ \mathbb{X}=\{\kappa_i: \kappa_i\in(0, 1),\; i=1, \dots, N\},
		\end{eqnarray*}
		for some integer $N\ge1$, such that
		\begin{itemize}
			\item [(1)] $\overline w\in C([0, a_c]\times(0, 1))\cap 
			W^{1, 1}((0, a_c), C^2((0, 1)\setminus\mathbb X))$;
			\item [(2)] $\overline w_x(a, x^+)<\overline w_x(a, x^-)$, for every $x\in\mathbb X$  and $a\in[0, a_c]$;
			\item [(3)] $\overline w$ satisfies
			\begin{eqnarray}\nonumber
				\begin{cases}
					\overline{w}_a\geq d\overline{w}_{xx}+\La q(x)\overline w_x-\mu(a, x)\overline{w},& (a, x)\in(0, a_c)\times((0, 1)\setminus\mathbb X),\\
					\overline w_x(a, 0)\le0, \;\overline w_x(a, 1)\ge0, & a\in[0, a_c],\\
					\overline{w}(0, x)\ge\int_{0}^{a_c}\beta(a, x)\overline{w}(a, x)da,& x\in(0, 1).
				\end{cases}
			\end{eqnarray}
		\end{itemize}
		A  generalized super-solution $\overline w$ is called a strict  generalized super-solution if it is not a solution of \eqref{equ}. Moreover, a function $\underline w$ is called a (strict)  generalized sub-solution of \eqref{equ} if $-\underline w$ is a (strict)  generalized super-solution.
	\end{definition}
	
	\begin{lemma}\label{super max principle}
		Let $\overline w\ge0$ be a  generalized super-solution of \eqref{equ} defined in Definition \ref{super/sub}. Then $\overline w>0$ in $[0, a_c]\times[0, 1]$ unless $\overline w\equiv0$.
	\end{lemma}
	
	The proof of Lemma \ref{super max principle} closely follows the approach used by Liu et al. in \cite[Lemma 2.2]{liu2021asymptotics}. Therefore, we omit the detailed proof here.
	
	\begin{definition}\label{SMP}
		{\rm We say that \eqref{equ} admits the \textit{strong maximum principle} if $w$ satisfying Definition \ref{super/sub} implies that $w>0$ in $[0, a_c]\times[0, 1]$ unless $w\equiv0$.}
	\end{definition}
	
	Let \(\lambda_0\) be the principal eigenvalue of \eqref{linear}. According to Remark \ref{remark-2.1}, \(\lambda_0\) is equivalently the principal eigenvalue of the following eigenvalue problem:
	\begin{equation}\label{equ-ei}
		\begin{cases}
			u_a= du_{xx}+\La q(x)u_x-\mu(a, x)u-\lambda u,& (a, x)\in(0, a_c)\times(0, 1),\\
			u_x(a, 0)=0, \; u_x(a, 1)=0, & a\in[0, a_c],\\
			u(0, x)=\int_{0}^{a_c}\beta(a, x)u(a, x)da,& x\in(0, 1).
		\end{cases}
	\end{equation}
	With the above preparations, we are now ready to present an equivalent characterization for the principal eigenvalue, the strong maximum principle, and a positive strict generalized super-solution. We omit the proof here, as it can be found in Kang \cite[Proposition 2.7]{kang2022effects}. 
	\begin{proposition}\label{equivalent}
		The following statements are equivalent:
		\begin{itemize}
			\item [(1)] \eqref{equ} admits the strong maximum principle property;
			\item [(2)] $\lambda_0<0$;
			\item [(3)] \eqref{equ} has a strict and positive  generalized super-solution.
		\end{itemize}
	\end{proposition}
	
	It follows directly from Proposition \ref{equivalent} that \(\lambda_0 < 0\) if and only if \eqref{equ} has a strict generalized super-solution that is positive on \([0, a_c] \times [0, 1]\). Conversely, Proposition \ref{equivalent} indicates that \(\lambda_0 > 0\) implies the existence of a positive strict generalized sub-solution of \eqref{equ}, which can be chosen as a principal eigenfunction corresponding to \(\lambda_0\). Here, we  also provide the following assertion, whose proof is identical to that of Peng and Zhao \cite[Corollary 2.1]{peng2015effects}.
	\begin{corollary}\label{corsub}
		Assume that there exists a strict  generalized sub-solution $\underline w$ of \eqref{equ} with $\underline w\ge0$ in $[0, a_c]\times[0, 1]$, then we have $\la_0\ge0$.
	\end{corollary}

	\section{Limiting behaviors of the principal eigenvalue of \eqref{linear}}\label{EA}
	In this section, we present the proofs of Theorems \ref{Lalambda}, \ref{Dlambda}, \ref{Dlambda La=0} and \ref{average}.
	
	\subsection{Limiting behavior for large advection}
	In this subsection, we give the proof of Theorem \ref{Lalambda}.
	\begin{proof}[Proof of Theorem \ref{Lalambda}] Since $\la_0(d, \La)$ is a simple eigenvalue, the continuity of $\La\mapsto\la_0(d, \La)$ follows from Kato \cite[Section IV. 3.5]{kato2013perturbation} for the classical perturbation theory. Next we verify the assertion \eqref{la-1}.
		
		\textbf{Step 1.} We claim that
		\begin{equation}\label{Nge}
			\limsup_{\La\to\infty}\lambda_0(d, \La)\le\al_1.
		\end{equation}
		For any fixed small $\ep>0$, in light of \eqref{equ-ei}, consider the following auxiliary eigenvalue problem:
		\begin{equation}\label{auxNge-1}
			\begin{cases}
				w_a-dw_{xx}-\La q(x)w_x+\overline\Psi_\ep(a, x)w=-\la w, &(a, x)\in(0, a_c)\times(0, 1),\\
				w_x(a, 0)=w_x(a, 1)=0, &a\in(0, a_c),\\
				w(0, x)=\int_0^{a_c}\beta(a, x)w(a, x)da, &x\in(0, 1),
			\end{cases}
		\end{equation}
		where 
		$$
		\overline\Psi_\ep(a, x)=\mu(a, x)+\al_1+\ep.
		$$
		Denote the principal eigenvalue of \eqref{auxNge-1} by $\la_0(\overline\Psi_\ep)$.
		
		To show \eqref{Nge}, we proceed with the construction of a generalized super-solution and then apply Proposition \ref{equivalent}. That is, we want to construct a strict super-solution $\overline w$ to the problem: 
		\begin{equation}\label{auxNge}
			\begin{cases}
				w_a-dw_{xx}-\La q(x)w_x+\overline\Psi_\ep(a, x)w=0, &(a, x)\in(0, a_c)\times(0, 1),\\
				w_x(a, 0)=w_x(a, 1)=0, &a\in(0, a_c),\\
				w(0, x)=\int_0^{a_c}\beta(a, x)w(a, x)da, &x\in(0, 1).
			\end{cases}
		\end{equation}
		
		In the following discussion, for simplicity, we refer to \(\overline{w}\) as a strict generalized super-solution of \eqref{auxNge}. To the end, we rewrite \(\overline{\Psi}_\epsilon\) as follows:
		\[
		\overline{\Psi}_\epsilon(a, x) = [\mu(a, x) - \mu(a, 1) + \epsilon] + [\mu(a, 1) + \alpha_1] =: \overline{\Phi}_\epsilon(a, x) + \mu(a, 1) + \alpha_1.
		\]
		This formulation motivates us to construct a strict generalized super-solution of variable-separated form. We begin by defining a positive function \(h\) that satisfies 
		\begin{equation}\label{fL}
			\begin{cases}
				h'(a) = -\mu(a, 1)h(a) - \alpha_1 h(a),\\
				h(0) = \int_0^{a_c} \beta(a, 1)h(a) \, da.
			\end{cases}
		\end{equation}
		Note that the existence of $h$ is guaranteed by the definition of $\al_1$ in \eqref{al_1}. In fact, $h$ can be written as 
		$$
		h(a)=e^{-\int_0^a\mu(s, 1)ds}e^{-\al_1a},
		$$
		and $(\al_1, h)$ is the principal eigen-pair of the standard age-structured operator. 
		
		As part of our preparation, we select a positive constant \(\delta < \varepsilon\) and define the constants
		\[
		c_1 = c_1(\varepsilon) = \max_{[0, a_c] \times [0, 1]} \left| \overline{\Phi}_\varepsilon(a, x) \right| + 1, \quad c_2 = c_2(\varepsilon, \delta) = \min_{[0, a_c] \times [1-\delta, 1]} \overline{\Phi}_\varepsilon(a, x).
		\]
		It is important to note that \(c_1\) does not depend on the small \(\delta\), and \(c_2\) is non-increasing with respect to \(\delta \in (0, \varepsilon)\). Additionally, if \(\delta\) is sufficiently small, we have
		\begin{equation}\label{c_2}
			\frac{\varepsilon}{2} < c_2 < 2\varepsilon.
		\end{equation}
		
		Furthermore, for some sufficiently small \(\overline{\delta} > 0\), we define the positive function \(h_{\overline{\delta}}\) by the following linear differential equation problem:
		\begin{equation}\label{fL_de}
			\begin{cases}
				h'_{\overline{\delta}}(a) = -\mu(a, 1)h_{\overline{\delta}}(a) - \alpha_{\overline{\delta}} h_{\overline{\delta}}(a),\\
				h_{\overline{\delta}}(0) = \int_0^{a_c} (\beta(a, 1) + \overline{\delta}) h_{\overline{\delta}}(a) \, da,
			\end{cases}
		\end{equation}
		where the existence of \(h_{\overline{\delta}}\) is ensured by the choice of \(\alpha_{\overline{\delta}}\) which satisfies
		\begin{equation}\nonumber
			\int_0^{a_c} (\beta(a, 1) + \overline{\delta}) e^{-\int_0^a \mu(s, 1) \, ds} e^{-\alpha_{\overline{\delta}} a} \, da = 1.
		\end{equation}
		Notice that we can scale \(h(0)\) and \(h_{\overline{\delta}}(0)\) so that \(h(0) = h_{\overline{\delta}}(0) = 1\). 
		
		We now choose \(\delta\) small enough such that \(\beta(a, x) \leq \beta(a, 1) + \overline{\delta}\) for all \(x \in [1-\delta, 1]\) and the following inequality holds:
		\begin{equation}\label{al_de}
			\alpha_1 - \alpha_{\overline{\delta}} > -\frac{\varepsilon}{2}.
		\end{equation}
		
		To satisfy the nonlocal initial condition on \(a\), we define a function \(\overline{w}\) as follows:
		\begin{equation}\label{large_advec_super_sol}
			\overline{w}(a, x) = 
			\begin{cases}
				h_{\overline{\delta}}(a) \overline{z}_1(x) = e^{-(\alpha_{\overline{\delta}} - \alpha_1)a} h(a) \overline{z}_1(x), & x \in [1-\delta, 1],\\
				\overline{w}_2(a, x), & x \in [\delta, 1-\delta],\\
				e^{-\gamma a} h(a) \overline{z}_3(x), & x \in [0, \delta],
			\end{cases}
		\end{equation}
		where \(\gamma > 0\) is a constant such that
		\begin{equation}\label{choice_of_gamma}
			\int_0^{a_c} \beta(a, x) e^{-\gamma a} h(a) \, da \leq h(0)=1, \quad \forall x \in [0, 1].
		\end{equation}
		Here, \(\overline{z}_1\) and \(\overline{z}_3\) are given by
		\begin{equation}\nonumber
			\overline{z}_1(x) = M_1 + (x - 1)^2, \quad \overline{z}_3(x) = M_2 (1 - M_2 x^2)
		\end{equation}
		with \(M_1\) and \(M_2\) being two positive constants to be determined later. Observe that \(\overline{z}_1\) and \(\overline{z}_3\) are strictly decreasing in \([1-\delta, 1]\) and \([0, \delta]\), respectively. Furthermore, we have
		\[
		\overline{z}_1'(1-\delta) = -2\delta < 0 \quad \text{and} \quad \overline{z}_3'(\delta) = -2M_2^2\delta < 0.
		\]
		
		Then we choose the constants \(M_1\) and \(M_2\) as 
		\begin{equation}\nonumber
			M_1 = \frac{2d}{c_2 - \alpha_{\overline{\delta}} + \alpha_1} > 0 \quad \text{(due to \eqref{c_2} and \eqref{al_de})}, \quad M_2 = \frac{1}{2\delta^2}.
		\end{equation}
		We then require that \(\frac{1}{8\delta^2} > \frac{2d}{c_2 - \alpha_{\overline{\delta}} + \alpha_1}\) by choosing a smaller \(\delta\) if necessary. Consequently, we have
		\[
		\overline{z}_1(1-\delta) = \frac{2d}{c_2 - \alpha_{\overline{\delta}} + \alpha_1} + \delta^2 < \frac{1}{4\delta^2} = \overline{z}_3(\delta)
		\]
		if \(\delta\) also satisfies \(\delta<8^{-1/4}\).
		One can further select an even smaller \(\delta > 0\) such that
		\[
		M_2 = \frac{1}{2\delta^2} \ge \frac{c_1 + \gamma}{2d}.
		\]
		
		Moreover, since  \(\gamma > 0\) can be chosen larger if necessary such that \(-\gamma - \alpha_1 - \mu(a, 1) < 0\), it is possible to find a function \(\overline{w}_2\) satisfying
		\begin{equation}\label{partial_a_w_2}
			\partial_a \overline{w}_2 > -\left(2\gamma + \alpha_1 + \mu(a, 1)\right)\overline{w}_2 \quad \text{on } [0, a_c] \times [\delta, 1-\delta].
		\end{equation}
		
		We can choose \(\overline{w}_2 \in C([0, a_c] \times [\delta, 1-\delta])\) that possesses the following properties:
		\begin{equation}\label{partial_x_bar_w_2}
			\begin{cases}
				\partial_x \overline{w}_2 \le -\tilde{c}_4(\delta) & \text{on } [0, a_c] \times [\delta, 1-\delta], \\
				\partial_{xx} \overline{w}_2 \le \tilde{c}_3(\delta) & \text{on } [0, a_c] \times [\delta, 1-\delta], \\
				\int_0^{a_c} \beta(a, \cdot) \overline{w}_2(a, \cdot) \, da \le \overline{w}_2(0, \cdot) & \text{on } [\delta, 1-\delta],
			\end{cases}
		\end{equation}
		and for any \(a \in [0, a_c]\), the following boundary conditions hold:
		\[
		\begin{cases}
			\overline{w}_2(a, \delta) = \overline{z}_3(\delta) e^{-\gamma a} h(a), \\
			\overline{w}_2(a, 1-\delta) = \overline{z}_1(1-\delta) h_{\overline{\delta}}(a), \\
			\partial_x \overline{w}_2(a, \delta^+) < \partial_x \overline{w}_2(a, \delta^-) = -2M_2^2 \delta e^{-\gamma a} h(a), \\
			-2\delta h_{\overline{\delta}}(a) = \partial_x \overline{w}_2(a, (1-\delta)^+) < \partial_x \overline{w}_2(a, (1-\delta)^-).
		\end{cases}
		\]
		Here, the positive constants \(\tilde{c}_3 = \tilde{c}_3(\delta)\) and \(\tilde{c}_4 = \tilde{c}_4(\delta)\) are independent of \(\Lambda\).
		With a fixed \(a\), an illustration of the profile of \(\overline{w}\) can be found in the appendix, as shown in Figure \ref{overline w}.
		
		Next, let us verify that \(\overline{w}\) is indeed a strict generalized super-solution of problem \eqref{auxNge}. For convenience, we introduce the operator
		\[
		\overline{\mathcal{L}}w := w_a - d w_{xx} - \Lambda q w_x + \overline{\Psi}_\epsilon w.
		\]
		To check that \(\overline{w}\) is a strict and positive generalized super-solution of \eqref{auxNge}, we need to distinguish three cases depending on the value of \(x\).  
		
		\textbf{Case 1.1: $x\in[1-\de, 1]$}. We start by computing $(\overline{\mathcal L}\overline w)(a, x)$. By the choice of $M_1$, there holds
		\begin{eqnarray}\label{case1}
			(\overline{\mathcal L}\overline w)(a, x)&=&h_{\overline\de}(a)\left[-\al_{\overline\de}\overline z_1(x)-d\overline z''_1(x)-\La q(x)\overline z'_1(x)+\overline\Phi_\ep(a, x)\overline z_1(x)+\al_1\overline z_1(x)\right]\nonumber\\
			&=&h_{\overline\de}(a)\left[-\al_{\overline\de}\overline z_1(x)-2d-2\La q(x)(x-1)+\overline\Phi_\ep(a, x)\overline z_1(x)+\al_1\overline z_1(x)\right]\nonumber\\
			&\ge&h_{\overline\de}(a)\left[-2d+(c_2-\al_{\overline\de}+\al_1)(M_1+(x-1)^2)\right]\nonumber\\
			&\ge&0.
		\end{eqnarray}
		In addition, one has $\overline w_x(a, 1)=h_{\overline\de}(a)\overline z'_{1}(1)=0$ and also by \eqref{fL_de}, 
		\begin{eqnarray}\label{case22}
			\int_0^{a_c}\beta(a, x)\overline w(a, x)da&=&\overline z_1(x)\int_0^{a_c}\beta(a, x)h_{\overline\de}(a)da\nonumber\\
			&\le& \overline z_1(x)\int_0^{a_c}(\be(a, 1)+\overline\de)h_{\overline\de}(a)da\nonumber\\
			&=& h_{\overline\de}(0)\overline z_1(x)=\overline w(0, x).
		\end{eqnarray}

		\textbf{Case 1.2: $x\in[0, \de]$}. By the choice of $M_2$, it is easily seen that
		\begin{eqnarray}\label{case3}
			(\overline{\mathcal L}\overline w)(a, x)&=&-\ga\overline w+e^{-\ga a}h(a)\left[-d\overline z''_3(x)-\La q(x)\overline z'_3(x)+\overline\Phi_\ep(a, x)\overline z_3(x)\right]\nonumber\\
			&=&e^{-\ga a}h(a)\left[-\ga\overline z_3+2dM_2^2+2\La q(x)M_2^2x+\overline\Phi_\ep(a, x)\overline z_3(x)\right]\nonumber\\
			&\ge&e^{-\ga a}h(a)\left[2dM_2^2-(c_1+\ga)M_2(1-M_2x^2)\right]\nonumber\\
			&\ge&0.
		\end{eqnarray}
		Moreover, one has
		$$
		\overline w_x(a, 0)=e^{-\ga a}h(a)\overline z'_3(0)=0,
		$$
		and due to \eqref{choice_of_gamma}, there holds
		\begin{eqnarray}\label{ga}
			\int_0^{a_c}\beta(a, x)\overline w(a, x)da=\overline z_3(x)\int_0^{a_c}\beta(a, x)e^{-\ga a}h(a)da\le h(0)\overline z_3(x)=\overline w(0, x).
		\end{eqnarray}
		
		\textbf{Case 1.3: $x\in[\de, 1-\de]$}. We first note that $q_{\min}:=\min_{x\in[0, 1]}q(x)>0$. In view of \eqref{partial_a_w_2} and the monotonicity of
		$\overline w$ in $x\in[0,1]$, direct calculation yields 
		\begin{eqnarray}\label{case4}
			&&(\overline{\mathcal L}\overline w)(a, x)\nonumber\\
			&=&\partial_a\overline w_2(a, x)-d\partial_{xx}\overline w_2(a, x)-\La q(x)\partial_x\overline w_2(a, x)+\overline\Psi_\ep(a, x)\overline w_2(a, x)\nonumber\\
			&\ge&-d\tilde c_3+\La q_{\min}\tilde c_4-(2\ga+c_1)M_2(1-M_2\de^2)\max_{a\in[0, a_c]}h(a)\nonumber\\
			&\ge&0
		\end{eqnarray}
		provided that $\La$ is sufficiently large such that
		$$
		\La\ge \frac{d\tilde c_3+(2\ga+c_1)M_2(1-M_2\de^2)\max_{a\in[0, a_c]}h(a)}{q_{\min}\tilde c_4}.
		$$ 
		
		As a result, considering equations \eqref{case1}-\eqref{case4} and \eqref{partial_x_bar_w_2}, for sufficiently large $\La$, the previously selected $\overline{w}$ serves as the desired generalized super-solution for \eqref{auxNge} with $\mathbb{X}=\{\de, 1-\de\}$. By applying Proposition \ref{equivalent} to \eqref{auxNge}, it follows immediately that $\la_0(\overline\Psi_\ep) \le 0$. Consequently, we have
		$$
		\lambda_0(d, \La) \le \alpha_1 + \epsilon \quad \text{for all sufficiently large } \La.
		$$
		This implies \eqref{Nge} by letting $\La \to \infty$ due to the arbitrariness of $\epsilon$.      
		
		{\bf Step 2.} In this step, we aim to show
		\begin{equation}\label{Nle}
			\liminf_{\La\to\infty}\lambda_0(d, \La)\ge\al_1.
		\end{equation}
		We will employ a strategy similar to that used in {\bf Step 1}. For any fixed small \(\epsilon > 0\), again in light of \eqref{equ-ei}, let us consider the following auxiliary eigenvalue problem:
		\begin{equation}\label{auxNle}
			\begin{cases}
				w_a - dw_{xx} - \Lambda q(x) w_x + \underline{\Psi}_\epsilon(a, x) w = -\lambda w, & (a, x) \in (0, a_c) \times (0, 1), \\
				w_x(a, 0) = w_x(a, 1) = 0, & a \in (0, a_c), \\
				w(0, x) = \int_0^{a_c} \beta(a, x) w(a, x) \, da, & x \in (0, 1),
			\end{cases}
		\end{equation}
		where 
		\[
		\underline{\Psi}_\epsilon(a, x) = \mu(a, x) + \alpha_1 - \epsilon = [\mu(a, x) - \mu(a, 1) - \epsilon] + [\mu(a, 1) + \alpha_1] =: \underline{\Phi}_\epsilon(a, x) + \mu(a, 1) + \alpha_1.
		\]
		
		Let \(\lambda_0(\underline{\Psi}_\epsilon)\) denote the principal eigenvalue of \eqref{auxNle}. We choose a small constant \(0 < \delta < \epsilon\) and define
		\[
		\tilde{c}_1 = \tilde{c}_1(\epsilon) = \max_{[0, a_c] \times [0, 1]} |\underline{\Phi}_\epsilon(a, x)| + 1, \quad \tilde{c}_2 = \tilde{c}_2(\epsilon, \delta) = \max_{[0, a_c] \times [1-\delta, 1]} \underline{\Phi}_\epsilon(a, x).
		\]
		Note that \(\tilde{c}_1\) does not depend on the small \(\delta\). Furthermore, one can require \(\delta\) to be sufficiently small such that
		\begin{equation}\label{tilde c_2}
			-2\epsilon < \tilde{c}_2 < -\frac{1}{2}\epsilon.
		\end{equation}
		
		On the other hand, for some sufficiently small \(\underline{\delta} > 0\), we define a positive function \(h_{\underline{\delta}}\) through
		\begin{equation}\label{h_1}
			\begin{cases}
				h'_{\underline{\delta}}(a) = -\mu(a, 1) h_{\underline{\delta}}(a) - \alpha_{\underline{\delta}} h_{\underline{\delta}}(a),\\
				h_{\underline{\delta}}(0) = \int_0^{a_c} (\beta(a, 1) - \underline{\delta}) h_{\underline{\delta}}(a) \, da,
			\end{cases}
		\end{equation}
		where the existence of \(h_{\underline{\delta}}\) is guaranteed by the choice of \(\alpha_{\underline{\delta}}\) which fulfills
		\begin{equation}\nonumber
			\int_0^{a_c} (\beta(a, 1) - \underline{\delta}) e^{-\int_0^a \mu(s, 1) \, ds} e^{-\alpha_{\underline{\delta}} a} \, da = 1.
		\end{equation}
		We can select \(\delta > 0\) small enough such that
		\begin{equation}\label{choice of de}
			\beta(a, x) \ge \beta(a, 1) - \underline{\delta}, \quad \text{for all } x \in [1-\delta, 1],
		\end{equation}
		as well as 
		\begin{equation}\label{al_de sub}
			0 < \alpha_1 - \alpha_{\underline{\delta}} < \frac{\epsilon}{2}.
		\end{equation}
		
		To establish \eqref{Nle}, it suffices to construct a strictly positive generalized sub-solution \(\underline{w}\) for the following problem 
		\begin{equation}\label{auxNle-a}
			\begin{cases}
				w_a - dw_{xx} - \Lambda q(x) w_x + \underline{\Psi}_\epsilon(a, x) w =0, & (a, x) \in (0, a_c) \times (0, 1), \\
				w_x(a, 0) = w_x(a, 1) = 0, & a \in (0, a_c), \\
				w(0, x) = \int_0^{a_c} \beta(a, x) w(a, x) \, da, & x \in (0, 1),
			\end{cases}
		\end{equation}
		and then apply Corollary \ref{corsub}. To this end, we define a function \(\underline{z}(x)\) as
		\begin{equation}\nonumber
			\underline{z}(x) = 
			\begin{cases}
				\underline{z}_1(x) := \tilde{M}_1 - (x-1)^2, & \text{if } x \in [1-\delta/2, 1], \\
				\underline{z}_2(x), & \text{if } x \in (1-\delta, 1-\delta/2), \\
				0, & \text{if } x \in [0, 1-\delta],
			\end{cases}
		\end{equation}
		where \(\tilde{M}_1 > 0\) is a constant to be determined later. We choose \(\underline{z}_2 \in C([1-\delta, 1-\delta/2])\) to fulfill the following properties:
		\[
		\begin{cases}
			\underline{z}'_2(x) > \underline{c}_4(\delta), \quad |\underline{z}''_2(x)| \le \underline{c}_3(\delta), & x \in (1-\delta, 1-\delta/2), \\
			\underline{z}_2(1-\delta) = 0, \quad \underline{z}_2(1-\delta/2) = \tilde{M}_1 - \delta^2/4, \\
			\underline{z}'_2(1-\delta/2)^- < \delta = \underline{z}'_2(1-\delta/2)^+,
		\end{cases}
		\]
		where \(\underline{c}_3 = \underline{c}_3(\delta)\) and \(\underline{c}_4 = \underline{c}_4(\delta)\) are two positive constants independent of \(\Lambda\).
		
		Now, set \(\underline{w} = h_{\underline{\delta}} \underline{z}\) and let us verify that \(\underline{w}\) is indeed a strict generalized sub-solution of \eqref{auxNle-a}. For convenience, we introduce the operator
		\[
		\underline{\mathcal{L}}w := w_a - dw_{xx} - \Lambda q w_x + \underline{\Psi}_\epsilon w.
		\]
		Clearly, if $x\in [0, 1-\delta]$, $\underline{\mathcal{L}}(\underline w)(a, x)=0$. In the following, we check $x\in[1-\de/2, 1]$ and $x\in(1-\de, 1-\de/2)$, separately.
		
		\textbf{Case 2.1:  $x\in[1-\de/2, 1]$}. In this region, in view of \eqref{tilde c_2} and \eqref{al_de sub}, elementary computation gives
		\begin{eqnarray}\label{case2.1-1 sub}
			(\underline{\mathcal L}\,\underline w)(a, x)&=&h_{\underline\de}(a)\left[-\al_{\underline\de}\underline z_1(x)+2d+2\La q(x)(x-1)+(\underline\Phi_\ep(a, x)+\al_1)\underline z_1(x)\right]\nonumber\\
			&\le&h_{\underline\de}(a)\left[2d+(\tilde c_2+\al_1-\al_{\underline\de})(\tilde M_1-\de^2/4)\right]\nonumber\\
			&\le&0
		\end{eqnarray}
		provided $\tilde M_1=\frac{2d}{\al_{\underline\de}-\al_1-\tilde c_2}+\de^2/4>0$. 
		
		\textbf{Case 2.2:  $x\in(1-\de, 1-\de/2)$}. In this region, we have 
		\begin{eqnarray}\label{case2.1-2 sub}
			(\underline{\mathcal L}\,\underline w)(a, x)&=&h_{\underline\de}(a)\left[-\al_{\underline\de}\underline z_2(x)-d\underline z''_2(x)-\La q(x)\underline z'_2(x)+(\underline\Phi_\ep(a, x)+\al_1)\underline z_2(x)\right]\nonumber\\
			&\le&h_{\underline\de}(a)(d\underline c_3-\La q_{\min}\underline c_4)\nonumber\\
			&\le&0
		\end{eqnarray}
		provided that 
		$$
		\La\ge\frac{d\underline c_3}{q_{\min}\underline c_4}.
		$$
		
		On the other hand, one has $\underline w_x(a, 0)=\underline w_x(a, 1)=0$ and on $[1-\de, 1]$,
		\begin{equation}\label{underline w(0)}
			\int_0^{a_c}\beta(a, x)\underline w(a, x)da=\underline z(x)\int_0^{a_c}\beta(a, x)h_{\underline\de}(a)da\ge\underline z(x)\int_0^{a_c}(\beta(a, 1)-\underline\de)h_{\underline\de}(a)da=\underline w(0, x)
		\end{equation}
		due to \eqref{choice of de} and the second equality in \eqref{h_1}. 
		
		As a consequence, in view of \eqref{case2.1-1 sub}-\eqref{underline w(0)} for $\La$ large enough, the above chosen $\underline w$ is a desired sub-solution of \eqref{auxNle-a} with $\mathbb{X}=\{1-\de/2, 1-\de\}$. By Corollary \ref{corsub}, as applied to \eqref{auxNle}, we deduce that $\la_0(\underline\Psi_\ep)\ge0$, and so
		$$
		\lambda_0(d, \La)\ge\al_1-\ep \ \ \text{ for all large }\La,
		$$
		which implies \eqref{Nle} due to the arbitrariness of $\ep$. Thus, the first assertion in \eqref{la-1} is established.
		
		Finally, the second assertion in \eqref{la-1} as \(\Lambda \to -\infty\) can be directly derived from the case where \(\Lambda \to \infty\) by applying the variable change \(y = -x\) to problem \eqref{linear}. 
		Thus, the proof is complete. \end{proof}

	\subsection{Limiting behavior for small diffusion with advection}
	This subsection is devoted to the proof of Theorem \ref{Dlambda}.  	
	
	\begin{proof}[Proof of Theorem \ref{Dlambda}]
		Since $\la_0(d, \La)$ is a simple eigenvalue, the continuity of $d\mapsto \la_0(d, \La)$ follows from Kato \cite[Section IV. 3.5]{kato2013perturbation} for the classical perturbation theory. 
		
		First of all, as in the proof of Theorem \ref{Lalambda}, we just need to verify the assertion in \eqref{small-d} for the case of $\La > 0$. From now on, we fix $\La > 0$. Our proof consists of two main steps. In the following, for the sake of convenience, we shall use the same notations as in the proof of Theorem \ref{Lalambda}.
		
		\textbf{Step 1.} We first prove
		\begin{equation}\label{Nge small d}
			\limsup_{d\to0}\lambda_0(d, \La)\le\al_1.
		\end{equation}
		To show \eqref{Nge small d}, it is sufficient to prove that, for fixed $\ep>0$, 
		\begin{equation}\label{<0 small d}
			\la_0(\overline\Psi_\ep)<0\ \ \text{ for all small $d$}.
		\end{equation}
		
		We are going to construct a positive function $\overline{w}$ in a separated form, which serves as a generalized super-solution of \eqref{auxNge}. To achieve this, let us define 
		\begin{equation}\label{small d super-sol}
			\overline{w}(a, x) = 
			\begin{cases}
				h_{\overline{\delta}}(a)e^{\overline{v}_1(x)}, & x \in [1-\delta, 1], \\
				e^{\overline{v}_2(a, x)}, & x \in [\delta, 1-\delta], \\
				e^{-\gamma a}h(a)e^{\overline{v}_3(x)}, & x \in [0, \delta],
			\end{cases}
		\end{equation}
		where $\gamma > 0$ is chosen such that \eqref{choice_of_gamma} holds, and $h$ and $h_{\overline{\delta}}$ are defined in \eqref{fL} and \eqref{fL_de}, respectively. Here, $\overline{v}_1$ and $\overline{v}_3$ are given by
		\begin{equation}\nonumber
			\overline{v}_1(x) = \frac{1}{\delta^2}(x-1)^2, \quad \overline{v}_3(x) = \frac{1}{\delta}e^{-\frac{1}{\delta}x},
		\end{equation}
		and the function $\overline{v}_2$ will be determined later.
		
		Observe that $\overline{v}_1$ and $\overline{v}_3$ are strictly decreasing in $[1-\delta, 1]$ and $[0, \delta]$, respectively. Furthermore, we have
		\[
		-\frac{1}{e\delta^2} = \overline{v}'_3(\delta) < \overline{v}'_1(1-\delta) = -\frac{2}{\delta} < 0, \quad \overline{v}_3(\delta) = \frac{1}{e\delta} > 1 = \overline{v}_1(1-\delta),
		\]
		and
		\[
		\frac{\overline{v}_1(1-\delta) - \overline{v}_3(\delta)}{(1-\delta) - \delta} \approx -\frac{1}{e\delta}.
		\]
		
		In light of the choices of $\overline{v}_1$ and $\overline{v}_3$, we then select $\overline{v}_2 \in C([0, a_c] \times [0, 1])$ with the following properties (refer to \eqref{partial_a_w_2} and \eqref{partial_x_bar_w_2}):
		\begin{equation}\label{partial_x bar e w_2}
			\begin{cases}
				-c_5(\delta) \leq \partial_x \overline{v}_2 \leq -c_4 \delta^{-1/2} & \text{on } [0, a_c] \times [\delta, 1-\delta], \\
				\partial_{xx} \overline{v}_2 \leq c_3(\delta) & \text{on } [0, a_c] \times [\delta, 1-\delta], \\
				\partial_a \overline{v}_2 \geq -(2\gamma + \alpha_1 + \mu(a, 1)) & \text{on } [0, a_c] \times [\delta, 1-\delta], \\
				\int_0^{a_c} \beta(a, \cdot) e^{\overline{v}_2(a, \cdot)} \, da \leq e^{\overline{v}_2(0, \cdot)} & \text{on } [\delta, 1-\delta],
			\end{cases}
		\end{equation}
		and for any \(a \in [0, a_c]\),
		\[
		\begin{cases}
			e^{\overline{v}_2(a, \delta)} = e^{\overline{v}_3(\delta)} e^{-\gamma a} h(a), \\
			e^{\overline{v}_2(a, 1-\delta)} = e^{\overline{v}_1(1-\delta)} h_{\overline{\delta}}(a), \\
			\partial_x \overline{v}_2(a, \delta^+) < \partial_x \overline{v}_2(a, \delta^-) = -\frac{1}{e\delta^2}, \\
			-\frac{2}{\delta} = \partial_x \overline{v}_2(a, (1-\delta)^+) < \partial_x \overline{v}_2(a, (1-\delta)^-),
		\end{cases}
		\]
		where the positive constants \(c_3 = c_3(\delta)\) and \(c_5 = c_5(\delta)\) are independent of \(d\), and the positive constant \(c_4\) is independent of both \(d\) and \(\delta\). With a fixed \(a\),  the profile of \(\overline{w}\) here is similar to that in Step 1 in the proof of Theorem \ref{Lalambda}, as shown in Figure \ref{overline w}.
		
		Next we further choose $\de>0$ such that
		\begin{equation}\label{de}
			0<\de<\min\left\{\ep,\ \ \sqrt{\frac{\La q_{\min}}{e(c_1+\ga)}},\ \ \left(\frac{\La  q_{\min}c_4}{c_1+2\ga}\right)^2\right\}.
		\end{equation}
		Setting $\overline z_i=e^{\overline v_i}$, we have
		$$
		\overline z'_i=\overline z_i\overline v'_i,\ \ \; \overline z''_i=\overline z_i[(\overline v'_i)^2+\overline v''_i],\ \text{ for }\ i=1, 3.
		$$
		Thus $\overline z_1$, $\overline z_2$ and $\overline z_3$ are also decreasing, respectively, on $[0, \de]$, $[\de,1-\de]$ and $[1-\de,1]$, and $\overline z'_3(0)<0,\ \overline z'_1(1)=0$. In addition, there holds
		$$
		\overline{\mathcal L}\overline z_i=\overline z_i\{-d[(\overline v'_i)^2+\overline v''_i]-\La q\overline v'_i+\overline\Phi_\ep+\mu(\cdot, 1)+\al_1\},\ \ i=1,3.
		$$
		
		In the sequel, we will verify that $\overline w$ defined by \eqref{small d super-sol} is indeed a strict generalized super-solution of \eqref{auxNge}.
		
		\textbf{Case 1.1: $x\in[1-\de, 1]$}.  By direct computation we have
		\begin{eqnarray}\label{case1 small d}
			(\overline{\mathcal L}\overline w)(a, x)&=&h_{\overline\de}(a)\overline z_1(x)\left[\al_1-\al_{\overline\de}-d\left(\frac{4}{\de^4}(x-1)^2+\frac{2}{\de^2}\right)-\frac{2\La q(x)}{\de^2}(x-1)+\overline\Phi_\ep(a, x)\right]\nonumber\\
			&\ge&h_{\overline\de}(a)\overline z_1(x)\left[-\frac{6d}{\de^2}+c_2+\al_1-\al_{\overline\de}\right]\nonumber\\
			&\ge&0
		\end{eqnarray}
		provided that 
		$$
		0<d\le \frac{(c_2+\al_1-\al_{\overline\de})\de^2}{6},
		$$
		and  the same analysis as in \eqref{case22}  gives
		\begin{eqnarray}\label{case22 small d}
			\int_0^{a_c}\beta(a, x)\overline w(a, x)da=\overline z_1(x)\int_0^{a_c}\beta(a, x)h_{\overline\de}(a)da\le\overline w(0, x).\nonumber
		\end{eqnarray}
		
		\textbf{Case 1.2: $x\in[0, \de]$}. In this region, we have
		\begin{eqnarray}\nonumber
			(\overline{\mathcal L}\overline w)(a, x)&=&-\ga\overline w+e^{-\ga a}h(a)\overline z_3(x)\left[-d\frac{1}{\de^2}\overline v_3(1+\overline v_3)+\frac{\La}{\de} q(x)\overline v_3(x)+\overline\Phi_\ep(a, x)\right]\nonumber\\
			&\ge&e^{-\ga a}h(a)\overline z_3(x)\left[-\frac{1+\de}{\de^4}d+\frac{\La q_{\min}}{e\de^2}-(\ga+c_1)\right]\nonumber\\
			&\ge&0\nonumber
		\end{eqnarray}
		provided that
		$$
		0<d\le\frac{\La q_{\min}-e(c_1+\ga)\de^2}{e(1+\de)}\de^2\ \ \ (\mbox{due to the choice of}\, \de;\ \mbox{see}\ \eqref{de}).
		$$
		And thanks to \eqref{choice_of_gamma}, there holds
		\begin{eqnarray}\nonumber
			\int_0^{a_c}\beta(a, x)\overline w(a, x)da=\overline z_3(x)\int_0^{a_c}\beta(a, x)e^{-\ga a}h(a)da\le h(0)\overline z_3(x)=\overline w(0, x).
		\end{eqnarray}
		
		\textbf{Case 1.3:  $x\in[\de, 1-\de]$}. Direct calculations yields 
		\begin{eqnarray}\nonumber
			(\overline{\mathcal L}\overline w)(a, x)&=&e^{\overline v_2(a, x)}\left[\partial_a\overline v_2(a, x)-d\partial_{xx}\overline v_2(a, x)-d(\partial_x\overline v_2(a, x))^2-\La q(x)\partial_x\overline v_2(a, x)+\overline\Psi_\ep(a, x)\right]\nonumber\\
			&\ge&e^{\overline v_2(a, x)}\left[-d((c_5)^2+c_3)+c_4\La q_{\min}\de^{-1/2}-c_1-2\ga\right]\nonumber\\
			&\ge&0\nonumber
		\end{eqnarray}
		provided that $d$ satisfies 
		\begin{eqnarray}
			0<d<\frac{c_4\La q_{\min}\de^{-1/2}-c_1-2\ga}{(c_5)^2+c_3}\nonumber
		\end{eqnarray}
		due to the choice of $\de$ in \eqref{de}. By \eqref{partial_x bar e w_2}, we also have 
		\begin{eqnarray}\label{case44 small d}
			\int_0^{a_c}\beta(a, x)\overline w(a, x)da\le\overline w(0, x).
		\end{eqnarray}
		
		As a consequence, from \eqref{case1 small d} to \eqref{case44 small d}, it follows that when
		\[
		0<d < \min\left\{\frac{(c_2 + \alpha_1 - \alpha_{\overline{\delta}}) \delta^2}{6}, \; \frac{\Lambda q_{\min} - e(c_1 + \gamma) \delta^2}{e(1 + \delta)} \delta^2, \; \frac{c_4 \Lambda q_{\min} \delta^{-1/2} - c_1 - 2\gamma}{(c_5)^2 + c_3}\right\},
		\]
		the chosen function $\overline{w}$ satisfies \eqref{auxNge}. Hence, $\overline{w}$ is the desired generalized super-solution with $\mathbb{X} = \{\delta, 1-\delta\}$. Therefore, by Proposition \ref{equivalent}, \eqref{<0 small d} holds, and thus \eqref{Nge small d} is also valid.

		\textbf{Step 2.} In this step, we are going to show
		\begin{equation}\label{Nle small d}
			\liminf_{d\to0}\lambda_0(d, \La)\ge\al_1.
		\end{equation}
		We will employ the similar strategy to that in {\textbf {Step 1}} above. For any fixed small $\ep>0$, let us recall the auxiliary eigenvalue problem \eqref{auxNle}. Then, we proceed to construct a positive generalized  sub-solution to \eqref{auxNle-a} of the following form:
		\begin{equation}\label{small d sub-sol}
			\underline w(a, x)=\begin{cases}
				h_{\underline\de}(a)e^{\underline v_1(x)}, &x\in[1-\de, 1],\\
				e^{\underline v_2(a, x)}, &x\in[\de, 1-\de],\\
				e^{\hat\ga a}h(a)e^{\underline v_3(x)}, &x\in[0, \de],
			\end{cases}
		\end{equation} 
		where $h_{\underline \de}$ is defined in \eqref{h_1} and $\hat\ga>0$ is chosen such that
		\begin{equation}\label{choice of hat ga}
			\int_0^{a_c}\be(a, x)e^{\hat\ga a}h(a)da\ge h(0),\ \;\forall x\in[0, 1].
		\end{equation}
		Here $\underline v_1$ and $\underline v_3$ are given by
		\begin{equation}\nonumber
			\underline v_1(x)=\frac{1}{\de^2}e^{-\frac{1}{\de^2}(x-1)^2},\ \ \; \underline v_3(x)=\frac{1}{\de}(x+1).
		\end{equation} 
		Clearly, $\underline v_1$ and $\underline v_3$ are strictly increasing in $[1-\de, 1]$ and $[0, \de]$ respectively. Furthermore,
		$$
		\frac{1}{\de}=\underline v'_3(\de)<\underline v'_1(1-\de)=\frac2{e\de^3}, \quad \underline v_3(\de)=\frac{1+\de}{\de}<\frac{1}{e\de^2}=\underline v_1(1-\de),
		$$
		and
		$$
		\frac{\underline v_1(1-\de)-\underline v_3(\de)}{(1-\de)-\de}\approx\frac{1}{e\de^2}.
		$$
		
		By requiring \(\hat{\gamma} > 0\)  to be larger (if necessary) such that \(\hat{\gamma} - \alpha_1 - \mu(a, 1) > 0\), we can always find a solution \(\underline{v}_2\) such that
		\begin{equation}\nonumber
			\partial_a \underline{v}_2 < 2\hat{\gamma} - \alpha_1 - \mu(a, 1) \quad \text{on } [0, a_c] \times [\delta, 1-\delta].
		\end{equation}
		Therefore, we can choose \(\underline{v}_2(a, x) \in C([0, a_c] \times [0, 1])\) with the properties
		\begin{equation}\label{v2}
			\begin{cases}
				\partial_x \underline{v}_2 \geq \widehat{c}_4 \delta^{-1/2} & \text{on } [0, a_c] \times [\delta, 1-\delta], \\
				|\partial_{xx} \underline{v}_2| \leq \widehat{c}_3(\delta) & \text{on } [0, a_c] \times [\delta, 1-\delta], \\
				\int_0^{a_c} \beta(a, \cdot) e^{\underline{v}_2(a, \cdot)} \, da \geq e^{\underline{v}_2(0, \cdot)} & \text{on } [\delta, 1-\delta],
			\end{cases}
		\end{equation}
		and for any \(a \in [0, a_c]\),
		\[
		\begin{cases}
			e^{\underline{v}_2(a, \delta)} = e^{\underline{v}_3(\delta)} e^{\hat{\gamma} a} h(a), \\
			e^{\underline{v}_2(a, 1-\delta)} = e^{\underline{v}_1(1-\delta)} h_{\underline{\delta}}(a), \\
			\partial_x \underline{v}_2(a, \delta^+) > \partial_x \underline{v}_2(a, \delta^-) = \frac{1}{\delta}, \\
			\frac{2}{e \delta^3} = \partial_x \underline{v}_2(a, (1-\delta)^+) > \partial_x \underline{v}_2(a, (1-\delta)^-),
		\end{cases}
		\]
		where the positive constant \(\widehat{c}_3(\delta)\) is independent of \(d\), and the positive constant \(\widehat{c}_4\) is independent of both \(d\) and \(\delta\). Given \(a\), the profile of \(\underline{w}\) is illustrated in Figure \ref{underline w} of the appendix.
		
		Next, we pick \(\delta > 0\) such that
		\begin{equation}\label{de small d}
			0 < \delta < \min\left\{\epsilon, \; \frac{\Lambda q_{\min}}{\tilde{c}_1 + \hat{\gamma}}, \; \left(\frac{\Lambda q_{\min} \widehat{c}_4}{\tilde{c}_1 + 2\hat{\gamma}}\right)^2\right\}.
		\end{equation}
		As before, by setting \(\underline{z}_i = e^{\underline{v}_i}\), we have
		\[
		\underline{z}'_i = \underline{z}_i \underline{v}'_i, \quad \underline{z}''_i = \underline{z}_i [(\underline{v}'_i)^2 + \underline{v}''_i], \quad \text{for } i = 1, 3.
		\]
		Thus, \(\underline{z}_i\) is also increasing and \(\underline{z}'_3(0) > 0,\ \underline{z}'_1(1) = 0\). In addition, we have
		\[
		\underline{\mathcal{L}}\, \underline{z}_i = \underline{z}_i \left\{-d [(\underline{v}'_i)^2 + \underline{v}''_i] - \Lambda q \underline{v}'_i + \underline{\Phi}_\epsilon + \mu(\cdot, 1) + \alpha_1\right\},\ i=1,3.
		\]
		
		Now, let us verify that \(\underline{w}\) is indeed a strict positive generalized sub-solution of \eqref{auxNle-a}.
		
		{\bf Case 2.1: \(x \in [1-\delta, 1]\)}. By means of \eqref{tilde c_2} and \eqref{al_de sub}, it easily follows that
		\begin{eqnarray}\label{case1sub small d}
			&&(\underline{\mathcal L}\, \underline w)(a, x)\nonumber\\
			&=&h_{\underline\de}(a)\overline z_1(x)\left\{-d\left[\frac{4(x-1)^2}{\de^4}\underline v_1^2+\left(\frac{4(x-1)^2}{\de^4}-\frac{2}{\de^2}\right)\underline v_1\right]+\frac{2(x-1)}{\de^2}\La q(x)\underline v_1(x)+\underline\Phi_\ep(a, x)+\al_1-\al_{\underline\de}\right\}\nonumber\\
			&\le&h_{\underline\de}(a)\overline z_1(x)\left[\frac{2d}{\de^4}+\tilde c_2+\al_1-\al_{\underline\de}\right]\nonumber\\
			&\le&0
		\end{eqnarray}
		provided that
		\[
		0<d \leq \frac{\delta^4 (\alpha_{\underline{\delta}} - \tilde{c}_2 - \alpha_1)}{2}.
		\]

		{\bf Case 2.2:  \(x \in [0, \delta]\)}. In this region, we have
		\begin{align}\label{case3sub small}
			(\underline{\mathcal{L}}\, \underline{w})(a, x) &= \hat{\gamma} \underline{w} + e^{\hat{\gamma} a} h(a) \underline{z}_3(x) \left[-\frac{1}{\delta^2} d - \frac{\Lambda}{\delta} q + \underline{\Phi}_\epsilon(a, x) \right] \nonumber \\
			&\leq e^{\hat{\gamma} a} h(a) \underline{z}_3(x) \left[-\frac{\Lambda}{\delta} q_{\min} + \hat{\gamma} + \tilde{c}_1 \right] \nonumber \\
			&\leq 0
		\end{align}
		due to the choice of \(\delta\) in \eqref{de small d}. 
		
		On the other hand, by \eqref{choice of hat ga}, we have
		\begin{equation}\label{gasub small}
			\int_0^{a_c} \beta(a, x) \underline{w}(a, x) \, da = \underline{z}_3(x) \int_0^{a_c} \beta(a, x) e^{\hat{\gamma} a} h(a) \, da \geq h(0) \underline{z}_3(x) = \underline{w}(0, x).
		\end{equation}
		
		{\bf Case 2.3: \(x \in [\delta, 1-\delta]\)}. Recall that \(q_{\min} = \min_{x \in [0, 1]} q(x) > 0\). Direct calculations yield
		
		\begin{align}\label{case4sub small}
			(\underline{\mathcal{L}}\, \underline{w})(a, x) &= e^{\underline{v}_2(a, x)} \left[\partial_a \underline{v}_2(a, x) - d \partial_{xx} \underline{v}_2(a, x) - d (\partial_x \underline{v}_2(a, x))^2 - \Lambda q(x) \partial_x \underline{v}_2(a, x) + \underline{\Psi}_\epsilon(a, x) \right] \nonumber \\
			&\leq e^{\underline{v}_2(a, x)} \left[d \widehat{c}_3 - \widehat{c}_4 \delta^{-1/2} \Lambda q_{\min} + \tilde{c}_1 + 2\hat{\gamma} \right] \nonumber \\
			&\leq 0
		\end{align}
		provided that \(d\) satisfies
		\begin{equation}\nonumber
			0<d \leq \frac{\widehat{c}_4 \delta^{-1/2} \Lambda q_{\min} - \tilde{c}_1 - 2\hat{\gamma}}{\widehat{c}_3}
		\end{equation}
		by the choice of \(\delta\) in \eqref{de small d}. Furthermore, by \eqref{v2}, we infer
		\begin{equation}\label{case44sub small}
			\int_0^{a_c} \beta(a, x) \underline{w}(a, x) \, da \geq \underline{w}(0, x).
		\end{equation}
		
		Combining \eqref{case1sub small d}–\eqref{case4sub small} and \eqref{case44sub small}, for \(d\) satisfying
		\[
		0<d < \min\left\{\frac{\delta^4 (\alpha_{\underline{\delta}} - \tilde{c}_2 - \alpha_1)}{2}, \; \frac{\widehat{c}_4 \delta^{-1/2} \Lambda q_{\min} - \tilde{c}_1 - 2\hat{\gamma}}{\widehat{c}_3}\right\},
		\]
		we see that the chosen \(\underline{w}\) is a desired strict generalized sub-solution of \eqref{auxNle-a} with \(\mathbb{X} = \{\delta, 1-\delta\}\).
		By applying Corollary \ref{corsub} to \eqref{auxNle-a}, we conclude that \(\lambda_0(\underline{\Psi}_\epsilon) \geq 0\), and consequently
		\[
		\lambda_0(d, \Lambda) \geq \alpha_1 - \epsilon\ \text{ for small } d,
		\]
		leading to \eqref{Nle small d}.
	\end{proof}
	
	\subsection{Limiting behavior for small diffusion without advection}
	In this subsection, we present the proof of Theorem \ref{Dlambda La=0}.
	\begin{proof}[Proof Theorem \ref{Dlambda La=0}] Let $\al_{\max}$ be defined through \eqref{almax}.
		First of all, for any $x\in[0, 1]$, consider the following problem
		\begin{equation}\label{dto0}
			\begin{cases}
				v_a(a, x)=-\mu(a, x)v(a, x)-\al v, \quad a\in(0, a_c),\\
				v(0, x)=\int_{0}^{a_c}\beta(a, x)v(a, x)da.
			\end{cases}
		\end{equation}
		Observe that for any $x\in[0, 1]$, \eqref{dto0} admits a principal eigenpair $(\al(x), v(a, x))$ with $\al(x)\le\al_{\max}$ and $v>0$. Without loss of generality, we can further assume that $\al:[0, 1]\to\R$ is of $C^2$  and $v\in C^2([0, 1], W^{1, 1}(0, a_c))$. In fact, one can approximate $\mu$ uniformly from above and below by $C^2$ functions in $x$. Then given the enhanced regularity of $\mu$, by the same argument as in Ducrot et al. \cite[Appendix]{Ducrot2022Age-structuredII}, one has the desired regularity for $(\al(x), v(a, x))$ with respect to $x$.
		
		To proceed further, we need to make a pivotal change of variable to ensure the derivative of $v$ with respect to $x$ on the boundary to have fixed sign. Recall from \cite[Theorem 1.3]{fernandez2019singular} and \cite[Lemma 2.1]{lopez2013linear} that there exists $\psi\in C^{2+\theta}(\R)$ with some $0<\theta<1$ such that $\psi(x)<0$ for all $x\in(0, 1), \psi(0)=\psi(1)=0$ and $\psi(x)>0$ for all $x\in\R\setminus[0, 1]$ along with
		$$
		\min\{-\psi'(0),\ \ \psi'(1)\}>0.
		$$
		Then setting 
		\begin{equation}\label{h(x)}
			h(x)=e^{\varrho \psi(x)},\ \ \, x\in(0, 1),
		\end{equation}
		for some constant $\varrho>0$ to be determined later, we have that $h\in C^{2+\theta}(\R)$ satisfies $h(x)>0$ for all $x\in\R$. Let us now make the following transformations:
		$$
		w(a, x)=\frac{u(a, x)}{h(x)}\; \text{ and }\; r(a, x)=\frac{v(a, x)}{h(x)},\quad (a, x)\in[0, a_c]\times[0, 1],
		$$ 
		then the eigenvalue problem \eqref{linear} can be rewritten as
		\begin{equation}\nonumber
			\begin{cases}
				w_a=dw_{xx}+\frac{2d h'(x)}{h(x)}w_x+\frac{dh''(x)}{h(x)}w-\mu(a, x)w-\lambda w,&\quad (a, x)\in(0, a_c)\times(0, 1),\\
				w_x(a, x)+\varrho\psi'(x)w(a, x)=0,  &\quad a\in(0, a_c)\times\{0, 1\},\\
				w(0, x)=\int_{0}^{a_c}\beta(a, x)w(a, x)da, &\quad x\in(0, 1),
			\end{cases}
		\end{equation}
		and $r$ still satisfies \eqref{dto0}. 
		
		By choosing 
		$$
		\varrho>\max\left\{0,\ \,\left[\max_{a\in[0, a_c]}\frac{-r_x(a, 1)}{r(a, 1)}\right]\frac{1}{\psi'(1)}, \; \left[\max_{a\in[0, a_c]}\frac{r_x(a, 0)}{r(a, 0)}\right]\frac{1}{-\psi'(0)}\right\},
		$$
		one can check that 
		\begin{equation*}
			r_x(a, 1)+\varrho\psi'(1)r(a, 1)\ge0 \text{ and }r_x(a, 0)+\varrho\psi'(0)r(a, 0)\le0, \quad \forall a\in[0, a_c].
		\end{equation*}
		For any sufficiently small $\ep>0$, we consider the following auxiliary eigenvalue problem:
		\begin{equation}\label{Robin perturb}
			\begin{cases}
				w_a=dw_{xx}+\frac{2d h'(x)}{h(x)}w_x+\frac{dh''(x)}{h(x)}w-\mu(a, x)w-(\al_{\max}+\ep) w,&\quad (a, x)\in(0, a_c)\times(0, 1),\\
				w_x(a, x)+\varrho\psi'(x)w(a, x)=0,  &\quad a\in(0, a_c)\times\{0, 1\},\\
				w(0, x)=\int_{0}^{a_c}\beta(a, x)w(a, x)da, &\quad x\in(0, 1).
			\end{cases}
		\end{equation}
		
		We then define an operator $L$:
		\begin{equation}\label{L}
			L w:=w_{xx}+\frac{2 h'(x)}{h(x)}w_x+\frac{h''(x)}{h(x)}w, \quad w\in W^{2, p}(0, 1).
		\end{equation}
		Observe that $\norm{dLr}_{L^\infty((0, a_c)\times(0, 1))}\to0$ as $d\to0$. Thus for any $\ep>0$, one can choose $d$ sufficiently small such that 
		$$
		0<d\le\frac{\ep \min_{[0, a_c]\times[0, 1]}r(a, x)}{\norm{Lr}_{L^\infty((0, a_c)\times(0, 1))}},
		$$ 
		then $r$ satisfies
		$$
		r_a\ge dr_{xx}+\frac{2d h'(x)}{h(x)}r_x+\frac{dh''(x)}{h(x)}r-\mu(a, x)r-(\al_{\max}+\ep) r.
		$$
		By letting $\ep\to0$, one applies the comparison principle to \eqref{Robin perturb} to yield 
		$$
		\limsup\limits_{d\to0}w(a, x)\le r(a, x), 
		$$
		that is, $r(a, x)$ is a super-solution of \eqref{Robin perturb} for $d=d(\ep)$ small enough and thus $v(a, x)$ is a positive super-solution of \eqref{linear} for $d$ small enough. This implies that $\limsup\limits_{d\to0}\la_0(d, 0)\le\al_{\max}$. 
		
		Next let us prove the reverse inequality. To this aim, we first claim
		\begin{claim}\label{>0}
			Assume
			\begin{equation}\label{>1}
				\int_0^{a_c}\be(a, x)e^{-\int_0^a\mu(s, x)ds}da>1,\quad x\in[0, 1],
			\end{equation} 
			then $\la_0(d, 0)>0$ for sufficiently small $d>0$.
		\end{claim} 
		At this moment, we assume that the above claim is true. Then rewrite the eigenvalue problem \eqref{linear} with $\La=0$ as
		\begin{equation}\label{-ep}
			\begin{cases}
				u_a=du_{xx}-(\mu(a, x)+\zeta)u-(\la-\zeta)u,&\quad (a, x)\in(0, a_c)\times(0, 1),\\
				u_x(a, x)=0,  &\quad a\in(0, a_c)\times\{0, 1\},\\
				u(0, x)=\int_{0}^{a_c}\beta(a, x)u(a, x)da, &\quad x\in(0, 1),
			\end{cases}
		\end{equation}
		where $\zeta$ is a number such that
		$$
		\int_0^{a_c}\be(a, x)e^{-\zeta a}e^{-\int_0^a\mu(s, x)ds}da>1, \quad x\in[0, 1].
		$$
		Thanks to Claim \ref{>0}, the principal eigenvalue $\la_0$ to \eqref{-ep} satisfies $\la_0(d, 0)\ge\zeta$ for sufficiently small $d>0$. Thus $\liminf\limits_{d\to0}\la_0(d, 0)\ge\al_{\max}$ follows using the definition of $\al_{\max}$ (see \eqref{almax}) and the desired result is established. 
		
		In the sequel, we provide the proof of Claim \ref{>0}. Our argument is inspired by Hess \cite[Proposition 17.3]{hess1991periodic}. Assume to the contrary that there are $d_n\to0$ and associated solutions $u_n$ of \eqref{linear}, such that $\la_n:=\la_0(d_n, 0)\le0$ for all $n$. Let us first normalize $u_n$ by $\norm{u_n}_{L^1((0, a_c)\times(0, 1))}=1$. Since
		$$
		\la_n=\int_0^1\int_0^{a_c}(\be(a, x)-\mu(a, x))u_n-\int_0^1u_n(a_c, x),
		$$
		the boundedness of $\la_n$ follows. In the following we normalize $u_n$ by $\norm{u_n}_{C([0, a_c]\times[0, 1])}=1$. Let $(a_n, x_n)\in[0, a_c]\times[0, 1]$ be such that $u_n(a_n, x_n)=1$ for all $n$. Passing to a subsequence we may assume that $a_n\to\bar a\in[0, 1], x_n\to\bar x\in[0, 1]$ and $\la_n\to\bar\la\le0$.
		
		First consider the case that $\bar x\in(0, 1)$. For each $n\geq1$, define $\textbf x:={d_n}^{-1/2}(x-x_n)$ and let $v_n(a, \textbf x):=u_n(a, x_n+{d_n}^{1/2}\textbf x)$. Then $v_n$ satisfies
		\begin{equation}\label{new variable}
			\begin{cases}
				(v_n)_a=(v_n)_{\textbf x\textbf x}-\mu_n(a, \textbf x)v_n-\la_nv_n,&\quad (a, \textbf x)\in(0, a_c)\times\tilde\Omega_n,\\
				(v_n)_{\textbf x}(a, \textbf x)=0,  &\quad a\in(0, a_c)\times\partial\tilde\Omega_n,\\
				v_n(0, \textbf x)=\int_{0}^{a_c}\beta_n(a, \textbf x)v_n(a, \textbf x)da, &\quad \textbf x\in\tilde\Omega_n,
			\end{cases}
		\end{equation}
		where $\mu_n(a, \textbf x):=\mu(a, x_n+{d_n}^{1/2}\textbf x), \be_n(a, \textbf x):=\be(a, x_n+{d_n}^{1/2}\textbf x)$ and $\tilde\Omega_n:=\{{d_n}^{-1/2}(x-x_n): x\in(0, 1)\}$. Obviously, $0\le v_n\le1$, $v_n(a_n, 0)=1$, and $\text{dist}(0, \partial\tilde{\Omega}_n)\to\infty$ as $n\to\infty$. Therefore, passing up to a subsequence, by the standard parabolic estimates as applied to \eqref{new variable}, $v_n\to \bar v(a, \textbf x)$ converges locally uniformly in $[0, a_c]\times\R$, where $0\le\bar v\le1$ fulfills
		\begin{equation}\label{bar v}
			\begin{cases}
				\bar v_a=\bar v_{\textbf x\textbf x}-\mu(a, \bar x)\bar v-\bar\la\bar v, &(a, \textbf x)\in(0, a_c)\times\R,\\
				\bar v(0, \textbf x)=\int_0^{a_c}\be(a, \bar x)\bar v(a, \textbf x)da, &\textbf x\in\R.
			\end{cases}
		\end{equation}
		Observe that $\bar v(\bar a, 0)=1$. 
		
		Since $\mu$ and $\be$ in \eqref{bar v} are independent of $\textbf x$, the solution $\bar v$ of \eqref{bar v} is separable (see \cite[Section 1]{langlais1988large}), i.e. $\bar v(a, \textbf x)=p(a)\vartheta(\textbf x)$, where $\vartheta$ solves $\vartheta_{\textbf x\textbf x}=0$ in $\R$, and $p(a)$ is a positive solution of the following equation
		\begin{equation}\nonumber
			\begin{cases}
				p_a=-\mu(a, \bar x) p-\bar\la p, &a\in(0, a_c),\\
				p(0)=\int_0^{a_c}\be(a, \bar x)p(a)da.
			\end{cases}
		\end{equation}
		However, due to \eqref{>1}, the principal eigenvalue $\bar \la$ of the above eigenvalue problem satisfies $\bar\la>0$, which is a contradiction.
		
		Finally, one can handle the case $\bar x\in\{0, 1\}$ exactly as in Hess \cite[Proposition 17.3]{hess1991periodic}. Here we omit the details and the proof is thus complete.
	\end{proof}
	
	\subsection{Limiting behavior for large diffusion}
	This subsection is devoted to the  study of the asymptotic behavior of $\lambda_0(d, \La)$ as the diffusion rate $d$ approaches infinity. Note that no positivity of $q$ is required any more in this subsection. Before proceeding, we provide an upper and lower bound estimate for $\lambda_0(d, \La)$, which is independent of $d>0$.
	\begin{lemma}\label{bound} Given $d>0$ and $\La\in\R$, we have $\lambda_1\le \lambda_0(d, \La)\le\lambda^1$ , where $\lambda^1$ and $\lambda_1$ are defined through \eqref{lambda^1} and \eqref{lambda_1} below, respectively.
	\end{lemma}
	
	\begin{proof} Let $(\lambda^1, \Phi^1)$ be the principal eigenpair of the following age-structured operator:
		\begin{equation}\nonumber
			\begin{cases}
				\overline\Phi_a(a)=-(\lambda^1+\underline{\mu}(a))\overline\Phi(a),\\
				\overline\Phi(0)=\int_{0}^{a_c}\overline{\beta}(a)\overline\Phi(a)da.
			\end{cases}
		\end{equation}
		Then $\lambda^1$ satisfies
		\begin{eqnarray}\label{lambda^1}
			\int_{0}^{a_c}\overline{\beta}(a)e^{-\lambda^1a}e^{-\int_{0}^{a}\underline{\mu}(s)ds}da=1.
		\end{eqnarray} 
		It is easy to compute that 
		\begin{equation}\nonumber
			\begin{cases}
				\overline\Phi_a-d\overline\Phi_{xx}-\La q(x)\overline\Phi_x+(\mu(a, x)+\la^1)\overline\Phi\ge0, &(a, x)\in(0, a_c)\times(0, 1),\\
				\overline\Phi_x(a, 0)=\overline\Phi_x(a, 1)=0, &a\in(0, a_c),\\
				\overline\Phi(0)\ge\int_0^{a_c}\beta(a, x)\overline\Phi(a)da, &x\in(0, 1).
			\end{cases}
		\end{equation} 
		Observe that $\overline\Phi$ is a positive strict super-solution of \eqref{equ} with $\mu$ replaced by $\mu+\la^1$. It follows by Proposition \ref{equivalent} that $\lambda_0(d, \La)\le\lambda^1$.
		
		Similarly, consider the following equation with a positive solution $\underline\Phi(a)$: 
		\begin{equation}\nonumber
			\begin{cases}
				\underline\Phi(a)=-(\lambda_1+\overline{\mu}(a))\underline\Phi(a),\\
				\underline\Phi(0)=\int_{0}^{a_c}\underline{\beta}(a)\underline\Phi(a)da,
			\end{cases}
		\end{equation}
		where $\lambda_1$ satisfies 
		\begin{eqnarray}\label{lambda_1}
			\int_{0}^{a_c}\underline{\beta}(a)e^{-\lambda_1a}e^{-\int_{0}^{a}\overline{\mu}(s)ds}da=1.
		\end{eqnarray} 
		Then similar computation yields 
		\begin{equation}\nonumber
			\begin{cases}
				\underline\Phi_a-d\underline\Phi_{xx}-\La q(x)\underline\Phi_x+(\mu(a, x)+\la_1)\underline\Phi\le0, &(a, x)\in(0, a_c)\times(0, 1),\\
				\underline\Phi_x(a, 0)=\underline\Phi_x(a, 1)=0, &a\in(0, a_c),\\
				\underline\Phi(0)\le\int_0^{a_c}\beta(a, x)\underline\Phi(a)da, &x\in(0, 1),
			\end{cases}
		\end{equation}  
		which implies that $\lambda_0(d, \La)\ge\lambda_1$. Thus the conclusion is proven.
	\end{proof}
	
	Now we provide the proof of Theorem \ref{average}.    
	\begin{proof}[Proof of Theorem \ref{average}] 	
		We will employ a similar approach used in \cite[Theorem 1.3]{peng2015effects}. Let $\phi$ be the principal eigenfunction corresponding to $\lambda_0(d, \La)$. We normalize $\phi$ such that $\norm{\phi}_{L^2((0, a_c)\times(0, 1))}=1$ for all $d>0$ and set $\overline q=\max_{[0, 1]}|q(x)|$. Thus $\overline q\ge0$ is independent of $d$ and $\La$. 
		
		Recall that $(\lambda_0(d, \La),\phi)$ satisfies 
		\begin{equation}\label{dLa}
			\begin{cases}
				\phi_a=d\phi_{xx}+\Lambda q(x)\phi_x-\mu(a, x)\phi-\lambda_0(d, \La)\phi,&\quad (a, x)\in(0, a_c)\times(0, 1),\\
				\phi_x(a, 0)=\phi_x(a, 1)=0, &\quad a\in(0, a_c),\\
				\phi(0, x)=\int_{0}^{a_c}\beta(a, x)\phi(a, x)da, &\quad x\in(0, 1).
			\end{cases}
		\end{equation}
		We multiply the equation \eqref{dLa} by $\phi$ and then integrate over $(0, a_c)\times(0, 1)$ to infer
		\begin{eqnarray}
			&&\frac12\norm{\phi(a_c, \cdot)}^2_{L^2(0, 1)}-\frac12\int_0^1\left(\int_0^{a_c}\beta(a, x)\phi(a, x)da\right)^2dx\nonumber\\
			&=&-d\int_0^1\int_0^{a_c}\phi_x^2dadx+\La\int_0^1\int_0^{a_c}q(x)\phi_x\phi dadx-\int_0^1\int_0^{a_c}(\mu(a, x)+\lambda_0(d, \La))\phi^2dadx.\nonumber
		\end{eqnarray}
		By the normalization, Lemma \ref{bound} and the H\"{o}lder inequality, we then find that
		\begin{eqnarray}
			d\int_0^1\int_0^{a_c}\phi_x^2dadx&\le&\overline q|\La|\left(\int_0^1\int_0^{a_c}\phi_x^2 dadx\right)^{1/2}-\lambda_0(d, \La)-\inf_{a\in(0, a_c)}\underline\mu(a)\nonumber\\
			&&+\max_{x\in[0, 1]}\frac12\norm{\beta(\cdot, x)}_{L^2(0, a_c)}^2-\frac12\norm{\phi(a_c, \cdot)}^2_{L^2(0, a_c)}\nonumber\\
			&\le&\overline q|\La|\left(\int_0^1\int_0^{a_c}\phi_x^2 dadx\right)^{1/2}+\max_{x\in[0, 1]}\frac12\norm{\beta(\cdot, x)}_{L^2(0, a_c)}^2-\la_1-\inf_{a\in(0, a_c)}\underline\mu(a)\nonumber\\
			&\le&\overline q|\La|\left(\int_0^1\int_0^{a_c}\phi_x^2 dadx\right)^{1/2}+C,\nonumber
		\end{eqnarray}
		where the positive constant $C$ is independent of $d, \La$ and may be different from place to place. Thus, for all large $d$, we can apply the H\"{o}lder inequality to obtain
		\begin{eqnarray}\label{poincare}
			\int_0^1\int_0^{a_c}\phi_x^2dadx\le C\frac{\La^2+d}{d^2}.
		\end{eqnarray}
		
		Define
		$$
		\phi_*(a)=\int_0^1\phi(a, x)dx\ \, \text{ and }\ \ \phi^*(a, x)=\phi(a, x)-\phi_*(a).
		$$
		It is not hard to see that 
		$$
		\int_0^1\phi^*(a, x)dx=0\ \ \ \text{ for all }a\in[0, a_c].
		$$
		Thus, by the well-known Poincar\'{e} inequality, one can conclude that
		$$
		\int_0^1(\phi^*)^2dx\le C\int_0^1(\phi_x^*)^2dx\ \ \ \text{ for all }a\in[0, a_c].
		$$
		Since $\phi^*_x=\phi_x$, because of \eqref{poincare}, we infer
		\begin{eqnarray}\label{poin}
			\int_0^1\int_0^{a_c}(\phi^*_x)^2dadx\le C\frac{\La^2+d}{d^2}.
		\end{eqnarray}
		This, combined with the H\"{o}lder inequality and \eqref{poincare}, implies that
		\begin{eqnarray}\label{u_x}
			\int_0^1\int_0^{a_c}|\phi^*|dadx\le C\left(\frac{\La^2+d}{d^2}\right)^{1/2},\ \ \; \int_0^1\int_0^{a_c}|\phi_x|dadx\le C\left(\frac{\La^2+d}{d^2}\right)^{1/2}.
		\end{eqnarray}
		Moreover, we integrate the equations of \eqref{dLa} over $(0, 1)$ to see that
		\begin{equation}\label{ave}
			\frac{d\phi_*}{da}=-\phi_*\int_0^1[\lambda_0(d, \La)+\mu(a, x)]dx+\int_0^1[\La q(x)\phi_x-\mu(a, x)\phi^*]dx.
		\end{equation}
		In view of \eqref{u_x}, there holds
		$$
		\int_0^{a_c}\left|\int_0^1\La q(x)\phi_x-\mu(a, x)\phi^*dx\right|da\to 0\  \text{ as }d\to\infty.
		$$
		Hence, solving the ordinary differential equation \eqref{ave} results in
		\begin{equation}\label{sol}
			\phi_*(a)=\phi_*(0)e^{-\int_0^a\int_0^1(\lambda_0(d, \La)+\mu(s, x))dxds}+\tau(a), 
		\end{equation}
		where 
		$$
		\tau(a)\to0\ \ \text{ uniformly on }[0, a_c]\ \  \text{ as }d\to\infty.
		$$
		Now plugging \eqref{sol} into the integral initial condition yields
		\begin{eqnarray}
			\phi_*(0)&=&\int_0^1\int_{0}^{a_c}\beta(a, x)\phi(a, x)dadx\nonumber\\
			&=&\int_0^1\int_{0}^{a_c}\beta(a, x)\phi_*(a)dadx+\int_0^1\int_{0}^{a_c}\beta(a, x)(\phi(a, x)-\phi_*(a))dadx\nonumber\\
			&=&\int_0^1\int_{0}^{a_c}\beta(a, x)\phi_*(0)e^{-\int_{0}^{a}\int_0^1(\lambda_0(d, \La)+\mu(s, x))dxds}dadx+\tau(a)\int_0^1\int_{0}^{a_c}\beta(a, x)dadx.\nonumber
		\end{eqnarray} 
		
		We then claim that $\phi_*(0)$ does not approach zero as $d\to\infty$. Otherwise, suppose that $\phi_*(0)\to0$ as $d\to\infty$ (up to a sequence of $d$), then \eqref{sol} implies that $\phi_*(a)\to0$ uniformly in $[0, a_c]$. This, together with \eqref{poin}, indicates
		$$
		\int_0^1\int_0^{a_c}\phi^2(a, x)dadx\to0\ \ \text{ as }d\to\infty.
		$$
		which is a contradiction with the normalization  $\norm{\phi}_{L^2((0, a_c)\times(0, 1))}=1$.
		
		Since $\beta$ is also bounded, after cancellation of $\phi_*(0)$ in both sides, we have 
		$$ 
		1=\int_0^1\int_{0}^{a_c}\beta(a, x)e^{-\lambda_0(d, \La)a}e^{-\int_{0}^{a}\int_0^1\mu(s, x)dxds}dadx+\frac{\tau(a)}{\phi_*(0)}\int_0^1\int_0^{a_c}\beta(a, x)dadx.
		$$
		By letting $d\to\infty$, we obtain
		$$ 
		\int_0^1\int_{0}^{a_c}\beta(a, x)e^{-a\lim_{d\to\infty}\lambda_0(d, \La)}e^{-\int_{0}^{a}\int_0^1\mu(s, x)dxds}dadx=1, 
		$$
		yielding $\lim_{d\to\infty}\lambda_0(d, \La)=\bar\al$, where $\bar\al$ is defined in \eqref{baral}.
		This completes the proof.
	\end{proof}
	
	\section{Analysis on the model \eqref{nonlinear}}\label{application} 
	In this section, we concentrate on analyzing the model \eqref{nonlinear} and provide detailed proofs for the main theorems presented in Subsection \ref{1.1}.
	
	\subsection{Global dynamics of \eqref{nonlinear}} This subsection is devoted to the proofs of Theorem \ref{stability MR} and Proposition \ref{asymptotic}.
	
	\begin{proof}[Proof of Theorem \ref{stability MR}] Keeping Remark \ref{remark-2.1} in mind, the existence, uniqueness, and global stability of the positive equilibrium \( u^* \) can be established by slightly modifying the arguments in Ducrot et al. \cite[Section 4]{Ducrot2022Age-structuredII}. We omit the details here.
		
		Next, we examine the global stability of the zero  equilibrium when \(\lambda_0 \le 0\). Following the approach used in the proof of Theorem 4.4 in Ducrot et al. \cite[Section 4]{Ducrot2022Age-structuredII}, we can select a sufficiently large \( L \) such that the initial condition \( u_0 \leq L e^{(\Lambda/d)x} \)  and \( L e^{(\Lambda/d)x} \)  is a super-solution to \eqref{nonlinear}. Let \(\overline{U}\) be the solution of \eqref{nonlinear} that starts from the initial condition \( L e^{(\Lambda/d)x} \). It follows that for any $a_+\in[a_c, a_m)$,
		\begin{equation}\label{u-comp}
			0 \le u(t, a, x) \le \overline{U}(t, a, x) \le L e^{(\Lambda/d)x} \quad \text{for all } (a, x) \in [0, a_+] \times [0, 1],\, t > 0.
		\end{equation}
		Note that \(\overline{U}\) is non-increasing in \( t \) for each \((a, x) \in [0, a_+] \times [0, 1]\). As \( t \to +\infty \), \(\overline{U}\) converges in \( C([0, a_+] \times [0, 1]) \) to a nonnegative solution of the stationary problem \eqref{logistic}. According to Proposition 4.7 in Ducrot et al. \cite[Section 4]{Ducrot2022Age-structuredII}, this implies that \(\overline{U} \to 0\) in \( C([0, a_+] \times [0, 1]) \) as \( t \to +\infty \). Consequently, by \eqref{u-comp}, we conclude that \( u \to 0 \) in \( C([0, a_+] \times [0, 1]) \) as \( t \to +\infty \). This completes the proof of Theorem \ref{stability MR}.
	\end{proof}  
	
	\begin{proof}[Proof of Proposition \ref{asymptotic}]
		According to Theorem \ref{Lalambda} and Theorem \ref{Dlambda}, \(\lambda_0 > 0\) for all fixed \(d > 0\) when \(\Lambda > 0\) is sufficiently large or for fixed \(\Lambda > 0\) when \(d > 0\) is sufficiently small, provided that condition \eqref{alpha} holds. Conversely, if we reverse condition \eqref{alpha}, i.e.,
		\[
		f_u(1, 0)\int_{0}^{a_c}\beta(a, 1)e^{-\int_0^a\mu(s, 1)ds}da < 1,
		\]
		then \(\lambda_0 < 0\) for all fixed \(d > 0\) and \(\Lambda > 0\) sufficiently large or for fixed \(\Lambda > 0\) and \(d > 0\) sufficiently small.
		
		When \(\Lambda = 0\), Theorem \ref{Dlambda} implies that \(\lambda_0 > 0\) for all \(d > 0\) sufficiently small, provided that condition \eqref{al max} holds. Conversely, \(\lambda_0 < 0\) for all \(d > 0\) sufficiently small if the condition \eqref{al max} is reversed.
		
		Finally, for fixed \(\Lambda \ge 0\), Theorem \ref{average} indicates that \(\lambda_0 > 0\) for all \(d > 0\) sufficiently large, provided that \eqref{al average} holds. Conversely, \(\lambda_0 < 0\) for all large \(d > 0\) if the condition \eqref{al average} is reversed.
		
		In light of the above analysis, Proposition \ref{asymptotic} follows directly from Theorem \ref{stability MR}.
	\end{proof}
	
	\subsection{Asymptotic profile for large advection or small diffusion}
	This subsection addresses the asymptotic profile of the positive equilibrium as the advection rate approaches infinity or the diffusion rate approaches zero when advection is present. We will prove Theorem \(\ref{AP LA MR}\).
	
	We start with two lemmas which will provide some upper estimates on the integral $\int_0^au^*(s, x)ds$ for any $(a, x)\in[0, a_m)\times[0, 1]$
	of the positive equilibrium $u^*$ of \eqref{nonlinear}. The first lemma reads as follows.
	\begin{lemma}\label{estimate of u^*} Let $u^*$ be the positive equilibrium  of \eqref{nonlinear}. For any $d>0$ and $\La>0$, we have 
		\begin{equation}\nonumber
			\int_0^au^*(s, x)ds\le\frac{\La La_++M_1\left[1-\frac{d}{\La}(1-e^{-\frac{\La}{d}})\right]}{d(1-e^{-\frac{\La}{d}})}e^{-\frac{\La}{d}(1-x)}+\frac{L}{\La}\left(1-e^{-\frac{\La}{d}(1-x)}\right), \;\forall(a, x)\in[0, a_+]\times[0, 1],
		\end{equation}
		where $M_1:=a_+L\norm{\overline\mu}_{L^\infty(0, a_+)}+2L$ and $a_+\in[a_c, a_m)$.
	\end{lemma}
	
	\begin{proof} Integrating the first equation of \eqref{logistic} over $(0, a)\times[0, 1]$, we obtain
		$$
		\int_0^1(u^*(a, x)-u^*(0, x))dx=-\int_0^1\int_0^a\mu(s, x)u^*(s, x)dsdx.
		$$
		It follows from $u^*>0$ and Assumption \ref{assump-f}-(iv) that 
		\begin{equation}\nonumber
			\int_0^1u^*(a, x)dx\le \int_0^1u^*(0, x)dx\le\int_0^1f\left(x, \int_0^{a_m}\beta(a, x)u^*(a, x)da\right)dx\le L, \;\forall a\in[0, a_+].
		\end{equation}
		Now integrating the first equation of \eqref{logistic} over $(0, a)\times(0, x)$, we obtain
		\begin{eqnarray}
			\int_0^a\left(du^*_x(s, x)-\La u^*(s, x)\right)ds&=&\int_0^a\int_0^x\mu(s, y)u^*(s, y)dyds+\int_0^x(u^*(a, y)-u^*(0, y))dy\nonumber\\
			&\le& a_+L\norm{\overline\mu}_{L^\infty(0, a_+)}+2L=M_1, \;\forall x\in[0, 1].\label{du_x}
		\end{eqnarray}
		We multiply both sides of the above inequality by $e^{-\frac{\La x}{d}}$ and then integrate the resulting inequality over $(x, 1)$ to see
		\begin{eqnarray}\nonumber
			\int_0^au^*(s, 1)ds\;e^{-\frac{\La}{d}(1-x)}-\frac{M_1}{\La}\left(1-e^{-\frac{\La}{d}(1-x)}\right)\le \int_0^au^*(s, x)ds, \;\forall (a, x)\in[0, a_+]\times[0, 1].
		\end{eqnarray}
		Upon an integration of the above inequality from $0$ to $1$, we find
		\begin{eqnarray}\label{u(a, 1)}
			\int_0^au^*(s, 1)ds\le\frac{\La La_++M_1\left[1-\frac{d}{\La}(1-e^{-\frac{\La}{d}})\right]}{d(1-e^{-\frac{\La}{d}})}.
		\end{eqnarray}
		On the other hand, we can see from \eqref{du_x} that 
		$$
		\int_0^a\left(du^*_x(s, x)ds-\La u^*(s, x)\right)ds\ge-\int_0^xu^*(0, y)dy\ge-L,\ \ x\in[0,1].
		$$
		Multiplying both sides of the above inequality by $e^{-\frac{\La x}{d}}$ and then integrating over $(x, 1)$, we are led to
		\begin{eqnarray}\label{u(a, x)}
			\int_0^au^*(s, x)ds\le \int_0^au^*(s, 1)ds\; e^{-\frac{\La}{d}(1-x)}+\frac{L}{\La}\left(1-e^{-\frac{\La}{d}(1-x)}\right), \;\forall (a, x)\in[0, a_+]\times[0, 1].
		\end{eqnarray}
		We have obtained the desired inequality of $u^*$ by combining \eqref{u(a, 1)} and \eqref{u(a, x)}. 
	\end{proof}
	
	The second lemma reads as follows.
	\begin{lemma}\label{estimate of u^*-2} Let $u^*$ be the positive equilibrium  of \eqref{nonlinear}. For $\La>0$ and small $d>0$, we have 
		\begin{equation}\label{u_d-a}
			\int_0^{a}u^*(s, x)ds\le \frac{\La La_++M_1\left[1-\frac{d}{\La}(1-e^{-\frac{\La}{d}})\right]}{d(1-e^{-\frac{\La}{d}})} \; e^{-\frac{\La}{d}\left(1-\frac{M_2d}{\La^2}\right)(1-x)},\ \ \forall(a, x)\in[0, a_+]\times[0, 1],
		\end{equation}
		where $M_1$ is given in Lemma \ref{estimate of u^*}, $M_2:=2\norm{\overline\be}_{L^\infty(0, a_m)}\max_{[0, 1]}f_u(x, 0)$ and $a_+\in[a_c, a_m)$.
	\end{lemma}
	
	\begin{proof} We set 
		$$
		w(a, x):=e^{-\frac{\La}{d}\left(1-\frac{M_2d}{\La^2}\right)x}u^*(a, x).
		$$
		Then $w$ satisfies
		\begin{equation}\label{w}
			\begin{cases}        		w_a=dw_{xx}+\La\left(1-2\frac{M_2d}{\La^2}\right)w_x+\left[M_2\left(\frac{M_2d}{\La^2}-1\right)-\mu(a, x)\right]w, &(a, x)\in(0, a_m)\times(0, 1),\\
				w(0, x)=e^{-\frac{\La}{d}\left(1-\frac{M_2d}{\La^2}\right)x}f\left(x, e^{\frac{\La}{d}\left(1-\frac{M_2d}{\La^2}\right)x}\int_{0}^{a_m}\beta(a, x)w(a, x)da\right), &x\in(0, 1),\\
				w_x(a, x)=\frac{M_2}{\La}w(a, x), &(a, x)\in(0, a_m)\times\{0, 1\}.
			\end{cases}
		\end{equation}
		
		For any $a_+\in[a_c, a_m)$, define 
		$$
		V(x)=\int_0^{a_+}w(s, x)ds, \ \ x\in[0,1].
		$$ 
		By integrating the equation \eqref{w} over $[0, a_+]$, we infer
		\begin{equation}\nonumber
			\begin{cases}
				dV_{xx}+\La\left(1-2\frac{M_2d}{\La^2}\right)V_x+M_2\left(\frac{M_2d}{\La^2}-1\right)V-w(a_+, x)\\
				\;+e^{-\frac{\La}{d}\left(1-\frac{M_2d}{\La^2}\right)x}f\left(x, e^{\frac{\La}{d}\left(1-\frac{M_2d}{\La^2}\right)x}\int_{0}^{a_m}\beta(a, x)w(a, x)da\right)-\int_0^{a_+}\mu(s, x)w(s, x)ds=0, &x\in(0, 1),\\
				V_x(x)=\frac{M_2}{\La}V(x), &x\in\{0, 1\}.
			\end{cases}
		\end{equation}
		Let $x_*\in[0, 1]$ be such that $V(x_*)=\max_{x\in[0, 1]}V(x)$. Since $V_x(0)=\frac{M_2}{\La}V(0)>0$, then $x_*\neq0$. If $x_*\in(0, 1)$, then $V_x(x_*)=0$ and $V_{xx}(x_*)\le0$. By the equation of $V$, we have
		\begin{eqnarray}\label{ge0}
			&&M_2\left(\frac{M_2d}{\La^2}-1\right)V(x_*)+e^{-\frac{\La}{d}\left(1-\frac{M_2d}{\La^2}\right)x_*}f\left(x, e^{\frac{\La}{d}\left(1-\frac{M_2d}{\La^2}\right)x_*}\int_{0}^{a_m}\beta(a, x_*)w(a, x_*)da\right)\nonumber\\
			&&-w(a_+, x_*)-\int_0^{a_+}\mu(s, x_*)w(s, x_*)ds\ge0.
		\end{eqnarray}
		On the other hand, by the definition of $w$ and Assumption \ref{assump-f}-(iii),  for all $ x\in[0,1] $, there holds
		\begin{eqnarray}
			&&M_2\left(\frac{M_2d}{\La^2}-1\right)V+e^{-\frac{\La}{d}\left(1-\frac{M_2d}{\La^2}\right)x}f\left(x, e^{\frac{\La}{d}\left(1-\frac{M_2d}{\La^2}\right)x}\int_{0}^{a_m}\beta(a, x)w(a, x)da\right)\nonumber\\
			&&-w(a_+, x)-\int_0^{a_+}\mu(s, x)w(s, x)ds\nonumber\\
			&\le&M_2\left(\frac{M_2d}{\La^2}-1\right)V+e^{-\frac{\La}{d}\left(1-\frac{M_2d}{\La^2}\right)x}f\left(x, e^{\frac{\La}{d}\left(1-\frac{M_2d}{\La^2}\right)x}\int_{0}^{a_m}\beta(a, x)w(a, x)da\right)\nonumber\\
			&\le&M_2\left(\frac{M_2d}{\La^2}-1\right)V+\norm{\overline\be}_{L^\infty(0, a_m)}\max_{x\in[0, 1]}f_u(x, 0)V\nonumber\\
			&<&0\nonumber
		\end{eqnarray}
		provided that
		\begin{equation}\label{d-small}
			0<d<\frac{\La^2}{4\norm{\overline\be}_{L^\infty(0, a_m)}\max_{x\in[0, 1]}f_u(x, 0)}.
		\end{equation}
		This is a contradiction against \eqref{ge0} once \eqref{d-small} holds. 
		
		Thus we must have $x_*=1$, i.e., $V(x)\le V(1)$ for all $x\in[0, 1]$. Hence, if \eqref{d-small} holds, by the definitions of $w$ and $V$, one has
		$$
		\int_0^{a_+}u^*(s, x)ds\le \int_0^{a_+}u^*(s, 1)ds \; e^{-\frac{\La}{d}\left(1-\frac{M_2d}{\La^2}\right)(1-x)}, \quad\forall x\in[0, 1].
		$$
		This, together with \eqref{u(a, 1)}, allows one to deduce
		\begin{equation}\nonumber
			\int_0^{a_+}u^*(s, x)ds\le \frac{\La La_++M_1\left[1-\frac{d}{\La}(1-e^{-\frac{\La}{d}})\right]}{d(1-e^{-\frac{\La}{d}})} \; e^{-\frac{\La}{d}\left(1-\frac{M_2d}{\La^2}\right)(1-x)}, \quad\forall x\in[0, 1],
		\end{equation}
		which obviously implies \eqref{u_d-a}.
	\end{proof}
	
	With the aid of Lemma \ref{estimate of u^*} and Lemma \ref{estimate of u^*-2}, we are now ready to present the proof of Theorem \ref{AP LA MR}.
	
	\begin{proof}[Proof of Theorem \ref{AP LA MR}] We begin by verifying Theorem \ref{AP LA MR}-(i). From Lemma \ref{estimate of u^*}, we observe that for any \(a \in [0, a_m)\),
		\begin{equation}\label{integral unif converg to 0}
			\int_0^au^*_{\La,\,d}(s, \cdot)ds\to0\ \ \text{ locally uniformly on } [0, 1)\ \ \text{ as } \La\to\infty.
		\end{equation}
		Since
		\[
		u^*_{\Lambda,\,d}(0, x) = f\left(x, \int_{0}^{a_m} \beta(a, x) u^*_{\Lambda,\,d}(a, x) \, da\right), \quad x \in (0, 1),
		\]
		it follows from Assumption \ref{assump-f}-(iii), Assumption \ref{Ass}-(i)(ii), and \eqref{integral unif converg to 0} that \(u^*_{\Lambda,\,d}(0, x) \to 0\) locally uniformly in \([0, 1)\) as \(\Lambda \to \infty\). Moreover, by Assumption \ref{assump-f}-(iv), we have \(u^*_\Lambda(0, x) \le L\) for all \(x \in [0, 1]\) and \(\Lambda > 0\). Thus, the Lebesgue dominated convergence theorem ensures that
		\[
		\int_0^1 u^*_{\Lambda,\,d}(0, x) \, dx \to 0 \quad \text{as } \Lambda \to \infty.
		\]
		
		On the other hand, integrating the first equation of \eqref{logistic} over \([0, 1]\) with respect to \(x\), we obtain
		\[
		\frac{d}{da} \int_0^1 u^*_{\Lambda,\,d}(a, x) \, dx = -\int_0^1 \mu(a, x) u^*_{\Lambda,\,d}(a, x) \, dx < 0, \quad \forall a \in [0, a_m).
		\]
		This implies that \(\int_0^1 u^*_{\Lambda,\,d}(a, x) \, dx\) is decreasing with respect to \(a \in [0, a_m)\). Therefore, for any \(a \in [0, a_m)\),
		\[
		\int_0^1 u^*_{\Lambda,\,d}(a, x) \, dx \to 0 \quad \text{as } \Lambda \to \infty,
		\]
		which proves Theorem \ref{AP LA MR}-(i).
		
		Theorem \ref{AP LA MR}-(ii) can be verified similarly. From Lemma \ref{estimate of u^*-2}, we conclude that for any \(a \in [0, a_m)\),
		\[
		\int_0^a u^*_{\Lambda,\,d}(s, \cdot) \, ds \to 0 \quad \text{locally uniformly in } [0, 1)\ \ \text{ as } d \to 0.
		\]
		Using this fact, we can apply the same reasoning as above to establish Theorem \ref{AP LA MR}-(ii). Thus, the proof is complete.
	\end{proof}
	
	\subsection{Asymptotic profile for small diffusion without advection}
	In this subsection, we focus on the asymptotic profile of \( u^* \) with small diffusion in the absence of advection and provide the proof of Theorem \ref{AP LA WA}. This proof is inspired by the work of L{\'o}pez-G{\'o}mez and his coauthors on singular perturbation results for periodic-parabolic problems with mixed boundary conditions; see \cite{cano2024singular,fernandez2019singular,lopez2013linear}. Here we extend their idea to the age-structured problem \eqref{nonlinear}.
	
	To this end, let $\Ga$ and $h$ be defined as in \eqref{Q} and \eqref{h(x)}, choose any $a_+\in[a_c, a_m)$ and make the following transformations:
	$$
	w^*_d(a, x)=\frac{u^*_d(a, x)}{h(x)}\ \ \text{ and }\ r^*(a, x)=\frac{v^*(a, x)}{h(x)}, \quad (a, x)\in[0, a_+]\times[0, 1],
	$$
	where $v^*$ for any $x\in[0, 1]$ is the unique solution of \eqref{solutionofQ MR} and $u^*_d$ is the unique positive equilibrium of \eqref{nonlinear} with $\La=0$ satisfying
	\begin{equation}\nonumber
		\begin{cases}
			u_a=du_{xx}-\mu(a, x)u, &(a, x)\in(0, a_m)\times(0, 1),\\
			u(0, x)=f\left(x, \int_{0}^{a_m}\beta(a, x)u(a, x)da\right), &x\in(0, 1),\\
			u_x(a, x)=0, &(a, x)\in(0, a_m)\times\{0, 1\}.
		\end{cases}
	\end{equation}
	Recall from Ducrot et al. \cite[Lemmas 5.3 and 5.4]{Ducrot2022Age-structuredII} that (i) $v^*\in C^2(I, W^{1, 1}(0, a_+))$ and $v>0$ in $I$ where $\Ga(x)>1$ for all $x\in I$; (ii) while $v^*\equiv0$ in $I$ where $\Ga(x)\le1$ for all $x\in I$. Then it is apparent that $w^*_d\in W^{1, 1}((0, a_+), W^{2, p}(0, 1))$ with $p>1$ and $r^*\in C^2(I, W^{1, 1}(0, a_+))$, with $w^*_d>0$ and $r^*>0$ in $[0, a_+]\times I$ where $\Ga(x)>1$ for all $x\in I$, which satisfy, respectively, 
	\begin{equation}\label{logistic2}
		\begin{cases}
			w_a=dw_{xx}+\frac{2d h'(x)}{h(x)}w_x+\frac{dh''(x)}{h(x)}w-\mu(a, x)w, &(a, x)\in(0, a_m)\times(0, 1),\\
			w(0, x)=f\left(x, h(x)\int_{0}^{a_m}\beta(a, x)w(a, x)da\right)/h(x), &x\in(0, 1),\\
			w_x(a, x)+\varrho\psi'(x)w(a, x)=0, &(a, x)\in(0, a_m)\times\{0, 1\},
		\end{cases}
	\end{equation} 
	and for each $x\in[0, 1]$,
	\begin{equation}\nonumber
		\begin{cases}
			r_a(a, x)=-\mu(a, x)r(a, x), \quad a\in(0, a_+),\\
			r(0, x)=f\left(x, h(x)\int_{0}^{a_m}\beta(a, x)r(a, x)da\right)/h(x).
		\end{cases}
	\end{equation}
	
	To achieve Theorem \ref{AP LA WA}, it becomes equivalent to show: for any $0<\de, \ep\ll1$, $a_+\in[a_c, a_m)$ and any compact subset $I$ of $(0, 1)$ with $\Ga(x)>1$ for all $x\in I$, there holds
	$$
	r^*(a, x)-\ep\le \liminf\limits_{d\to0}w^*_d(a, x)\le\limsup\limits_{d\to0}w^*_d(a, x)da\le (1+\de)r^*(a, x), \quad \forall(a, x)\in[0, a_+]\times I.
	$$
	
	As a first step, we are going to show
	\begin{proposition}\label{auto}
		Assume $\Gamma(x)>1$ in $[0, 1]$, then for any $a_+\in[a_c, a_m)$ and $0<\de\ll1$, there holds
		\begin{eqnarray}\nonumber
			\limsup\limits_{d\to0}w^*_d(a, x)\le (1+\de)r^*(a, x), \quad \forall (a, x)\in [0, a_+]\times[0, 1].
		\end{eqnarray}
	\end{proposition}
	\begin{proof}
		Define $\overline r_\de(a, x):=(1+\de)r^*(a, x)$ and then $\overline r_\de>0$.  
		By choosing 
		$$
		\varrho>\max\left\{0,\ \ \max_{a\in[0, a_+]}\frac{-r^*_x(a, 1)}{r^*(a, 1)}\frac{1}{\psi'(1)}, \; \max_{a\in[0, a_+]}\frac{r^*_x(a, 0)}{r^*(a, 0)}\frac{1}{-\psi'(0)}\right\},
		$$
		one can check that 
		\begin{equation}\label{r derivative}
			r^*_x(a, 1)+\varrho r^*(a, 1)\psi'(1)\ge0 \text{ and }r^*_x(a, 0)+\varrho r^*(a, 0)\psi'(0)\le0, \quad \forall a\in[0, a_+].
		\end{equation}
		Observe that due to \eqref{r derivative}, $\overline r_\de(a, x)$ satisfies
		\begin{equation}\nonumber
			\begin{cases}
				(\overline r_\de)_a=-\mu(a, x)\overline r_\de, &(a, x)\in(0, a_+)\times(0, 1),\\
				\overline r_\de(0, x)=(1+\de)f\left(x, h(x)\int_{0}^{a_m}\beta(a, x)r^*(a, x)da\right)/h(x), &x\in(0, 1),\\
				(\overline r_\de)_x(a, 1)+\varrho\overline r_\de(a, 1)\psi'(1)\ge0, &a\in(0, a_+),\\
				(\overline r_\de)_x(a, 0)+\varrho\overline r_\de(a, 0)\psi'(0)\le0, &a\in(0, a_+).
			\end{cases}
		\end{equation}
		Recalling the operator $L$ defined in \eqref{L}, one has $\norm{dL\overline r_\de}_{L^\infty((0, a_+)\times(0, 1))}\to0$ as $d\to0$. Thus for any $\ep>0$, one can choose $d$ sufficiently small such that 
		$$
		0<d\le\frac{\ep \min_{[0, a_+]\times[0, 1]}\overline r_\de(a, x)}{\norm{L\overline r_\de}_{L^\infty((0, a_+)\times(0, 1))}},
		$$ 
		and so $\overline r_\de$ satisfies
		$$
		(\overline r_\de)_a\ge d(\overline r_\de)_{xx}+\frac{2d h'(x)}{h(x)}(\overline r_\de)_x+\frac{dh''(x)}{h(x)}\overline r_\de-\mu(a, x)\overline r_\de-\ep \overline r_\de.
		$$
		On the other hand, the initial data $\overline r_\de(0, x)$ satisfies
		\begin{eqnarray}
			\overline r_\de(0, x)&=&(1+\de)f\left(x, h(x)\int_{0}^{a_m}\beta(a, x)r^*(a, x)da\right)/h(x)\nonumber\\
			&>&f\left(x, h(x)\int_{0}^{a_m}\beta(a, x)(1+\de)r^*(a, x)da\right)/h(x),\nonumber\\
			&=&f\left(x, h(x)\int_{0}^{a_m}\beta(a, x)\overline r_\de(a, x)da\right)/h(x),\nonumber
		\end{eqnarray}
		where we used the Assumption \ref{assump-f}-(iii) for the strict inequality. Thus by letting $\ep\to0$, the comparison principle as applied to \eqref{logistic2} restricted on $[0, a_+]\times[0, 1]$ yields 
		$$
		\limsup\limits_{d\to0}w^*_d(a, x)\le \overline r_\de(a, x), \quad \forall (a, x)\in [0, a_+]\times[0, 1]. 
		$$
		Now the proof is complete.
	\end{proof}
	Proposition \ref{auto} implies that if $\Ga(x)>1$ in $[0, 1]$, then
	\begin{eqnarray}\label{sub f}
		\limsup\limits_{d\to0}u^*_d(a, x)\le (1+\de) v^*(a, x), \; \forall(a, x)\in[0, a_+]\times [0, 1]. 
	\end{eqnarray}
	Next, we will show the reverse inequality, i.e. 
	\begin{eqnarray}\nonumber
		\liminf\limits_{d\to0}u^*_d(a, x)\ge v^*(a, x)-\ep, \; (a, x)\in[0, a_+]\times I\ \, \text{ for any $I\subset(0, 1)$, where $\Ga(x)>1$ in $I$.} 
	\end{eqnarray}
	
	For the sake of clarity,  we first consider the homogeneous case where $\mu=\mu(a), \be=\be(a)$ and $f=f(u)$, and then handle the general case. Thus, we divide our analysis into two subsections.
	
	\subsubsection{The homogeneous case}
	In this subsection, we choose $a_+\in[a_c, a_m)$ and work on the homogeneous case: $\mu=\mu(a), \be=\be(a)$ and $f=f(u)$ with $\be\equiv0$ in $[a_c, a_m)$. Consider the following two equations:
	\begin{equation}\label{v(a)}
		\begin{cases}
			v_a=-\mu(a)v, &a\in(0, a_+),\\
			v(0)=f(\int_0^{a_+}\be(a)v(a)da),
		\end{cases}
	\end{equation}
	and
	\begin{equation}\label{Dirichlet}
		\begin{cases}
			w_a=d w_{xx}-\mu(a)w, &(a, x)\in(0, a_m)\times(0, 1),\\
			w(a, x)=0, &(a, x)\in(0, a_m)\times\{0, 1\},\\
			w(0, x)=f\left(\int_{0}^{a_m}\beta(a)w(a, x)da\right), &x\in(0, 1).
		\end{cases}
	\end{equation}
	Then we have the following result.
	\begin{proposition}\label{Homo Dirichlet}
		Assume that
		\begin{equation}\label{homo>1}
			f'(0)\int_0^{a_c}\be(a)e^{-\int_0^a\mu(s)ds}da>1,
		\end{equation} 
		then the solutions of \eqref{v(a)} and \eqref{Dirichlet} with $d>0$ sufficiently small, denoted by $v^*$ and $\widetilde w^*_d$ respectively, exist and are unique. Moreover, for any $0<\ep\ll1$, any $a_+\in[a_c, a_m)$ and any compact subset $I$ of $(0, 1)$, there holds that
		\begin{eqnarray}\label{gg Dirichlet}
			\liminf\limits_{d\to0}\widetilde w^*_d(a, x)\ge v^*(a)-\ep, \quad \forall (a, x)\in [0, a_+]\times I.
		\end{eqnarray}
	\end{proposition}
	\begin{proof}
		First note that the existence and uniqueness of positive solution of \eqref{v(a)} and \eqref{Dirichlet} for $d>0$ sufficiently small are guaranteed by \eqref{homo>1} and Kang \cite[Theorem 4.2]{kang2022effects}. Choose $x_0\in(0, 1)$ arbitrary and take $R>0$ such that $J:=(x_0-R, x_0+R)\subset(0, 1)$ and $\overline J=[x_0-R, x_0+R]\subset (0, 1)$. 
		
		Denote by $(\sigma_1, \varphi_1)$ the principal eigenpair of \begin{equation}\label{Diri}
			\begin{cases}\nonumber
				\varphi_{xx}=\sig \varphi, &\text{ in }J,\\
				\varphi=0, &\text{ on }\partial J:=\{x_0-R, x_0+R\},
			\end{cases}
		\end{equation} 
		Normalize the eigenfunction by $\norm{\varphi_1}_{\infty}:=\max_{x\in[x_0-R, x_0+R]}\varphi_1(x)=1$. Note that $\sig_1<0$. Then one can verify that $\psi:=\de v^*\varphi_1$ with $\de\in(0, 1)$ is a positive strict sub-solution of 
		\begin{equation}\nonumber
			\begin{cases}
				w_a=d w_{xx}-\mu(a)w+\ep w, &(a, x)\in(0, a_+)\times J,\\
				w(a, x)=0, &(a, x)\in(0, a_+)\times\partial J,\\
				w(0, x)=f\left(\int_{0}^{a_+}\beta(a)w(a, x)da\right), &x\in J,
			\end{cases}
		\end{equation}
		provided $d>0$ sufficiently small. Indeed, if $d<-\ep/\sig_1$, then 
		\begin{eqnarray}\nonumber
			\begin{cases}
				\psi_a-d\psi_{xx}+\mu(a)\psi-\ep\psi=-d\sig_1\psi-\ep\psi<0, &(a, x)\in(0, a_+)\times J,\\
				\psi(a, x)=0, &(a, x)\in(0, a_+)\times\partial J,\\
				\psi(0, x)=\de\varphi_1f\left(\int_0^{a_+}\be(a)v^*(a)da\right)\le f\left(\int_{0}^{a_+}\beta(a)\psi(a, x)da\right), &x\in J.
			\end{cases}
		\end{eqnarray} 
		Thus the comparison principle applied to \eqref{Dirichlet} restricted on $[0, a_+]\times[x_0-R, x_0+R]$ gives that $\liminf\limits_{d\to0}\widetilde w_d^*(a, x)\ge\de v^*(a)\varphi_1(x)$ on $[0, a_+]\times[x_0-R, x_0+R]$. 
		
		Since $\varphi_1$ attains its maximal value at $x_0$, for any $\ep>0$, one can choose $R_1$ sufficiently small and $\de$ sufficiently close to $1$ such that
		$$
		\psi=\de v^*\varphi_1\ge v^*-\ep\ \;\text{ on }[0, a_+]\times[x_0-R_1, x_0+R_1].
		$$
		It follows that $\liminf\limits_{d\to0}\widetilde w_d^*(a, x)\ge v^*(a)-\ep$ on $[0, a_+]\times[x_0-R_1, x_0+R_1]$.
		
		Next we will prove \eqref{gg Dirichlet} in $[0, a_+]\times I$ for any compact subset $I$ of $(0, 1)$. According to the previous arguments, for every $x\in I$, there exists $R_1(x)>0$ such that $(x-R_1(x), x+R_1(x))\subset(0, 1)$ and 
		\begin{eqnarray}\label{d_i Diri}
			\liminf_{d\to0}\tilde w^*_d(a, x)\ge v^*(a)-\ep\ \;\text{ on }[0, a_+]\times[x-R_1, x+R_1].
		\end{eqnarray}
		By the compactness of $I$, there are an integer $m\ge1$ and $m$ points $x_i\in I, i\in\{1, 2,\dots, m\}$ such that $I\subset \cup_{i=1}^m(x_i-R_1(x_i), x_i+R_1(x_i))\subset(0, 1)$. Thus by sending these $d=d(x_i)$ in \eqref{d_i Diri} to $0$, one has
		$$
		\liminf_{d\to0}\tilde w^*_d(a, x)\ge v^*(a)-\ep \;\text{ on }[0, a_+]\times I.
		$$ 
		It follows that \eqref{gg Dirichlet} holds.
	\end{proof}

	\subsubsection{The general case}
	In this subsection, we continue to choose any $a_+\in[a_c, a_m)$ and consider the general case and finish the proof of Theorem \ref{AP LA WA}. Let $\Ga$ be defined as in \eqref{Q}.
	\begin{proposition}\label{General}
		Let $I$ be a compact subset of $(0, 1)$ and satisfy $\Ga(x)>1$ for all $x\in I$. Then for any $0<\ep\ll1$ and any $a_+\in[a_c, a_m)$, there holds
		\begin{eqnarray}\label{gg general}
			\liminf\limits_{d\to0}u^*_d(a, x)\ge v^*(a, x)-\ep, \quad \forall (a, x)\in [0, a_+]\times I.
		\end{eqnarray}
	\end{proposition}
	\begin{proof} 	Pick any $\ep>0$. We first claim that for any $x_0\in I$ with $\Ga(x_0)>1$, there exists $R=R(x_0)>0$ such that $B_{2R}(x_0):=(x_0-2R, x_0+2R)\subset (0, 1)$ and 
		\begin{equation}\label{claim11}
			\liminf_{d\to0}u^*_d(a, x)\ge v^*(a, x)-2\ep,\ \; \forall (a, x)\in [0, a_+]\times \overline B_{R}(x_0):=[0, a_+]\times[x_0-R, x_0+R].
		\end{equation}
		
		To the end, for any $x_0\in I$ fulfilling
		\begin{equation}\label{general}
			f_u(x_0, 0)\int_0^{a_+}\be(a, x_0)e^{-\int_0^a\mu(s, x_0)ds}da>1,
		\end{equation}
		we can take $R=R(x_0, \ep)>0$ sufficiently small such that $B_{2R}(x_0)\subset (0, 1)$ and 
		\begin{eqnarray}\label{v^* le v}
			v^*(a, x)\le v^*_{\mu_{\max,B_{2R}(x_0)}, \be_{\min, B_{2R}(x_0)}, f_{\min, B_{2R}(x_0)}}(a)+\ep\ \; \text{ on }[0, a_+]\times\overline B_{2R}(x_0),
		\end{eqnarray}
		where $v^*_{\mu_{\max,B_{2R}(x_0)}, \be_{\min, B_{2R}(x_0)}, f_{\min, B_{2R}(x_0)}}$ denotes the unique positive solution of 
		\begin{equation}\label{vvv}
			\begin{cases}
				v_a=-\mu_{\max,B_{2R}(x_0)}(a)v, &a\in(0, a_+),\\
				v(0)=f_{\min, B_{2R}(x_0)}\left(\int_0^{a_+}\be_{\min, B_{2R}(x_0)}(a)v(a)da\right)
			\end{cases}
		\end{equation}
		with
		$$
		\begin{cases}
			\mu_{\max,B_{2R}(x_0)}(a):=\max_{x\in\overline B_{2R}(x_0)}\mu(a, x),\\
			\be_{\min, B_{2R}(x_0)}(a):=\min_{x\in\overline B_{2R}(x_0)}\be(a, x),\\
			f_{\min, B_{2R}(x_0)}(u):=\min_{x\in\overline B_{2R}(x_0)}f(x, u).
		\end{cases}
		$$
		
		In what follows, we verify \eqref{v^* le v}. In fact, one can choose $R$ small enough such that 
		\begin{eqnarray}\label{cont}
			\begin{cases}
				|\mu(a, x)-\mu_{\max, B_{2R}(x_0)}(a)|\le C\ep,\\
				|\be(a, x)-\be_{\min, B_{2R}(x_0)}(a)|\le C\ep,\\
				|f(x, u)-f_{\min, B_{2R}(x_0)}(u)|\le C\ep,
			\end{cases}
			\text{ for all }(a, x)\in [0, a_+]\times\overline B_{2R}(x_0) \text{ with some }C>0.
		\end{eqnarray}
		Because of \eqref{cont}, by requiring $\ep$ to be smaller if necessary, we get from \eqref{general} that
		\begin{eqnarray}
			f'_{\min, B_{2R}(x_0)}(0)\int_0^{a_+}\be_{\min, B_{2R}(x_0)}(a)e^{-\int_0^a\mu_{\max, B_{2R}(x_0)}(s)ds}da>1,\label{homohomo>1}
		\end{eqnarray}
		which shows that the solution $v^*_{\mu_{\max,B_{2R}(x_0)}, \be_{\min, B_{2R}(x_0)}, f_{\min, B_{2R}(x_0)}}$ of \eqref{vvv} exists and is unique. 
		
		Due to Assumption \ref{assump-f}-(iii), for any $x\in[0, 1]$, the function $g(x, u):=\frac{f(x, u)}{u}$ is invertible and monotone. It follows that the function $\tilde g(u):=\frac{f_{\min, B_{2R}(x_0)}(u)}{u}$ is also invertible and monotone. Denote
		\begin{equation}\nonumber
			\theta(x_0):=\tilde g^{-1}\left(\frac{1}{\int_0^{a_+}\be_{\min, B_{2R}(x_0)}(a)e^{-\int_0^a\mu_{\max, B_{2R}(x_0)}(s)ds}da}\right)>0,
		\end{equation}
		which implies (by \eqref{vvv}) that
		$$
		\theta(x_0)=v(0)\int_0^{a_+}\be_{\min, B_{2R}(x_0)}(a)e^{-\int_0^a\mu_{\max,B_{2R}(x_0)}(s)ds}da.
		$$
		Thus the solution of \eqref{vvv} can be written as 
		\begin{equation}\label{v^*}
			v^*_{\mu_{\max,B_{2R}(x_0)}, \be_{\min, B_{2R}(x_0)}, f_{\min, B_{2R}(x_0)}}(a)=\frac{\theta(x_0) e^{-\int_0^a\mu_{\max,B_{2R}(x_0)}(s)ds}}{\int_0^{a_+}\be_{\min, B_{2R}(x_0)}(a)e^{-\int_0^a\mu_{\max,B_{2R}(x_0)}(s)ds}da}.
		\end{equation}
		Similarly, for each $x\in B_{2R}(x_0)$, by \eqref{general} again, \eqref{solutionofQ MR} has a unique solution 
		\begin{equation}\label{r^*}
			v^*(a, x)=\frac{\theta(x)e^{-\int_0^a\mu(s, x)ds}}{\int_0^{a_+}\be(a, x)e^{-\int_0^a\mu(s, x)ds}da},
		\end{equation}
		where 
		\begin{equation}\nonumber
			\theta(x):=g^{-1}\left(x, \frac{1}{\int_0^{a_+}\be(a, x)e^{-\int_0^a\mu(s, x)ds}da}\right)>0,\; x\in B_{2R}(x_0).
		\end{equation}
		As a result, $|\theta(x)-\theta(x_0)|\le C\ep$ for all $x\in\overline B_{2R}(x_0)$ by \eqref{cont}. By the continuity of compositions of functions, \eqref{v^* le v} follows from \eqref{v^*} and \eqref{r^*}.
		
		Moreover, due to \eqref{v^* le v} and \eqref{homohomo>1}, applying Proposition \ref{Homo Dirichlet} with $I$ replaced by $B_R(x_0)$ yields that 
		\begin{eqnarray}
			&&v^*(a, x)-2\ep\nonumber\\
			&\le&v^*_{\mu_{\max,B_{2R}(x_0)}, \be_{\min, B_{2R}(x_0)}, f_{\min, B_{2R}(x_0)}}(a)-\ep\nonumber\\
			&\le& \liminf_{d\to0}u^*_{d, \mu_{\max,B_{2R}(x_0)}, \be_{\min, B_{2R}(x_0)}, f_{\min, B_{2R}(x_0)}; B_{2R}(x_0)}(a, x)\ \; \text{ on }[0, a_+]\times\overline B_R(x_0),\label{ep2}
		\end{eqnarray}
		where $u^*_{d, \mu_{\max,B_{2R}(x_0)}, \be_{\min, B_{2R}(x_0)}, f_{\min, B_{2R}(x_0)}; B_{2R}(x_0)}$ stands for the unique positive solution of 
		\begin{equation}\nonumber
			\begin{cases}
				w_a=d w_{xx}-\mu_{\max,B_{2R}(x_0)}(a)w, &(a, x)\in(0, a_m)\times B_{2R}(x_0),\\
				w(a, x)=0, &(a, x)\in(0, a_m)\times\partial B_{2R}(x_0),\\
				w(0, x)=f_{\min, B_{2R}(x_0)}\left(\int_{0}^{a_m}\beta_{\min, B_{2R}(x_0)}(a)w(a, x)da\right), &x\in B_{2R}(x_0),
			\end{cases}
		\end{equation}
		whose existence and uniqueness for sufficiently small $d>0$ follow from \eqref{homohomo>1} and Kang \cite[Theorem 4.2]{kang2022effects}.

		On the other hand, since $u^*_d$ is a positive strict super-solution of
		\begin{equation}\label{B_2R}
			\begin{cases}
				w_a=d w_{xx}-\mu(a, x)w, &(a, x)\in(0, a_+)\times B_{2R}(x_0),\\
				w(a, x)=0, &(a, x)\in(0, a_+)\times\partial B_{2R}(x_0),\\
				w(0, x)=f\left(x, \int_{0}^{a_+}\beta(a, x)w(a, x)da\right), &x\in B_{2R}(x_0),
			\end{cases}
		\end{equation}
		the comparison principle gives
		\begin{eqnarray}\label{w*d}
			u^*_{d, \mu_{\max,B_{2R}(x_0)}, \be_{\min, B_{2R}(x_0)}, f_{\min, B_{2R}(x_0)}; B_{2R}(x_0)}\le u^*_{d, B_{2R}(x_0)}(a, x)\le u^*_d(a, x)\ \;\text{ on }[0, a_+]\times B_{2R}(x_0),
		\end{eqnarray}
		where $u^*_{d, B_{2R}(x_0)}(a, x)$ is the unique positive solution of \eqref{B_2R}, whose existence and uniqueness for sufficiently small $d>0$ are ensured by \eqref{general} and Kang \cite[Theorem 4.2]{kang2022effects}. Indeed, the principal eigenvalue $\la_{0, B_{2R}(x_0)}$ of \eqref{B_2R} linearized at zero converges to $\al_{\max, B_{2R}(x_0)}$ as $d$ goes to zero, where $\al_{\max, B_{2R}(x_0)}$ satisfies
		$$
		\max_{x\in\overline B_{2R}(x_0)}f_u(x, 0)\int_0^{a_+}\be(a, x)e^{-\al_{\max, B_{2R}(x_0)}a}e^{-\int_0^a\mu(s, x)ds}da=1.
		$$
		While \eqref{general} and \eqref{cont} imply $\al_{\max, B_{2R}(x_0)}>0$ and thus $\la_{0, B_{2R}(x_0)}>0$.
		Hence we get from \eqref{ep2} and \eqref{w*d} that 
		$$
		v^*(a, x)-2\ep\le\liminf_{d\to0}u^*_{d}(a, x)\ \; \text{ on }[0, a_+]\times\overline B_R(x_0).
		$$       
		Thus the claim \eqref{claim11} has been verified.
		
		Following the same argument as in the last paragraph of the proof of Proposition \ref{Homo Dirichlet}, we can prove that \eqref{gg general} holds on $[0, a_+]\times I$ for any $a_+\in[a_c, a_m)$ and any compact subset $I$ of $(0, 1)$ with $\Ga(x)>1$ for all $x\in I$.
	\end{proof}
	
	With the above preparations, we are now in a position to complete the proof of Theorem \ref{AP LA WA}.
	\begin{proof}[Proof of Theorem \ref{AP LA WA}] Let $\Ga$ be defined as in \eqref{Q},choose any $a_+\in[a_c, a_m)$ and suppose that $I$ is a compact subset of $(0, 1)$ and $\Ga(x)>1$ for all $x\in I$. By the continuity of $\Ga(x)$, there exist open neighborhoods $U$ and $V$ of $I$ with smooth boundaries such that
		\begin{eqnarray}\label{UV}
			I\subset U\subset \overline U\subset V\subset (0, 1)\; \text{ and }\;\Ga(x)>1\ \text{ for all }x\in V.
		\end{eqnarray}
		Now let $\eta\in C^\infty([0, 1])$ be such that
		$$
		\eta\equiv0\ \, \text{ on }\overline U,\ \, \;\eta\equiv1\ \text{ on }[0, 1]\setminus V,\; \text{ and }\eta(x)\in(0, 1)\ \text{ for all }x\in V\setminus \overline U,
		$$
		and for every $\ga>0$, consider the function
		$$
		\be_\ga(a, x)=\be(a, x)+\ga\eta(x).
		$$
		Thus $\be_\ga(a, x)=\be(a, x)$ for all $(a, x)\in[0, a_+]\times U$ and $\be_\ga\ge\not= \beta$ in $[0, a_+]\times[0, 1]$. It follows that 
		\begin{equation}\label{equal}
			v^*\equiv v^*_\ga\ \ \;\text{ on }[0, a_+]\times I, 
		\end{equation}
		where $v^*_\ga$ for each $x\in I$ denotes the positive solution of 
		\begin{equation}\label{v ga}
			\begin{cases}
				v_a=-\mu(a, x)v, &a\in(0, a_+),\\
				v(0, x)=f\left(x, \int_0^{a_+}\be_\ga(a, x)v(a, x)\right),
			\end{cases}
		\end{equation}
		whose existence and uniqueness follow from 
		$$
		\Ga_\ga(x):=f_u(x, 0)\int_0^{a_+}\be_\ga(a, x)e^{-\int_0^a\mu(s, x)ds}da\ge\Ga(x)>1\ \; \text{ for all $x\in I$.}
		$$
		Moreover, according to \eqref{UV}, we have that, $\Ga_\ga(x)\ge1$ for every $x\in V$.  Similarly, if $x\in[0, 1]\setminus V$, then for sufficiently large $\ga>0$, we have 
		$$
		\Ga_\ga(x)=f_u(x, 0)\int_0^{a_+}(\be(a, x)+\ga)e^{-\int_0^a\mu(s, x)ds}da>1.
		$$
		Therefore, for sufficiently large $\ga>0$, one has $\Ga_\ga(x)>1$ for all $x\in[0, 1]$. Thus by Proposition \ref{auto} and \eqref{sub f}, we infer
		\begin{eqnarray}\label{sub ff}
			\limsup\limits_{d\to0}u^*_{d, \ga}(a, x)\le (1+\de) v^*_\ga(a, x),\ \; \forall(a, x)\in[0, a_+]\times [0, 1], 
		\end{eqnarray}
		where $u^*_{d, \ga}$ denotes the positive solution of
		\begin{equation}\label{u ga}
			\begin{cases}
				u_a=d u_{xx}-\mu(a, x)u, &(a, x)\in(0, a_m)\times (0, 1),\\
				u_x(a, x)=0, &(a, x)\in(0, a_m)\times\{0, 1\},\\
				u(0, x)=f\left(x, \int_{0}^{a_m}\beta_\ga(a, x)u(a, x)da\right), &x\in (0, 1),
			\end{cases}
		\end{equation}
		whose existence and uniqueness for sufficiently small $d>0$ follow from $\Ga_\ga(x)>1$ for all $x\in[0, 1]$ and Theorem \ref{Dlambda La=0}.
		
		On the other hand, due to $\be_\ga\ge\be$, a comparison analysis yields
		\begin{eqnarray}\label{u ga u}
			u^*_{d, \ga}(a, x)\ge u^*_d(a, x)\ \, \; \text{ on }[0, a_+]\times[0, 1].
		\end{eqnarray} 
		By \eqref{equal}, \eqref{sub ff} and \eqref{u ga u}, we deduce
		\begin{eqnarray}\label{super}
			\limsup\limits_{d\to0}u^*_{d}(a, x)\le (1+\de) v^*(a, x),\ \, \; \forall(a, x)\in[0, a_+]\times I.
		\end{eqnarray}
		In summary, combining \eqref{super} and Proposition \ref{General}, we have obtained that for any $0<\de, \ep\ll1$, any $a_+\in[a_c, a_m)$ and any compact subset $I$ with $\Ga(x)>1$ for all $x\in I$, 
		$$
		(1+\de)v^*(a, x)\ge\limsup\limits_{d\to0}u^*_d(a, x)\ge\liminf\limits_{d\to0}u^*_d(a, x)\ge v^*(a, x)-\ep, \quad \forall (a, x)\in [0, a_+]\times I.
		$$
		Due to the arbitrariness of $\de$ and $\ep$, the first assertion of Theorem \ref{AP LA WA} holds. 
		
		For the second assertion of Theorem \ref{AP LA WA}, we can use a similar argument as above. Indeed, suppose that $I$ is a compact subset of $(0, 1)$ and $\Ga(x)\le1$ for all $x\in I$. For every $\eta>0$, define 
		\begin{equation}\nonumber
			\be_\eta(a, x)=\be(a, x)e^{[(1+\eta)|\al(x)|+\eta]a},
		\end{equation}
		where $\al(x)$ is defined through
		$$
		f_u(x, 0)\int_0^{a_+}\be(a, x)e^{-\al(x)a}e^{-\int_0^a\mu(s, x)ds}da=1.
		$$
		Since $\eta>0$, it is apparent that $\be_\eta>\be$ in $[0, a_+]\times[0, 1]$. Moreover, there holds
		\begin{equation}\label{Ga>1}
			\Ga_\eta(x):=f_u(x, 0)\int_0^{a_+}\be_\eta(a, x)e^{-\int_0^a\mu(s, x)ds}da>1\ \;\text{ for all }x\in[0, 1].
		\end{equation}
		Indeed, \eqref{Ga>1} is obvious if $\Ga(x)>1$. Suppose $\Ga(x)\le1$, then by definition, $\al(x)\le0$ and 
		$$
		\Ga_\eta(x)=f_u(x, 0)\int_0^{a_+}\be(a, x)e^{-\al(x)a}e^{\eta(1-\al(x))a}e^{-\int_0^a\mu(s, x)ds}da>1.
		$$
		In view of Proposition \ref{auto} and \eqref{sub f}, we have
		\begin{eqnarray}\label{sub fff}
			\limsup\limits_{d\to0}u^*_{d, \eta}(a, x)\le (1+\de) v^*_\eta(a, x),\ \, \; \forall(a, x)\in[0, a_+]\times [0, 1], 
		\end{eqnarray}
		where $u^*_{d, \eta}$ and $v^*_\eta$ denote the positive solutions of \eqref{u ga} and \eqref{v ga} with $\be_\ga$ replaced by $\be_\eta$, respectively, whose existence and uniqueness for sufficiently small $d>0$ are guaranteed by $\Ga_\eta(x)>1$ for all $x\in[0, 1]$ and Theorem \ref{Dlambda La=0} again. 
		
		On the other hand, due to $\be_\eta\ge\be$, a comparison argument gives
		\begin{eqnarray}\label{u eta u}
			u^*_{d, \eta}(a, x)\ge u^*_d(a, x)\ \ \; \text{ on }[0, a_+]\times[0, 1].
		\end{eqnarray} 
		Consequently, by \eqref{sub fff} and \eqref{u eta u}, we find that
		\begin{eqnarray}\label{super eta}
			\limsup\limits_{d\to0}u^*_{d}(a, x)\le (1+\de) v^*_\eta(a, x),\ \, \; \forall(a, x)\in[0, a_+]\times[0, 1].
		\end{eqnarray}
		Due to the continuity of $\eta\longmapsto v^*_\eta$, letting $\eta\to0$ in \eqref{super eta} yields
		$$
		\limsup\limits_{d\to0}u^*_{d}(a, x)\le (1+\de) v^*_0(a, x),\ \, \; \forall(a, x)\in[0, a_+]\times[0, 1].
		$$
		As $\al(x)\le0$ in $I$, and in turn
		$$
		f_u(x, 0)\int_0^{a_+}\be_0(a, x)e^{-\int_0^a\mu(s, x)ds}da=f_u(x, 0)\int_0^{a_+}\be(a, x)e^{-\al(x)a}e^{-\int_0^a\mu(s, x)ds}da=1\ \;\text{ for all }x\in I,
		$$
		this implies that $v^*_0\equiv0$ in $I$ by Ducrot et al. \cite[Lemma 5.4]{Ducrot2022Age-structuredII}, and so the second assertion of Theorem \ref{AP LA WA} holds. The proof of Theorem \ref{AP LA WA} is now complete.
	\end{proof}
	
	\subsection{Asymptotic profile for large diffusion}
	
	In this subsection, we aim to examine the asymptotic profile of \( u^* \) under large diffusion regardless of the presence of advection and provide the proof of Theorem \ref{d to inf}. As before, we choose any $a_+\in[a_c, a_m)$ from now on. 
	
	\begin{proof}[Proof of Theorem \ref{d to inf}]
		First, we can apply a similar super- and sub-solution method as in Ducrot et al. \cite[Theorem 4.4]{Ducrot2022Age-structuredII} to establish the existence and uniqueness of the solution to problem \eqref{h}. Specifically, the constant \( L \) can be chosen as a super-solution, while a sub-solution can be selected as \( \epsilon \phi_\delta \), where \(\phi_\delta\) satisfies:
		\[
		\begin{cases}
			\phi'_\delta(a) = -\int_0^1 \mu(a, x) \, dx - \lambda_\delta \phi_\delta(a), & a \in (0, a_+),\\
			\phi_\delta(0) = \int_0^1 \int_0^{a_+} (f_u(x, 0) - \delta) \beta(a, x) \phi_\delta(a) \, da \, dx.
		\end{cases}
		\]
		Observe that \eqref{al average} implies \(\lambda_\delta > 0\) for \(\delta > 0\) sufficiently small, thus ensuring that \(\epsilon \phi_\delta\), when $\epsilon$ is chosen small enough, indeed serves as a sub-solution to \eqref{h}. Then, by employing a monotone argument similar to that in \cite[Theorem 4.4]{Ducrot2022Age-structuredII}, we can obtain the existence and uniqueness of solutions. Here, we omit the detailed steps.
		
		Next we set $v(a, x)=e^{-(\La/d)x}u^*(a, x)$ and transform the problem \eqref{logistic} into 
		\begin{equation}\label{vv}
			\begin{cases}
				v_a=dv_{xx}+\La v_x-\mu(a, x)v, &(a, x)\in(0, a_m)\times(0, 1),\\
				v(0, x)=e^{-(\La/d)x}f\left(x, e^{(\La/d)x}\int_0^{a_m}\beta(a, x)v(a, x)da\right), &x\in(0, 1),\\
				v_x(a, x)=0, &(a, x)\in(0, a_m)\times\{0, 1\}.
			\end{cases}
		\end{equation}
		We make a variable transformation: $x=d^{1/2}y$ and denote $w(a, y)=v(a, d^{1/2}y)=v(a, x)$. Thus \eqref{vv} becomes
		\begin{equation}\label{w-equation}
			\begin{cases}
				w_a=w_{yy}+\frac{\La}{d^{1/2}} w_y-\widetilde\mu(a, y)w, &(a, y)\in(0, a_m)\times(0, d^{-1/2}),\\
				w(0, y)=e^{-(\La/d^{1/2})y}f\left(d^{1/2}y, e^{(\La/d^{1/2})y}\int_{0}^{a_m}\widetilde\beta(a, y)w(a, y)da\right), &y\in(0, d^{-1/2}),\\
				w_y(a, y)=0, &(a, y)\in(0, a_m)\times\{0, d^{-1/2}\},
			\end{cases}
		\end{equation}
		where $\widetilde\mu(a, y)=\mu(a, d^{1/2}y)$ and $\widetilde\be(a, y)=\beta(a, d^{1/2}y)$.
		
		Since $\widetilde\mu$ and $\widetilde\beta$ are bounded on $[0, a_+]\times[0, d^{-1/2}]$, combined with the estimate $v(a, x)\le L$ on $[0, a_+]\times[0, 1]$, for any given $p>1$, we can apply the standard global $L^p$ estimates for parabolic equations (up to the boundary and initial data; see, for instance \cite{LSU,Lie}) to the system \eqref{w-equation} restricted on $[0, a_+]\times[0, d^{-1/2}]$ to assert that 
		$$
		\norm{w}_{W^{1, p}((0, a_+), W^{2, p}(0, d^{-1/2}))}\le C_0
		$$
		for some positive constant $C_0$, which may vary from place to place but is independent of all large $d$. Hence, appealing to the embedding theory (by taking sufficiently large $p$), we can infer
		$$
		\norm{w}_{C^{(1+\de)/2, 1+\de}([0, a_+]\times[0, d^{-1/2}])}\le C_0
		$$
		for some $\de\in(0, 1)$. Going back to the solution $u^*$, for sufficiently large $d>0$, we then have
		\begin{equation}\label{partial u_d}
			\norm{u^*}_{C^{(1+\de)/2, 1+\de}([0, a_+]\times[0, 1])}\le C_0, \quad \norm{\partial_xu^*}_{C([0, a_+]\times[0, 1])}\le\frac{C_0}{d^{1/2}}.
		\end{equation}
		Therefore, by passing up to a sequence, we may assume that 
		$$
		u^*\to u_*\ \  \text{ in }C^{(1+\de/2)/2, 1+\de/2}([0, a_+]\times[0, 1])\ \ \text{ as }d\to\infty,
		$$
		where $u_* \in C^{(1+\de/2)/2, 1+\de/2}([0, a_+]\times[0, 1])$ is a nonnegative function. Due to the last estimate in \eqref{partial u_d}, clearly $u_*$ satisfies $\partial_xu_*\equiv0$ in $[0, a_+]\times[0, 1]$, and so $u_*(a, x)=u_*(a)$ is a function independent of the spatial variable.
		
		On the other hand, integrating the equation \eqref{logistic} that $u^*$ satisfies over $(0, a)\times(0, 1)$, we find
		$$
		\int_0^1u^*(a, x)dx-\int_0^1u^*(0, x)dx=-\int_0^a\int_0^1\mu(s, x)u^*(s, x)dxds.
		$$
		By taking $d\to\infty$, it is easily seen that
		$$
		u_*(a)-u_*(0)=-\int_0^{a}\int_0^1\mu(s, x)u_*(s)dxds.
		$$
		Therefore we differentiate the above equality with respect to $a$ to see that $u_*$ solves the first equation in the ODE problem \eqref{h}. In addition, due to the second equation of \eqref{logistic}, $u_*(0)$ satisfies
		$$
		u_*(0)=\int_0^1f\left(x, \int_0^{a_m}\beta(a, x)u_*(a)da\right)dx.
		$$
		
		To complete the proof, it remains to show that $u_*>0$ in $[0, a_+]$. Let us prove it by contradiction via supposing that $u_*(a)=0$ for some $a\in[0, a_+]$. By \eqref{h}, it is easily seen that $u_*\equiv0$ in $[0, a_+]$. Thus,
		$$
		u^*(a, x)\to0\ \ \text{ uniformly on $[0, a_+]\times[0, 1]$\ \, as $d\to\infty$}.
		$$
		Now we define 
		$$
		z(a, x):=\frac{u^*(a, x)}{\max_{[0, a_+]\times[0, 1]}u^*(a, x)}, \;\forall  d>0.
		$$
		Then $\norm{z}_{L^\infty((0, a_+)\times(0, 1))}=1$ for all $d>0$ and $z$ solves the following problem:
		\begin{equation}\label{zz}
			\begin{cases}
				z_a=dz_{xx}-\La z_x-\mu(a, x)z, &(a, x)\in(0, a_m)\times(0, 1),\\
				z(0, x)=\frac{f\left(x, \int_0^{a_m}\beta(a, x)u(a, x)da\right)}{u(a, x)}z(a, x), &x\in(0, 1),\\
				dz_x(a, x)-\La z(a, x)=0, &(a, x)\in(0, a_m)\times\{0, 1\}.
			\end{cases}
		\end{equation}
		A similar analysis as before, together with Assumption \ref{assump-f}, enables one to conclude that up to a subsequence of $d$, 
		$$
		z(a, x)\to z_*(a)\ \ \text{ uniformly on $[0, a_+]\times[0, 1]$\ \ as $d\to\infty$},
		$$
		where $z_*\in C^{(1+\de/2)/2, 1+\de/2}([0, a_+]\times[0, 1])$. Due to $\norm{z_d}_{L^\infty((0, a_+)\times(0, 1))}=1$ for all $d>0$, it follows that 
		$z_*(a)\geq,\not\equiv0$ on $[0, a_+]$. 
		
		By integrating the first equation of \eqref{zz} over $[0, a]\times[0, 1]$, we find 
		$$
		\int_0^1z(a, x)dx-\int_0^1z(0, x)dx=-\int_0^a\int_0^1\mu(s, x)z(s, x)dxds.
		$$
		Sending $d\to\infty$ yields that 
		$$
		z_*(a)-z_*(0)=-\int_0^az_*(s)\int_0^1\mu(s, x)dxds,
		$$
		which implies that
		\begin{equation}\label{z*}
			z'_*(a)=-\int_0^1\mu(a, x)dx\; z_*(a).
		\end{equation}
		Integrating the second equation of \eqref{zz} over $[0, 1]$ and letting $d\to\infty$ yields that (via chain rule) 
		\begin{equation}\label{z*(0)}
			z_*(0)=\int_0^1\int_0^{a_c}f_u(x, 0)\beta(a, x)z_*(a)dadx.
		\end{equation}
		Finally, combining \eqref{z*} and \eqref{z*(0)}, we have 
		$$
		\int_0^1\int_0^{a_c}f_u(x, 0)\beta(a, x)e^{-\int_0^a\int_0^1\mu(s, x)dxds}dadx=1,
		$$
		which is a contradiction with \eqref{al average}. Thus $u_*(a)>0$ on $[0, a_+]$ and the proof is complete.
	\end{proof}
	
	\section{Appendix}\label{appendixtheories}
	
	To illustrate our constructions of generalized super- and sub-solutions in the proofs of Theorems \ref{Lalambda} and \ref{Dlambda}, we consider a fixed \(a\), and plot two profiles: the super-solution \(\overline{w}\) (see \eqref{large_advec_super_sol}) for large advection in the proof of Theorem \ref{Lalambda}, and the sub-solution \(\underline{w}\) (see \eqref{small d sub-sol}) for small diffusion in the proof of Theorem \ref{Dlambda}. These profiles are depicted in Figures \ref{overline w} and \ref{underline w}, respectively.
	\begin{figure}\label{super sub-solutions}
		\begin{subfigure}{0.5\textwidth}
			\centering
			\includegraphics[width=\textwidth]{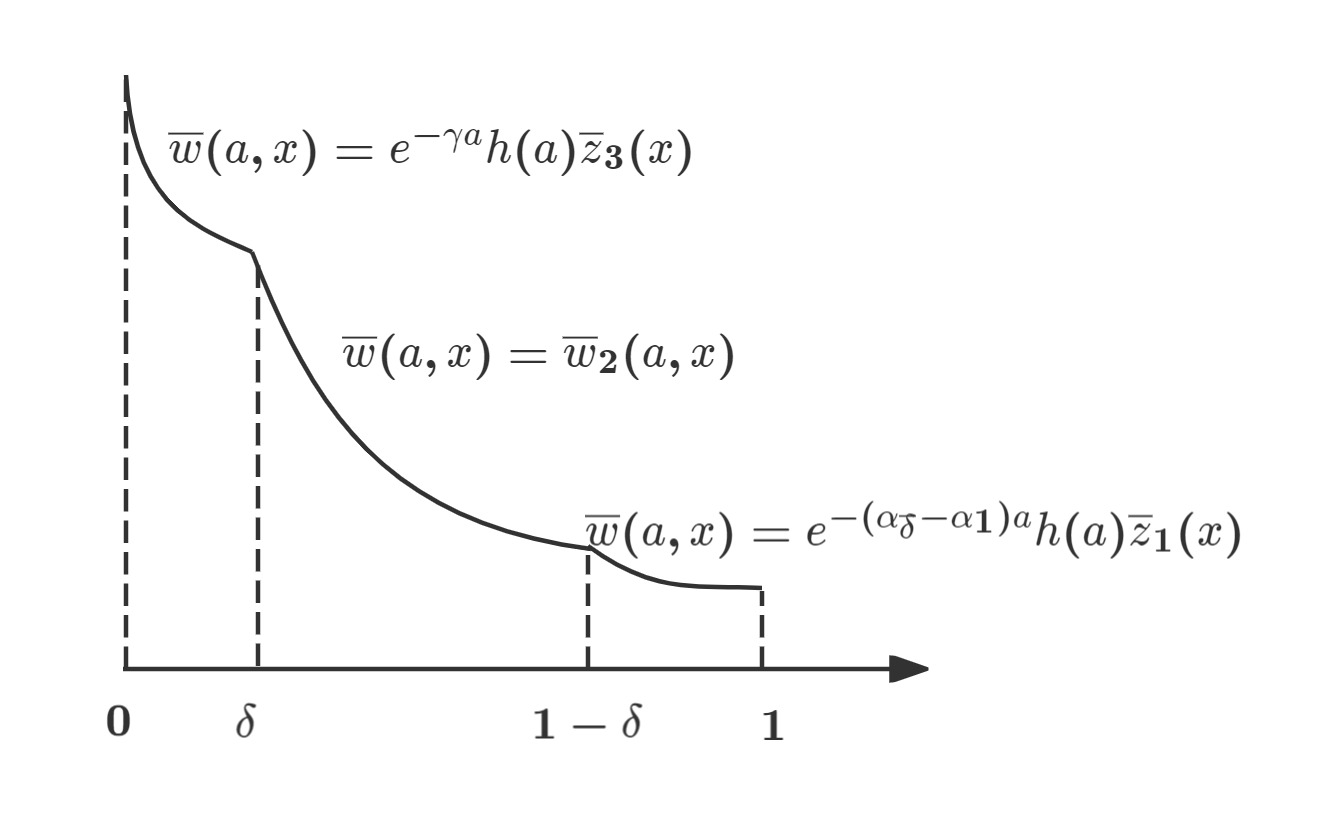}
			\subcaption{Profile of $\overline w$ for fixed $a$ with large advection.\label{overline w}}
		\end{subfigure}
		\begin{subfigure}{0.5\textwidth}
			\centering
			\includegraphics[width=\textwidth]{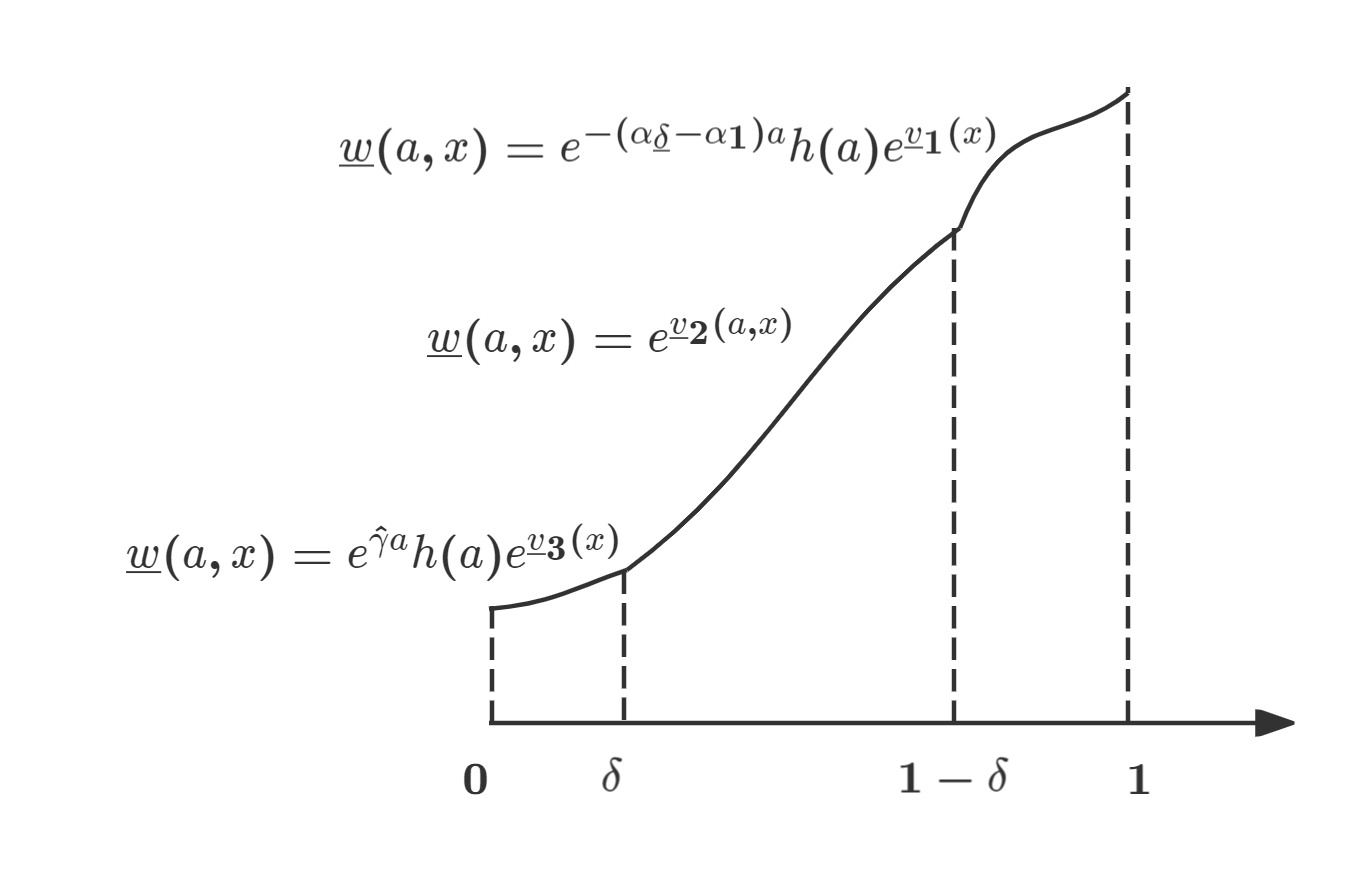}
			\subcaption{Profile of $\underline w$ for fixed $a$ with small diffusion.\label{underline w}}
		\end{subfigure}
		\caption{Profiles of $\overline w$ and $\underline w$ for fixed $a$.}
	\end{figure}

	\vskip25pt
	\noindent{\bf Data availability statement}:\ \ No new data were created or analysed in this study.\\
	\bigskip
	
	\noindent\textbf{Acknowledgements}. We thank the anonymous reviewers for their helpful comments and suggestions.
	\vskip25pt
	
	\bibliographystyle{plain}

\begin{thebibliography}{}
		\setlength{\itemsep}{-1mm}
		{\footnotesize
			
			\bibitem{berestycki2005elliptic}
			H.~Berestycki, F.~Hamel and N.~Nadirashvili.
			\newblock Elliptic eigenvalue problems with large drift and applications to
			nonlinear propagation phenomena.
			\newblock {\em Commun. Math. Phys.}, 253:451--480, 2005.
			
			\bibitem{busenberg198nonlinear}
			S.~ Busenberg and M.~ Iannelli.
			\newblock Nonlinear diffusion problems in age-structured population dynamics.
			\newblock {\em Lecture Notes in Biomathematics, 54, Mathematical Ecology}, 425--440, 1982.
			
			\bibitem{cano2024singular}
			S.~Cano-Casanova, S.~ Fern{\'a}ndez-Rinc{\'o}n and J.~L{\'o}pez-G{\'o}mez.
			\newblock A singular perturbation result for a class of periodic-parabolic BVPs.
			\newblock {\em Open Math.}, 22(1): 20240020, 2024.
			
			
			\bibitem{chan1990semigroups}
			W.L. Chan and B.Z. Guo.
			\newblock On the semigroups of age-size dependent population dynamics with
			spatial diffusion.
			\newblock {\em Manuscripta Math.}, 66(1):161--181, 1990.
			
			\bibitem{chen2008principal}
			X.~Chen and Y.~Lou.
			\newblock Principal eigenvalue and eigenfunctions of an elliptic operator with
			large advection and its application to a competition model.
			\newblock {\em Indiana Univ. Math. J.}, 57(2):627--658, 2008.
			
			\bibitem{chen2012effects}
			X.~Chen and Y.~Lou.
			\newblock Effects of diffusion and advection on the smallest eigenvalue of an
			elliptic operator and their applications.
			\newblock {\em Indiana Univ. Math. J.}, 61(1):45--80, 2012.
			
			
			\bibitem{cui2017dynamics}
			R.~Cui, K.-Y. Lam and Y.~Lou.
			\newblock Dynamics and asymptotic profiles of steady states of an epidemic
			model in advective environments.
			\newblock {\em J. Differential Equations}, 263(4):2343--2373, 2017.
			
			\bibitem{cui2021concentration}
			R.~Cui, H.~Li, R.~Peng and M.~Zhou.
			\newblock Concentration behavior of endemic equilibrium for a
			reaction--diffusion--advection {SIS} epidemic model with mass action
			infection mechanism.
			\newblock {\em Calc. Var. Partial Differential Equations}, 60(5), 2021, Paper No. 184, 38 pp.
			
			\bibitem{cui2016spatial}
			R.~Cui and Y.~Lou.
			\newblock A spatial {SIS} model in advective heterogeneous environments.
			\newblock {\em J. Differential Equations}, 261(6):3305--3343, 2016.
			
			\bibitem{delgado2006nonlinear}
			M.~Delgado, M.~Molina-Becerra and A.~Su{\'a}rez.
			\newblock A nonlinear age-dependent model with spatial diffusion.
			\newblock {\em J. Math. Anal. Appl.}, 313(1):366--380, 2006.
			
			\bibitem{delgado2008nonlinear}
			M.~Delgado, M.~Molina-Becerra and A.~Su{\'a}rez.
			\newblock Nonlinear age-dependent diffusive equations: A bifurcation approach.
			\newblock {\em J. Differential Equations}, 244(9):2133--2155, 2008.
			
			\bibitem{di1979non}
			G.~Di~Blasio.
			\newblock Non-linear age-dependent population diffusion.
			\newblock {\em J. Math. Biol.}, 8(3):265--284, 1979.
			
			\bibitem{Ducrot2022Age-structuredII}
			A.~Ducrot, H.~Kang and S.~Ruan.
			\newblock Age-structured models with nonlocal diffusion of {D}irichlet type,
			{II}: global dynamics.
			\newblock {\em Israel J. Math.}, 266(1):219--257, 2025.
			
			\bibitem{Ducrot2022Age-structuredI}
			A.~Ducrot, H.~Kang and S.~Ruan.
			\newblock Age-structured models with nonlocal diffusion of {D}irichlet typee,
			{I}: principal spectral theory and limiting properties.
			\newblock {\em J. Anal. Math.}, 155(1):327--390, 2025.
			
			\bibitem{fernandez2019singular}
			S.~Fern{\'a}ndez-Rinc{\'o}n and J.~L{\'o}pez-G{\'o}mez
			\newblock The singular perturbation problem for a class of generalized logistic equations under non-classical mixed boundary conditions.
			\newblock {\em Adv. Nonlinear Stud.}, 19(1): 1--27, 2019.
			
			\bibitem{fife2017integrodifferential}
			P.C. Fife.
			\newblock An integrodifferential analog of semilinear parabolic {PDE}’s.
			\newblock In {\em Partial Differential Equations and Applications}, pages
			137--145. Routledge, 2017.
			
			
			\bibitem{guo1994semigroup}
			B.Z. Guo and W.L. Chan.
			\newblock On the semigroup for age dependent population dynamics with spatial
			diffusion.
			\newblock {\em J. Math. Anal. Appl.}, 184(1):190--199, 1994.
			
			\bibitem{gurtin1973system}
			M.E. Gurtin.
			\newblock A system of equations for age-dependent population diffusion.
			\newblock {\em J. Theoret. Biol.}, 40(2):389--392, 1973.
			
			\bibitem{gurtin1981diffusion}
			M.E. Gurtin and R.C. MacCamy.
			\newblock Diffusion models for age-structured populations.
			\newblock {\em Math. Biosci.}, 54(1):49--59, 1981.
			
			\bibitem{hastings1992age}
			A.~Hastings.
			\newblock Age dependent dispersal is not a simple process: Density dependence,
			stability, and chaos.
			\newblock {\em Theoret. Population Biol.}, 41(3):388--400, 1992.
			
			
			\bibitem{hess1991periodic}
			P.~Hess.
			\newblock {\em Periodic-Parabolic Boundary Value Problems and Positivity}.
			\newblock Longman, London, 1991.
			
			\bibitem{hutson2003evolution}
			V.~Hutson, S.~Martinez, K.~Mischaikow and G.~T. Vickers.
			\newblock The evolution of dispersal.
			\newblock {\em J. Math. Biol.}, 47(6):483--517, 2003.
			
			\bibitem{huyer1994semigroup}
			W.~Huyer.
			\newblock Semigroup formulation and approximation of a linear age-dependent
			population problem with spatial diffusion.
			\newblock In {\em Semigroup Forum}, volume~49, pages 99--114. Springer, 1994.
			
			\bibitem{inaba2017age} 
			H.~Inaba. 
			\textit{Age-Structured Population Dynamics in Demography and Epidemiology}, Springer, New York, 2017.
			
			\bibitem{kang2022effects}
			H.~Kang.
			\newblock The effects of diffusion on the principal eigenvalue for
			age-structured models with random diffusion.
			\newblock {\em Proc. Roy. Soc. Edinburgh Sect. A}, 152(1):258--280, 2022.
			
			\bibitem{kato2013perturbation}
			T.~Kato.
			\newblock {\em Perturbation Theory for Linear Operators}, volume 132.
			\newblock Springer Science \& Business Media, 2013.
			
			
			\bibitem{KMP2017} 
			{K. Kuto, H. Matsuzawa and R. Peng}. Concentration profile of endemic equilibrium of a reaction-diffusion-advection SIS epidemic model, {\em Calc. Var. Partial Differential Equations}. 56 (2017), Paper No. 112, 28 pp.
			
			
			\bibitem{LSU}
			O.A. Ladyzenskaja, V.A. Solonnikov and N.N. Uralceva, \newblock {\em Linear and Quasi-Linear Equations of Parabolic Type}, AMS, Providence, RI, 1968. 
			
			\bibitem{lam2016asymptotic}
			K.-Y. Lam and Y.~Lou.
			\newblock Asymptotic behavior of the principal eigenvalue for cooperative
			elliptic systems and applications.
			\newblock {\em J. Dynam. Differential Equations}, 28(1):29--48, 2016.
			
			\bibitem{langlais1985nonlinear}
			M.~Langlais.
			\newblock A nonlinear problem in age-dependent population diffusion.
			\newblock {\em SIAM J. Math. Anal.}, 16(3):510--529, 1985.
			
			\bibitem{langlais1988large}
			M.~Langlais.
			\newblock Large time behavior in a nonlinear age-dependent population dynamics problem with spatial diffusion.
			\newblock {\em J. Math. Biol.}, 26:319--346, 1988.
			
			\bibitem{Lie}
			G.M.~Lieberman, \newblock {\em Second Order Parabolic Differential Equations}, World Scientific Publ. Co., Inc., River Edge, NJ, 1996. 
			
			\bibitem{LiuL2022} S. Liu, Y. Lou, Classifying the level set of principal eigenvalue for time-periodic parabolic operators and applications, J. Funct. Anal. 282 (2022) 109338.
			
			\bibitem{LiuL2024} S. Liu, Y. Lou, Monotonicity, asymptotics and level sets for principal eigenvalues of some elliptic operators with shear flow. J. Math. Pures Appl. 191 (2024), Paper No. 103622, 42 pp.
			
			\bibitem{liu2019monotonicity}
			S.~Liu, Y.~Lou, R.~Peng and M.~Zhou.
			\newblock Monotonicity of the principal eigenvalue for a linear time-periodic
			parabolic operator.
			\newblock {\em Proc. Amer. Math. Soc.}, 147(12):5291--5302, 2019.
			
			\bibitem{liu2021asymptotics}
			S.~Liu, Y.~Lou, R.~Peng and M.~Zhou.
			\newblock Asymptotics of the principal eigenvalue for a linear time-periodic
			parabolic operator {I}: Large advection.
			\newblock {\em SIAM J. Math. Anal.}, 53(5):5243--5277, 2021.
			
			\bibitem{liu2021asymptoticsII}
			S.~Liu, Y.~Lou, R.~Peng and M.~Zhou.
			\newblock Asymptotics of the principal eigenvalue for a linear time-periodic
			parabolic operator {II: Small} diffusion.
			\newblock {\em Trans. Amer. Math. Soc.}, 374(7):4895--4930, 2021.
			
			\bibitem{liu2011hopf}
			Z.~Liu, P.~Magal and S.~Ruan.
			\newblock Hopf bifurcation for non-densely defined {C}auchy problems.
			\newblock {\em Z. Angew. Math. Phys.}, 62(2):191--222, 2011.
			
			\bibitem{lopez2013linear}
			J.~L{\'o}pez-G{\'o}mez.
			\newblock {\em Linear Second Order Elliptic Operators.}
			\newblock { World Scientific Publishing Company}, 2013.
			
			\bibitem{lou2016qualitative}
			Y.~Lou, D.~Xiao and P.~Zhou.
			\newblock Qualitative analysis for a lotka-volterra competition system in
			advective homogeneous environment.
			\newblock {\em Discrete Contin. Dyn. Syst.}, 36(2):953, 2016.
			
			\bibitem{lou2019global}
			Y.~Lou, X.-Q. Zhao and P.~Zhou.
			\newblock Global dynamics of a {Lotka--Volterra}
			competition--diffusion--advection system in heterogeneous environments.
			\newblock {\em J. Math. Pures Appl.}, 121:47--82, 2019.
			
			\bibitem{lou2015evolution}
			Y.~Lou and P.~Zhou.
			\newblock Evolution of dispersal in advective homogeneous environment: the
			effect of boundary conditions.
			\newblock {\em J. Differential Equations}, 259(1):141--171, 2015.
			
			\bibitem{maccamy1981population}
			R.C. ~MacCamy.
			\newblock A population model with nonlinear diffusion.
			\newblock {\em J. Differential Equations}, 39(1):52--72, 1981.
			
			\bibitem{magal2009center}
			P.~Magal and S.~Ruan.
			\newblock Center manifolds for semilinear equations with non-dense domain and
			applications to hopf bifurcation in age structured models.
			\newblock {\em Memoirs Amer. Math. Soc.}, 202(951):1--71, 2009.
			
			\bibitem{magal2018theory}
			P.~Magal and S.~Ruan.
			\newblock {\em Theory and Applications of Abstract Semilinear Cauchy Problems}.
			\newblock Springer, New York, 2018.
			
			\bibitem{marek1970frobenius}
			I.~Marek.
			\newblock Frobenius theory of positive operators: Comparison theorems and
			applications.
			\newblock {\em SIAM J. Appl. Math.}, 19(3):607--628, 1970.
			
			\bibitem{medlock2003spreading}
			J.~Medlock and M.~Kot.
			\newblock Spreading disease: integro-differential equations old and new.
			\newblock {\em Math. Biosci.}, 184(2):201--222, 2003.
			
			\bibitem{monmarche2025impacts}
			P.~Monmarché, S.J.~Schreiber, and É.~Strickler.
			\newblock Impacts of Tempo and Mode of Environmental Fluctuations on Population Growth: Slow- and Fast-Limit Approximations of Lyapunov Exponents for Periodic and Random Environments.
			\newblock {\em Bull. Math. Biol.}, 87(6), 81, 2025.
			
			\bibitem{murray2007mathematical}
			J.~D.~Murray.
			\newblock {\em Mathematical Biology: I. An introduction}, volume~17.
			\newblock Springer Science \& Business Media, 2007.
			
			
			\bibitem{pazy2012semigroups}
			A.~Pazy.
			\newblock {\em Semigroups of Linear Operators and Applications to Partial
				Differential Equations}.
			\newblock Springer, New York, 1983.
			
			
			\bibitem{peng2019asymptotic}
			R.~Peng, G.~Zhang and M.~Zhou.
			\newblock Asymptotic behavior of the principal eigenvalue of a linear second
			order elliptic operator with small/large diffusion coefficient.
			\newblock {\em SIAM J. Math. Anal.}, 51(6):4724--4753, 2019.
			
			\bibitem{peng2015effects}
			R.~Peng and X.-Q.~Zhao.
			\newblock Effects of diffusion and advection on the principal eigenvalue of a
			periodic-parabolic problem with applications.
			\newblock {\em Calc. Var. Partial Differential Equations}, 54(2):1611--1642,
			2015.
			
			\bibitem{peng2018effects}
			R.~Peng and M.~Zhou.
			\newblock Effects of large degenerate advection and boundary conditions on the
			principal eigenvalue and its eigenfunction of a linear second-order elliptic
			operator.
			\newblock {\em Indiana Univ. Math. J.}, pages 2523--2568, 2018.
			
			
			
			\bibitem{sawashima1964spectral}
			I.~Sawashima.
			\newblock On spectral properties of some positive operators.
			\newblock {\em Nat. Sci. Rep., Ochanomizu University}, 15(2):53--64, 1964.
			
			
			\bibitem{thieme2009spectral}
			H.R.~Thieme.
			\newblock Spectral bound and reproduction number for infinite-dimensional
			population structure and time heterogeneity.
			\newblock {\em SIAM J. Appl. Math.}, 70(1):188--211, 2009.
			
			\bibitem{walker2009positive}
			C.~Walker.
			\newblock Positive equilibrium solutions for age-and spatially-structured
			population models.
			\newblock {\em SIAM J. Math. Anal.}, 41(4):1366--1387, 2009.
			
			\bibitem{walker2008age}
			C.~Walker.
			\newblock Age-dependent equations with non-linear diffusion.
			\newblock {\em Discrete Contin. Dyn. Syst.}, 26(2):691--712, 2010.
			
			\bibitem{walker2010global}
			C.~Walker.
			\newblock Global bifurcation of positive equilibria in nonlinear population
			models.
			\newblock {\em J. Differential Equations}, 248(7):1756--1776, 2010.
			
			\bibitem{walker2011bifurcation}
			C.~Walker.
			\newblock Bifurcation of positive equilibria in nonlinear structured population
			models with varying mortality rates.
			\newblock {\em Ann. Mat. Pura Appl.}, 190(1):1--19, 2011.
			
			
			\bibitem{walker2011positive}
			C.~Walker.
			\newblock On positive solutions of some system of reaction-diffusion equations
			with nonlocal initial conditions.
			\newblock {\em J. Reine Angew. Math.}, 2011(660):149--179, 2011.
			
			\bibitem{walker2013global}
			C.~Walker.
			\newblock Global continua of positive solutions for some quasilinear parabolic
			equation with a nonlocal initial condition.
			\newblock {\em J. Dynam. Differential Equations}, 25(1):159--172, 2013.
			
			\bibitem{walker2013some}
			C.~Walker.
			\newblock Some remarks on the asymptotic behavior of the semigroup associated
			with age-structured diffusive populations.
			\newblock {\em Monatsh. Math.}, 170(3):481--501, 2013.
			
			
			\bibitem{walker2023stability}
			C.~Walker.
			\newblock Stability and instability of equilibria in age-structured diffusive
			populations.
			\newblock {\em J. Dynam. Differential Equations}, 37 (2025), no. 2, 1315--1354.
			
			\bibitem{walker2022principle}
			C.~Walker and J.~Zehetbauer.
			\newblock The principle of linearized stability in age-structured diffusive
			populations.
			\newblock {\em J. Differential Equations}, 341:620--656, 2022.
			
			\bibitem{wang2025on}
			L.~Wang, K.-Y.~Lam and B.~Zhang.
			\newblock On the principal eigenvalue of cooperative elliptic systems with applications to a population model with two reversible states.
			\newblock {\em Disc. Contin. Dyn. Syst. Ser. B}, 30(7):2306-2325, 2025.
			
			
			
			\bibitem{webb2008population}
			G.F.~Webb.
			\newblock Population models structured by age, size, and spatial position.
			\newblock In {\em Structured Population Models in Biology and Epidemiology},
			pages 1--49. Springer, 2008.
			
			
			
			\bibitem{zhao2016lotka}
			X.-Q.~Zhao and P.~Zhou.
			\newblock On a {Lotka--Volterra} competition model: the effects of advection
			and spatial variation.
			\newblock {\em Calc. Var. Partial Differential Equations}, 55(4):1--25, 2016.
			
			\bibitem{zhou2016lotka}
			P.~Zhou.
			\newblock On a {Lotka-Volterra} competition system: diffusion vs advection.
			\newblock {\em Calc. Var. Partial Differential Equations}, 55(6):1--29, 2016.
			
			\bibitem{zhou2018global}
			P.~Zhou and D.~Xiao.
			\newblock Global dynamics of a classical {Lotka--Volterra}
			competition--diffusion--advection system.
			\newblock {\em J. Funct. Anal.}, 275(2):356--380, 2018.
		}
	\end{thebibliography}

\end{document}